\documentclass{amsart}
\usepackage{amssymb,epsfig,amsmath,latexsym,amsthm}
\theoremstyle{plain}
\newtheorem{theorem}{Theorem}[section]
\newtheorem{lemma}[theorem]{Lemma}
\newtheorem*{claim}{Claim}
\newtheorem{proposition}[theorem]{Proposition}
\newtheorem{corollary}[theorem]{Corollary}
\newtheorem{question}[theorem]{Question}
\newtheorem{conjecture}[theorem]{Conjecture}
\newtheorem{problem}[theorem]{Problem}
\theoremstyle{definition}
\newtheorem{definition}[theorem]{Definition}
\newtheorem{remark}[theorem]{Remark}

\newtheorem{acknowledgements}[theorem]{Acknowledgements}

\newtheorem{remarks}[theorem]{Remarks}
\newtheorem{correction}[theorem]{Correction}

\newtheorem{notation}[theorem]{Notation}
\newtheorem{example}[theorem]{Example}

\DeclareMathOperator{\length}{length}

\DeclareMathOperator{\Isom}{Isom}

\DeclareMathOperator{\genus}{genus}

\DeclareMathOperator{\textinf}{inf}

\DeclareMathOperator{\textlim}{lim}
\DeclareMathOperator{\Stryker}{Stryker}
\DeclareMathOperator{\st}{st}

\DeclareMathOperator{\sub}{sub}
\DeclareMathOperator{\inte}{int}
\DeclareMathOperator{\cw}{cw}

\newcommand{\finv}{f^{-1}}
 
\newcommand{\hinv}{h^{-1}}
\newcommand{\ginv}{g^{-1}}

\newcommand{\pinv}{p^{-1}}
\newcommand{\piinv}{\pi^{-1}}

\newcommand{\phiinv}{\phi^{-1}}
\newcommand{\psiinv}{\psi^{-1}}
\newcommand{\hatphiinv}{\hat\phi^{-1}}
\newcommand{\BH}{\mathbb H}
\newcommand{\BR}{\mathbb R}

\newcommand{\BN}{\mathbb N}
\newcommand{\BQ}{\mathbb Q}
\newcommand{\BZ}{\mathbb Z}

\newcommand{\mA}{\mathcal{A}}
\newcommand{\mB}{\mathcal{B}}
\newcommand{\mC}{\mathcal{C}}

\newcommand{\mE}{\mathcal{E}}
\newcommand{\mF}{\mathcal{F}}

\newcommand{\mJ}{\mathcal{J}}
\newcommand{\mL}{\mathcal{L}}
\newcommand{\mM}{\mathcal{M}}
\newcommand{\mP}{\mathcal{P}}

\newcommand{\mR}{\mathcal{R}}
\newcommand{\mS}{\mathcal{S}}
\newcommand{\mT}{\mathcal{T}}
\newcommand{\mU}{\mathcal{U}}
\newcommand{\mW}{\mathcal{W}}

\newcommand{\PML}{\mP \mM \mL (S)}
\newcommand{\ML}{\mM \mL (S)}

\newcommand{\LS}{\mL (S)}
\newcommand{\EL}{\mE \mL (S)}

\newcommand{\PMLEL}{\mP\mM\mL\mE\mL(S)}
\newcommand{\MLEL}{\mM\mL\mE\mL(S)}
\newcommand{\FPML}{\mF\PML}
\newcommand{\UPML}{\mU\PML}

\newcommand{\Sk}{S^{k-1}}
\newcommand{\dpts}{d_{PT(S)}}
\newcommand{\Npts}{N_{PT(S)}}
\newcommand{\Npml}{N_{\PML}}
\newcommand{\dpml}{d_{\PML}}

\def\C{\mathcal{C}}

\def\CS{\C (S)}
\def\CSsub{\CS_{\sub}}

\def\UPMLcw{\UPML_{\cw}}

\begin{document}

\title{On the topology of ending lamination space}
\author{David Gabai}\footnote{Partially supported by NSF grants
DMS-0504110 and DMS-1006553.}\address{Department of Mathematics\\Princeton
University\\Princeton, NJ 08544}

\thanks{Version 0.916, May 18, 2011}

\begin{abstract}  We show that if $S$ is a finite type orientable surface of genus $g$ and $p$ punctures where $3g+p\ge 5$, then $\EL$ is $(n-1)$-connected and $(n-1)$-locally connected, where $\dim(\PML)=2n+1=6g+2p-7$.   Furthermore, if $g=0$, then  $\EL$ is homeomorphic to the $p-4$ dimensional Nobeling space.  \end{abstract}

\maketitle

\setcounter{section}{-1}

\section{Introduction}\label{S0}

This paper is about the topology of the space $\EL$ of \emph{ending laminations} on a finite type 
hyperbolic surface, i.e. a complete hyperbolic surface S of genus-g with p punctures.  An ending 
lamination is a geodesic lamination $\mL$ in S  that is minimal and filling, i.e. 
every leaf of $\mL$ is dense in $\mL$ and any simple closed geodesic in S nontrivally intersects 
$\mL$ transversely.  

Since Thurston's seminal work on surface automorphisms in the mid 1970's, laminations in surfaces have played central roles in low dimensional topology, hyperbolic geometry, geometric group theory and the theory of mapping class groups.  From many points of view, the ending laminations are the most interesting laminations.  For example, the stable and unstable laminations of a pseudo Anosov mapping class are ending laminations \cite{Th1} and associated to a degenerate end of a complete  hyperbolic 3-manifold with finitely generated fundamental group is an ending lamination \cite{Th4}, \cite{Bon}.  Also, every ending lamination arises in this manner \cite{BCM}.  

The Hausdorff metric on closed sets induces a metric topology on $\EL$.  Here, two elements $\mL_1$, $\mL_2$ in $\EL$ are \emph{close} if each point in $\mL_1$ is close to a point of $\mL_2$ and vice versa.  In 1988, Thurston \cite{Th2} showed that with this topology $\EL$ is totally disconnected and in 2004 Zhu and Bonahon \cite{ZB} showed that with respect to the Hausdorff metric, $\EL$ has Hausdorff dimension zero.  

It is the \emph{coarse Hausdorff topology}  that makes  $\EL$ important for applications and gives 
$\EL$ a very interesting topological structure.  This is the topology on $\EL$ induced from that of 
$\PML$, the space of projective measured laminations of $ S$.    Let $\FPML$ denote the 
subspace of $\PML$ consisting of those measured laminations whose underlying lamination is an 
ending lamination.  (Fact:  every ending lamination fully supports a measure.)  Then $\EL$ is a 
quotient of $\FPML$ and is topologized accordingly.  Alternatively, a sequence $\mL_1, \mL_2, 
\cdots$ converges to $\mL$ in the coarse Hausdorff topology if after passing to subsequence, the 
sequence converges in the Hausdorff topology to a diagonal extension of $\mL$, i.e. a lamination 
obtained by adding finitely many leaves, \cite{Ha}.  From now on $\EL$ will have the coarse 
Hausdorff topology.

In 1999, Erica Klarreich \cite{K} showed that  $\EL$ is the Gromov boundary of $\CS$, the curve 
complex of S.  (See also \cite{H1}.)  As a consequence of many results in hyperbolic 
geometry e.g. \cite{Ber}, \cite{Th3}, \cite{M}, \cite{BCM}, \cite{Ag}, \cite{CG}; Leininger and 
Schleimer \cite{LS}  showed that the space of doubly degenerate hyperbolic structures on $S
\times \BR$ is  homeomorphic to $\EL\times\EL\setminus \Delta$, where $\Delta$ is the diagonal.   
For other applications see  \cite{RS} and \S 18.

If $S$ is the thrice punctured sphere, then $\EL=\emptyset$.  If $S$ is the 4-punctured sphere or once-punctured torus, then $\EL=\FPML=\BR\setminus\BQ$.  In 2000 Peter Storm conjectured that if $S$ is not one of these exceptional surfaces, then $\EL$ is connected.  Various partial results on the connectivity and local connectivity of $\EL$ were obtained by Leininger, Mj and Schleimer, \cite{LS}, \cite{LMS}.  

Using completely different methods, essentially by bare hands, we showed \cite{G1} that if $S$ is neither the 3 or 4-punctured sphere nor the once-punctured torus, then $\EL$ is path connected, locally path connected, cyclic and has no cut points.   We also asked whether  ending lamination spaces of sufficiently complicated surfaces were $n$-connected.\\

Here are our two main results.

\begin{theorem}   \label{main} Let S be a finite type hyperbolic surface of genus-g and p-punctures.  Then $\EL$ is $(n-1)$-connected and $(n-1)$-locally connected, where $2n+1=\dim(\PML)=6g+2p-7$.\end{theorem}

\begin{theorem}   Let $S$ be a $(4+n)$-punctured sphere.  Then $\EL$ is homeomorphic to the $n$-dimensional Nobeling space.\end{theorem}

\begin{remark}Let $T$ denote the compact surface of genus-$g$ with $p$ open discs removed and $S$ the p-punctured genus-g surface.  It is well known that there is a natural homeomorphism between $\EL$ and $\mE\mL(T)$.  In particular, the topology of $\EL$ is independent of the hyperbolic metric and hence all topological results about $\EL$ are applicable to  $\mE\mL(T)$.   Thus the main results of this paper are purely topological and applicable to compact orientable surfaces.\end{remark}

The $m$-dimensional Nobeling space $\mR^{2m+1}_m$ is the space of points in $\mR^{2m+1}$ with at most $m$ rational coordinates.   In 1930 G. Nobeling \cite{N} showed that the $m$-dimensional Nobeling space is a universal space for $m$-dimensional separable metric spaces, i.e. any $m$-dimensional separable metric space embeds in $\mR^{2m+1}_m$.   This extended the work of his mentor K. Menger \cite{Me}  who in 1926 defined the $m$-dimensional Menger spaces $M^{2m+1}_m$, showed that the Menger curve is universal for 1-dimensional compact metric spaces and suggested that $M_m^{2m+1}$ is universal for $m$-dimensional compact metric spaces.  In 1984 Mladen Bestvina \cite{Be} showed that  any map of a compact $\le m$-dimensional space into $M_m^{2m+1}$ can be approximated by an embedding and in 2006 Nagorko \cite{Na} proved that  any map of a $\le m$-dimensional complete separable metric space into $\mR^{2m+1}_m$ is approximable by a closed embedding. 

A recent result of Ageev \cite{Av}, Levin  \cite{Le} and Nagorko \cite{Na}  gave a positive proof of a major long standing conjecture characterizing the m-dimensional Nobeling space.  Nagorko \cite{HP} recast this result to show that the $m$-dimensional Nobeling space is one that  satisfies a series of topological properties that are discussed in detail in \S7, e.g. the space is $(m-1)$-connected, $(m-1)$-locally connected and satisfies the \emph{locally finite $m$-discs property}.   To prove Theorem \ref{nobeling} we will show that $\EL$ satisfies these conditions for $m=n$.

Using \cite{G1},  Sebastian Hensel and Piotr Przytycki \cite{HP} earlier showed  that if $S$ is either the 5-punctured sphere or twice-punctured torus, then $\EL$ is homeomorphic to the one dimensional Nobeling space. They also boldly conjectured that all ending lamination spaces are homeomorphic to Nobeling spaces.

Independently, in 2005 Bestvina and Bromberg \cite{BB} asked whether ending lamination spaces are Nobeling spaces.    They were motivated by the fact that Menger spaces frequently arise as boundaries of locally compact Gromov hyperbolic spaces, Klarreich's theorem and the fact that the curve complex is not locally finite.  

Theorem 0.2 gives a positive proof of the Hensel - Przytycki conjecture for punctured spheres.   In section \S 19 we offer various conjectures on the topology of ending lamination spaces including three duality conjectures relating the curve complex with ending lamination space.

The methods of this paper are essentially also by bare hands.  There 
are two main difficulties that must be overcome to generalize the methods of \cite{G1} to prove 
Theorem \ref{main}.  First of all it is problematic getting started.  To prove path connectivity, given 
$\mu_0$, $\mu_1\in \EL$ we first chose $\lambda_0$, $\lambda_1\in\PML$ such that  $\phi(\lambda_i)
=\mu_i$ where $\phi$ is the forgetful map.  The connectivity of $\PML$ implies there exists a path $
\mu:[0,1]\to \PML$ such that $\mu(0)=\mu_0$ and $\mu(1)=\mu_1$.  In \cite{G1} we found an 
appropriate sequence of generic such paths, projected them into lamination space,  and took an 
appropriate limit which was a path in $\EL$ between $\mu_0$ and $\mu_1$.  To prove simple connectivity, say for $\EL$ where $S$ is the surface of genus 2, the first step is already problematic, for there is a simple closed curve $\gamma$ in $\EL$ 
whose preimage in $\PML$ does not contain a loop projecting  to $\gamma$.  In the general case, 
the preimage is a Cech-like loop and that turns out to be good enough.   The second issue is that points along a 
generic path in $\PML$ project to laminations that are \emph{almost filling almost minimal}, a 
property that was essential in order to take limits in \cite{G1}.  In the general case, the analogous laminations are not 
close to being either almost filling or almost minimal.    To deal with this we introduce the idea of 
\emph{markers} that is a technical device that enables us to take limits of laminations with the 
desired controlled properties.

This paper basically accomplishes two things.  It shows that if $k\le n$, then 
any generic PL map $f:B^k\to \PML$ can be \emph{$\epsilon$-approximated} by a map $g:B^k\to 
\EL$ and conversely any map $g:B^k\to \EL$ can be \emph{$\epsilon$-approximated} by 
a map $
f:B^k\to \PML$.  Here $\dim(\PML)=2n+1$.  See section \S 13 for the precise statements.

  In \S1 we provide some basic information and facts about ending lamination space, point out an omission in  \S7 \cite{G1} and prove the following.  See \cite{LMS} and \cite{LS} for earlier partial results.
  
\begin{theorem}  If $S$ is a finite type hyperbolic surface that is not the 3 or 4-holed sphere or 1-holed torus, then $\FPML$ is connected.\end{theorem}

In \S2 we show that  if $g:S^k\to \EL$ is continuous, then there exists a continuous map $F:B^{k+1} \to \PMLEL$ that extends  $g$.  Here $\PMLEL$ is the disjoint union of $\PML$ and $\EL$ appropriately topologized. In \S3 we introduce \emph{markers}.  In \S 4 we give more facts relating the topologies of $\EL$ and $\PML$.   In  \S 5 we give a criterion for a sequence $f_1, f_2, \cdots$ of maps of $B^{k+1}$ into $\PMLEL$ that restrict to $g:S^{k}\to \EL$ to \emph{converge} to a continuous map $G:B^{k+1}\to \EL$, extending $g$.  The core technical work of this paper is carried out in \S 6-10.  In \S 11-12 we prove Theorem \ref{main}.  In \S 13 we isolate out our $\PML$ and $\EL$ approximation theorems.  In \S 14 we develop a theory of good cellulation sequences of $\PML$, which may be of independent interest. In \S 15 we give various upper and lower estimates of $\dim(\EL)$ and prove that $\dim(S_{0,n+p})=n$ and that $\pi_n(S_{0, n+p})\neq 0$.  In \S 16 we state Nagorko's recharacterization of  Nobeling spaces.  In \S 17 we prove that $S_{0, n+p}$ satisfies the locally finite $n$-discs property.  In \S 18 various applications are given.  In \S 19 we offer some problems and  conjectures.

\begin{acknowledgements}   The author learned about ending lamination spaces from hearing many beautiful lectures by Saul Schleimer and Chris Leininger.  The author is grateful to Andrzej Nagorko for informative communications about Nobeling spaces and to Piotr Przytycki for many discussions and interest in this work.   He thanks Mladen Bestvina, Alexander Dranishnikov and Ross Geoghegan for very helpful conversations and communications.  \end{acknowledgements}

\section{Basic definitions and facts}

 In what follows, $S$ or $S_{g,p}$ will denote a complete hyperbolic surface of genus 
$g$ and $p$ punctures.    We will assume that the reader is 
familiar with the basics of Thurston's theory of curves and laminations on surfaces, e.g. $\mL(S)$ the space of geodesic laminations with topology induced from the Hausdorff 
topology on closed sets, $\ML$ the space of measured geodesic laminations endowed with the 
weak* topology, $\PML$ projective measured lamination space, as well as the standard definitions and properties of 
of train tracks.  For example, see \cite{PH}, \cite{Ha}, \cite{M}, \cite{Th1}, or \cite{FLP}.   
All laminations in this paper will be compactly supported.  See   \cite{G1} for   various ways to view and measure distance between laminations as well as for standard notation.  For example, 
$d_{PT(S)}(x,y)$ denotes the distance between the geodesic laminations $x$ and $y$ as 
measured in the projective unit tangent bundle.  Unless said otherwise, distance between 
elements of $\mL$ are  computed via the Hausdorff topology on closed sets in 
$PT(S)$.  Sections \S1 and \S2 (through Remark 2.5) of \cite{G1} are also needed.  Among 
other things, important aspects of the PL structure of $\PML$ and $\ML$ are described there.

\begin{notation}  We denote by $p:\ML\setminus 0 \to \PML$, the canonical projection and $
\phi:\PML\to \LS$ and $\hat\phi:\ML\to \LS$, the forgetful maps.  If $\tau$ is a train track, then 
$V(\tau)$ will denote the cell of measures supported on $\tau$ and $P(\tau)$ the polyhedron 
$p(V(\tau)\setminus 0)$.  \end{notation}

\begin{definition}  Let $\EL$ denote the set of ending laminations on S, i.e. the set of geodesic laminations that are \emph{filling} and \emph{minimal}.  A lamination $\mL\in \LS$ is \emph{minimal} if every leaf is dense and \emph{filling} if the metric closure (with respect to the induced path metric) of $S\setminus \mL$ supports no simple closed geodesic.  \end{definition}

The following is well known.

\begin{lemma}  \label{filling} If $x\in \PML$, then $\phi(x)$ is the disjoint union of minimal laminations.  In particular $\phi(x)\in \EL$ if and only if $\phi(x)$ is filling.  \end{lemma}

\begin{proof}  Since the measure on $x$ has full support, no non compact leaf $L$ is proper, i.e. non compact leaves limit on themselves.  It follows that $\phi(x)$ decomposes into a disjoint union of minimal laminations.  If $\phi(x)$ is filling, then there is only one such component.\end{proof}

\begin{notation}  Let $\FPML$ denote the subspace of $\PML$ consisting of filling laminations  and $\UPML$ denote the subspace of unfilling laminations.  Thus $\PML$ is the disjoint union of $\FPML$ and $\UPML$.\end{notation}

\begin{definition}  Topologize $\EL$ by giving it the quotient topology induced from the surjective map $\phi:\FPML\to \EL$  where $\FPML$ has the subspace topology induced from $\PML$.   After \cite{H1} we call this the \emph{coarse Hausdorff topology}.  Hamenstadt observed that this topology is a slight coarsening of the Hausdorff topology on $\EL$; a sequence $\mL_1, \mL_2\cdots$  in $\EL$  limits to $\mL\in \EL$ if and only if any convergent subsequence in the Hausdorff topology converges to a diagonal extension $\mL'$ of $\mL$.  A \emph{diagonal extension} of $\mL$ is a lamination obtained by adding finitely many leaves.\end{definition} 

\begin{remark} \label{separable and complete}  It is well known that $\EL$ is separable and complete.  Separability follows from the fact that $\PML$ is a sphere and $\FPML$ is dense in $\PML$.  (E.g. $\FPML$ is the complement of countably many codimension-1 PL-cells in $\PML$.)  Masur and Minsky \cite{MM} showed that the curve complex $\CS$ is Gromov-hyperbolic and  Klarreich \cite {K} (see also \cite{H1}) showed that the Gromov boundary of $\CS$ is 
homeomorphic to $\EL$ with the coarse Hausdorff topology.  Being the boundary of a Gromov 
hyperbolic space, $\EL$ is metrizable.  Bonk and Schramm showed that with 
appropriate constants in the Gromov product, the induced Gromov metric is complete \cite{BS}.  See also \cite{HP}. \end{remark}

The following are characterizations of continuous maps in $\EL$ analogous to Lemmas 1.13-1.15 of \cite{G1}.  Here $X$ is a metric space.  

\begin{lemma} \label{continuity1}A function $f:X\to \EL$ is continuous if 
and only if for each $t\in X$ and each sequence $\{t_i\}$ converging to $t, f(t)$ 
is the coarse Hausdorff limit of the sequence $f(t_1), f(t_2), 
\cdots.$\qed\end{lemma}

\begin{lemma}\label{continuity2}  A function $f:X\to \EL$ is 
continuous if and only if for each $\epsilon>0$ and $t\in X$ there exists a 
$\delta >0$ such that $d_X(s,t)<\delta$ implies that the maximal angle of intersection 
between leaves of $f(t)$ and leaves of $f(s)$ is $<\epsilon$. \qed\end{lemma}

\begin{lemma}\label{continuity3} A function $f:X\to \EL$ is continuous 
if and only if for each $\epsilon>0$ and $t\in X$ there exists a $\delta >0$ such 
that $d_X(s,t)<\delta$ implies that  $\dpts(f'(t),(f'(s))<\epsilon$, where $f'(s)$ (resp $f'(t)$) is any diagonal extension of $f(s)$ (resp. $f(t)$). 
\qed\end{lemma}

The forgetful map $\phi:\PML\to \LS$ is 
discontinuous, for any simple closed curve viewed as a point in $\PML$ is the limit of 
filling laminations and any Hausdorff limit of a sequence of filling laminations 
is filling.  

\begin{definition}  Let $X_1, X_2, \cdots$ be a sequence of subsets of the 
topological space $Y$.  We say that the subsets $\{X_i\}$ \emph{super converges} to $X$ 
if for each $x\in X$, there exists $x_i\in X_i$ so that lim$_{i\to \infty} x_i= x$.  
In this case we say $X $ is a \emph{sublimit} of $\{X_i\}$.\end{definition}

We will repeatedly use the following result that first appears in \cite{Th4}. See Proposition 3.2 \cite{G1} for a proof.

\begin{proposition} \label{super convergence}   If  the projective measured laminations $\lambda_1, \lambda_2, \cdots$ 
 converge to  $\lambda \in \PML$, then $\phi(\lambda_1), 
\phi(\lambda_2), \cdots$ super converges to $\phi(\lambda)$ as subsets of $PT(S)$.  \end{proposition} \vskip8pt

The following consequence of super convergence, Proposition \ref{super convergence} is repeatedly used in this paper and in \cite{G1}.

\begin{lemma}  \label{convergent sequence}  If $z_1, z_2, \cdots$ is a convergent sequence in $\EL$ limiting to $z_\infty$ and $x_1, x_2, \cdots$ is a sequence in $\PML$ such that for all $i$, $\phi(x_i)=z_i$, then any convergent subsequence of the $x_i$'s converges to a point of $\phiinv(z_\infty)$.\qed\end{lemma}

\begin{proof}  After passing to subsequence it suffices to consider the case that $x_1, x_2, \cdots$ coverges to $x_\infty$ in $\PML$ and that $z_1, z_2, \cdots$ converges to $\mL\in \LS$ with respect to the Hausdorff topology.  By super convergence, Proposition \ref{super convergence} $ \phi(x_\infty)  $ is a sublamination of $\mL$  and by definition of coarse Hausdorff topology $\mL$ is a diagonal extension of $z_\infty$.    Since $z_\infty$ is minimal and each leaf of $\mL\setminus z_\infty$ is non compact and proper it follows that $\phi(x_\infty)=z_\infty$.  \end{proof}

\begin{corollary}  \label{preimage compact} If $L\subset \EL$ is compact, then $\phiinv(L)\subset \PML$ is compact.    \qed\end{corollary}

\begin{corollary}  \label{closed map} $\phi:\FPML\to \EL$ is a closed map.\qed\end{corollary}

\begin{lemma} \label{pml epsilon closeness}  If $\mu\in \EL$, $x_1, x_2, \cdots\to x$ is a convergent sequence in $\PML$ and $\textlim_{i\to \infty} \dpts(\phi(x_i), \mu')=0$ for some diagonal extension $\mu'$ of $\mu$, then $x\in \phiinv(\mu)$.\end{lemma}

\begin{proof}  After passing to subsequence and super convergence $\phi(x_1), \phi(x_2), \cdots $ converges to $\mL\in \LS$ with respect to the Hausdorff topology where $\phi(x)\subset \mL$.  If $\phi(x)\neq\mu$, then $\dpts(\mL,\mu')>0$ and hence $\dpts(\phi(x_i), \mu')$ is uniformly bounded away from $0$ for every diagonal extension $\mu'$ of $\mu$. \end{proof}

\begin{lemma} \label{train track closeness}  Let  $\tau$ be a train track and $\mu\in \EL$.  Either $\tau$ carries $\mu$ or $\textinf \{\dpts(\phi(t), \mu)|t\in P(\tau)\}>0$.\end{lemma}

\begin{proof}  If $\textinf \{\dpts(\phi(t), \mu)|t\in P(\tau)\}=0$, then by compactness of $P(\tau)$ and the previous lemma, there exists $x\in P(\tau)$ such that $\phi(x)=\mu$ and hence $\tau$ carries $\mu$.\end{proof}

The following is well known, e.g. it can be deduced from Proposition 1.9 \cite{G1}.

\begin{lemma}  \label{preimage simplex}  If $z\in \EL$, then $\phiinv(z)=\sigma_z$ is a compact convex cell, i.e. if $\tau$ is any train track that carries $z$,  then $\sigma_z=p(V)$, where $V\subset V(\tau)$ is the open cone on a compact convex cell in $V(\tau)$.\qed\end{lemma}

\begin{theorem}  \label{fpml connectivity} If $S$ is a finite type hyperbolic surface that is not the 3 or 4-holed sphere or 1-holed torus, then $\FPML$ is connected.\end{theorem}

\begin{proof}  If $\FPML$ is disconnected, then it is the disjoint union of non empty closed sets $A$ and $B$.  By the previous lemma, if $\mL\in \EL$, then $\phiinv(\mL)$ is connected.  It follows that $\phi(A)\cap\phi(B)=\emptyset$ and hence by Lemma \ref{closed map}, $\EL$ is the disjoint union of the non empty closed sets $\phi(A)$, $\phi(B)$ and hence $\EL$ is disconnected.  This contradicts \cite {G1}.\end{proof}

\begin{remark}  In the exceptional cases, $\mF\mP\mM\mL(S_{0,3})=\emptyset$, $\mF\mP\mM\mL(S_{0,4})=\mF\mP\mM\mL(S_{1,1})=\BR\setminus\BQ.$\end{remark}

\begin{definition} \label{curve complex}  The curve complex $ \CS$ introduced by Harvey \cite{Ha} is the simplicial complex with vertices the set of simple closed geodesics and $(v_0, \cdots, v_p)$ span a simplex if the $v_i$'s are pairwise disjoint.  

There is a natural injective continuous map $\hat I:\CS\to \ML$.  If $v$ is a vertex, then $I(v)$ is the measured lamination supported on $v$ with transverse measure $1/\length(v)$.  Extend $\hat I$ linearly on simplices.  Define $I:\CS\to \PML$ by $I=p\circ \hat I$.  \end{definition}

\begin{remark} \label{sub is coarser} The map $I$ is not a homeomorphism onto its image.  If $\CSsub$ denotes the topology on $\CS$ obtained by pulling back the subspace topology on $I(\CS)$, then $\CSsub$ is coarser than $\CS$.  Indeed, if $C$ is a vertex, then there exists a sequence $C_0, C_1,\cdots$ of vertices that converges in $\CSsub$ to $C$, but does not have any limit points in $\CS$.  See \S 19 for some conjectures relating the topology of $\CS$ and $\CSsub$.  \end{remark}

\begin{correction}  In \S 7 of \cite{G1} the author states without proof Corollary 7.4 which asserts that if $S$ is not one of the three exceptional surfaces, then $\EL$ has no cut points.  It is   not a corollary of the statement of Theorem 7.1 because as  the figure-8 shows, cyclic does not imply no cut points.  It is not difficult to prove Corollary 7.4 by extending the proof of Theorem 7.1  to show that given $x\neq y\in \EL$, there exists a simple closed curve in $\EL$ passing through both $x$ and $y$.   Alternatively, no cut points immediately follows from the locally finite 1-discs property, Proposition  \ref{m-discs} and Remark \ref{s12}. \end{correction}

\begin{notation}  If $X$ is a space, then $|X|$ denotes the number of components of $X$.  If $X$ and $Y$ are sets, then $X\setminus Y$ is $X$ with the points of $Y$ deleted, but if $X,Y\in \LS$, then $X\setminus Y$ denotes the union of leaves of $X$ that are not in $Y$.
Thus by Lemma \ref{filling} if $X=\phi(x)$, $Y=\phi(y)$ with $x,y\in \PML$, then $X\setminus Y$ is the union of those minimal sublaminations of $X$ which are not sublaminations of $Y$.  \end{notation}


\section{Extending maps of spheres into $\EL$ to maps of balls into $\PMLEL$}

\begin{definition} \label{ml cup el} Let $\PMLEL$ (resp. $\MLEL$) denote the disjoint union $\PML\cup\EL$ (resp. $\ML\cup\EL$).  Define a topology on $\PMLEL$ (resp. $\MLEL$) as 
follows.  A basis consists of all  sets of the form $U\cup V$ where $\phiinv(V)\subset U$ (resp. 
$\hatphiinv(V)\subset U$) where $U$ is open in $\PML$ (resp. $\ML$) and $V$ (possibly $
\emptyset$) is open in $ \EL$.  We will call this the \emph{$PMLEL$ topology} (resp. \emph{$MLEL$ 
topology}).  \end{definition} 

\begin{lemma} \label{pmlel topology} The PMLEL (resp. MLEL) topology has the following properties.

\smallskip i) $\PMLEL$ (resp. $\ML\cup\EL$) is non Hausdorff.  In fact $x$ and $y$ cannot be separated if and only if $x\in \EL$ and $y\in \phiinv(x)$ (resp. $y\in \hatphiinv(x))$, or vice versa.  \smallskip

ii)  $\PML$ (resp. $\ML$) is an open subspace.\smallskip

iii) $\EL$ is a closed subspace.\smallskip

iv) If $U\subset \PML$  is a neighborhood of $\phiinv(x)$ where $x\in \EL$, then there exists an open set $V\subset \EL$, such that $x\in V$ and $\phiinv(V)\subset U$, i.e. $U\cup V$ is open in $\PMLEL$.\smallskip

v)  A sequence $x_1, x_2, \cdots$ in $\PML$ converges to $x\in \EL$ if and only if every limit 
point of the sequence lies in $\phiinv(x)$.  A sequence $x_1, x_2, \cdots$ in $\ML$ bounded 
away from both $0$ and $\infty$ converges to $x\in \EL$ if and only if every limit point of the 
sequence lies in $\hatphiinv(x)$.\end{lemma}

\begin{proof}  Parts i)-iii) and v) are immediate.  Part iv) follows from Lemma \ref{convergent sequence}.   \end{proof}

  \begin{definition}  Let $V$ be the underlying space of a finite simplicial complex.  A \emph{generic PL map} $f:V\to \PML$ is a PL map transverse to each $B_{a_1}\cap\cdots B_{a_r}\cap \delta B_{b_1}\cap\cdots\cap \delta B_{b_s}$, where $a_1, \cdots, a_r, b_1, \cdots, b_s$ are simple closed geodesics.  (Notation as in \cite{G1}.) More generally $f:V\to \PMLEL$ is called  a \emph{generic PL map} if $\finv(\EL)=W$ is a subcomplex of $V$ and $f|(V\setminus W)$ is a generic PL map.\end{definition}
  
  \begin{lemma} \label{missing} Let $L=I(L_0)$ where $L_0$ is a finite $q$-subcomplex of $\CS$.    If $f:V\to \PMLEL$ is a generic PL-map where $\dim(V)=p$ and   $p+q\le 2n$, then $f(V)\cap L=\emptyset$.  If either $p+q\le 2n-1$ or $p \le n$, then for every simple closed geodesic $C$, $f(V)\cap Z=\emptyset$, where $Z=\phiinv(C)*(\partial B_C\cap L)$.  \end{lemma}

\begin{proof}  By genericity and the dimension hypothesis, the first conclusion is immediate.  For the second, $L\cap \partial B_C$ is at most $\min(q, n-1) $ dimensional, hence any simplex in the cone $(L\cap \partial B_C)*C$ is at most $\min(q+1, n)$ dimensional.  Thus the result again follows by genericity.\end{proof}
  
  \begin{remark}   In this paper, $V$ will be either a $B^k$ or $S^k\times I$ and $X=\partial B^k$ or $S^k\times 1$.\end{remark}

\begin{notation}  Fix once and for all a map  $\psi:\EL\to \PML$ such that $\phi \circ \psi = 
id_{\EL}$. If $i:\PML\to \ML$ is the map sending a projective measured geodesic lamination to the 
corresponding measured geodesic lamination of length 1, then define $\hat\psi:\EL\to \ML$ by $\hat\psi=i
\circ\psi$.  Let $p:\ML\setminus 0 \to \PML$ denote the standard projection map and let $\hat 
\phi=\phi\circ p$.    \end{notation}

While $\psi$ is very discontinuous, it is continuous enough to carry out the following which is the main result of this section.

\begin{proposition} \label{extension}  Let $g:S^{k-1}\to \EL$ be continuous where $k\le \dim(\PML)=2n+1$.  Then there exists a generic PL map $F:B^{k}\to \PMLEL$ (resp. $\hat F:B^k\to \MLEL$) 
such that $F|S^{k-1}=g$ (resp. $\hat F|S^{k-1}=g$) and $F(\inte(B^{k}))\subset \PML$ (resp. $\hat F(\inte(B^k))\subset\ML)$.\end{proposition} 

\noindent\emph{Idea of Proof for $\PML$}:  It suffices to first find  a continuous extension and then 
perturb  to a generic PL map.  To obtain a continuous extension $F$, 
consider a sequence $K_1, K_2, \cdots$ of finer and finer triangulations of $\Sk$.  For each $i$, 
consider the map $f_i:\Sk\to \PML$ defined as follows.  If $\kappa$ is a simplex of $K_i$ with 
vertices $v_0, \cdots, v_m$, then define $f_i(v_j)=\psi(g(v_j))$ and extend $f_i$ linearly on $
\kappa$.  Extend $f_1$ to a map of $ B^k$ into $\PML$ and extend $f_i$ and $f_{i+1}$ to a map 
of $F_i:\Sk\times [i, i+1]$ into $\PML$.  Concatenating these maps and taking a limit yields the 
desired continuous map  $F:B^k\cup\Sk\times [1,\infty]\to \PMLEL$  where $\partial B^k$ is 
identified with $\Sk\times 1, G|\Sk\times [i, i+1]=F_i$ and $H|S^n\times \infty = g$.  \vskip8pt

\begin{remark}  The key technical issue is making precise the phrase ``extend $f_i$ linearly on $\sigma$".  \end{remark}

Before we prove the Proposition we establish some notation and then prove a series of lemmas.

\begin{notation}  Let $\Delta$ be the cellulation on $\PML$ whose cells are the various $P(\tau_i)$'s, where the $ \tau_i$'s are the standard train tracks to some fixed parametrized pants decomposition of $S$.   If $\sigma=P(\tau_i)$ is a cell of $\Delta$ and $Y\subset \sigma$, then define the \emph{convex hull} of $Y$ to be $p(C(\pinv(Y)))$, where $C(Z)$ is the convex hull of $Z$ in $V(\tau_i)$.  If $x\in \EL$, then let $\tau_x$ denote the unique standard train track that fully carries $x$ and $\sigma_x$ the cell $P(\tau_x)$.   

If $\mL\in \mL(S)$, then $\mL^\prime$ will denote a \emph{diagonal extension of $\mL$},   i.e. a lamination obtained by adding finitely many non compact leaves.\end{notation}

\begin{remark}   It follows from Lemma \ref{preimage simplex} that $\phiinv(x)\subset \inte(\sigma_x)$ and is closed and convex.   \end{remark}

\begin{lemma}  \label{psi continuity} Let $x_1, x_2, \cdots \to x$ be a convergent 
sequence in $\EL$.  If $y\in \PML$ is a limit point of $\{\psi(x_i)\}$, then $\phi(y)=x$.   \end{lemma}

\begin{proof}  This follows from Lemma \ref{convergent sequence}.  \end{proof}

\begin{definition} Fix  $\epsilon_1< 1/1000(\min(\{d(\sigma, \sigma')|\sigma, 
\sigma'$ disjoint cells of $\Delta\}$.  For each  cell $ \sigma$ of $\Delta$ define a retraction  $r_
\sigma:N(\sigma,\epsilon_1)\to \sigma$.  For every $\delta<\epsilon_1$ define a discontinuous map $\pi_\delta:\PML\to \PML$ as follows.       Informally, $\pi_\delta$ retracts a closed neighborhood 
of $\delta^0$ to $\delta^0$, then after deleting this neighborhood retracts a closed neighborhood of $
\Delta^1$ to $ \Delta^1$, then after deleting this neighborhood retracts a closed neighborhood of $
\Delta^2$ to $\Delta^2$ and so on.  As $\delta \to 0$, the neighborhoods are required to get smaller.  
More formally,    let $\delta(0)=\delta$.  If $\sigma$ is a 0-cell of $\Delta$, then define $\pi_\delta|N(\sigma,
\delta(0))=r_v|(N(\sigma,\delta(0))$.  Now choose  $\delta(1)\le \delta$ such that if $\sigma, \sigma'$ 
are distinct 1-cells, then $N(\sigma, \delta(1))\cap N(\sigma', \delta(1)) \subset N(\Delta^0, \delta)$.  If $
\sigma$ is a 1-cell, then define $\pi_\delta|N(\sigma, \delta(1))\setminus N(\Delta^{0}, \delta(0))=r_
\sigma|N(\sigma, \delta(1))\setminus N(\Delta^{0}, \delta(0))$.  Having defined $\pi_\delta$ on 
$N(\Delta^0, \delta(0))\cup\cdots\cup N(\Delta^k, \delta(k))$, extend $\pi_\delta$ in a similar way over $N(\Delta^{k+1}, 
\delta(k+1))$ using a sufficiently small $\delta(k+1)$.  We require 
that if $\delta<\delta'$, then for all $i, \delta(i)<\delta'(i)$.  \end{definition}

\begin{remark}  \label{pi delta face}  Let $x\in \sigma$ a cell of $\Delta$.  If $\delta_1<\delta_2
\le \epsilon_1$ and $\sigma_1, \sigma_2$ are the lowest dimensional cells of $\Delta$ respectively containing $\pi_{\delta_1}(x), \pi_{\delta_2}(x)$, then $\sigma_2$ is  a face of $\sigma_1$ which is a face of $\sigma$.  \end{remark} 

\begin{lemma}\label{local convexity}  Let  $f_i:B^k\to \EL$ be a sequence of maps such that $f_i(B^k)\to x\in \EL$ in the Hausdorff topology on closed sets in $\EL$. \smallskip 

i)  Then for i sufficiently large $\psi(f_i(B^k))\subset st(\sigma_x)$, the open star of $\sigma_x$ in $\Delta$. \smallskip

ii)  Given $\epsilon>0$ and  $\delta<\epsilon_1$, then for $i$ sufficiently large, if $t\in C_i$, where $C_i$ denotes the convex hull of $r_{\sigma_x}( \psi(f_i(B^k)))$ in $\sigma_x$,  then $x\in d_{PT(S)}(\phi(t), \epsilon)$. \end{lemma}

\begin{proof}  If i) is false, then after passing to subsequence, for each $i$ there exists $a_i \in B^k$ such that $\psi(f_i(a_i))\notin \st(\sigma_x)$.  This contradicts Lemma \ref{psi continuity} which implies that any limit point of $\psi(f_i(a_i))$ lies in $\phiinv(x)\subset \sigma_x$.  

This argument actually shows that if $U$ is any neighborhood of $\phiinv(x)$, then $\psi(f_i(B^k))\subset U$ for $i$ sufficiently large  and hence $r_{\sigma_x}(\psi(f_i(B^k)))\subset U\cap \sigma_x$ for $ i $sufficiently large.  Since $\phiinv(x)$ is convex it follows that $C_i\subset U$ for $i$\ \  sufficiently large.  By super convergence  there exists a neighborhood $V$ of $\phiinv(x)$ such that if $y\in V$, then $x\in d_{PT(S)}(\phi(y), \epsilon)$.  Now choose $U$ so that the $r_{\sigma_x}(U)\subset V$.\end{proof}



\begin{lemma}  \label{pi delta} If $\epsilon > 0$ and $x_1, x_2, \cdots \to x \in \EL$, then there exists $\delta>0$ such that for $i$ sufficiently large $x\subset N_{PT(S)}(\phi(\pi_\delta(\psi(x_i))), \epsilon)$.\end{lemma}

\begin{proof}  Let $U $ be a neighborhood of $\phiinv(x)$ such that $\lambda\in U$ implies $x\subset N_{PT(S)}(\phi(\lambda),\epsilon)$.  Next choose $\delta>0$ so that $\pi_\delta(N(\phiinv(x),\delta))\subset U$.  Now apply Lemma \ref{psi continuity} to show that for $i$ sufficiently large the conclusion holds.\end{proof}

\begin{lemma}  \label{pi convexity} Let $\epsilon >0, g_i:B^k\to \EL$ and $g_i(B^k)\to x$.  There 
exists $N\in \BN, \delta>0$,  such that if $\delta_1, \delta_2\le\delta,  i\ge N$ and $\sigma$ a 
simplex of $\Delta$, then  if $C_i $ is the convex hull of $\sigma\cap( \pi_{\delta_1}(\psi( 
g_i(B^k)))\cup  \pi_{\delta_2}(\psi (g_i(B^k)))\cup (\psi(g_i(B^k))))$ and $ t\in C_i$, then $x\subset 
N_{PT(S)}(\phi(t),\epsilon)$.\end{lemma}

\begin{proof}  Let $V$ be a neighborhood of $\phiinv(x)$ such that if $y\in V$, then $x\in N_{PT(S)}(\phi(y), \epsilon)$.  Let $V_1\subset V$ be a neighborhood of $\phiinv(x)$ such that for each cell $\sigma\subset \Delta$ containing $\sigma_x$ as a face,  the convex hull of $V_1\cap \sigma\subset V$.  Also assume that $\bar V_1\cap \kappa=\emptyset$ if $\kappa$ is a cell of $\Delta$ disjoint from $\sigma_x$.  Choose $\delta<\epsilon_1$ so that $d(\kappa, V_1)>2\delta$ for all cells $\kappa$ disjoint from $\sigma_x$. Choose $N$ so that $i\ge N$ and $\delta'\le \delta$, then $\pi_{\delta^\prime}(g_i(B^k))\subset V_1$ and $g_i(B^k)\subset V_1$.  Therefore, if $\delta_1, \delta_2<\delta$ and $i\ge N$ then $C_i$ lies in $V$.\end{proof}

We now address how to approximate continuous maps into $\EL$ by PL maps into $\PML$.

\begin{definition}  Let $\sigma$ be a cell of $\Delta$.  Let $\kappa$ be a 
$p$-simplex and $H^0:\kappa^0\to\sigma$, where $\kappa^0$ are the 
vertices of $\kappa$.  Define the \emph{induced map}  $H:\kappa\to \sigma$, 
such that $H|\kappa^0=H^0$ as follows.  Let $\hat H$ be the linear map of $
\kappa$ into $\hat \sigma = \pinv(\sigma)$ such that $\hat H|\kappa^0=i\circ 
H^0$.  Then define $H = p\circ \hat H$.  In a similar manner, if $K$ is a simplicial 
complex and $h|K^0\to \PML$ is such that for each simplex $\kappa, 
h(\kappa^0)\subset \sigma$, for some cell $\sigma$ of $\Delta$, then $h$ 
extends to a map $H:K\to \PML$  also called the \emph{induced map}.  Since 
the linear structure on a face of a cell of $\ML$ is the restriction  of the linear 
structure of the cell, $H$ is well defined.  \end{definition}

\begin{lemma} \label{pml approximation} Let $g:K_1\to \EL$ be continuous 
where $K_1$ is a finite simplicial complex.  Let $K_1, K_2, \cdots$ be  such that 
mesh$(K_i)\to 0$ and each $K_{i+1}$ is a subdivision of $K_i$.  For every $
\delta<\epsilon_1$ there exists an $i(\delta)\in \BN$, monotonically increasing as 
$\delta\to 0$, such that if $\delta'\le \delta,   i= i(\delta)$ and $\kappa$ is a 
simplex of $K_i$,  then $\pi_{\delta^\prime}(\psi(g(\kappa)))\cup \pi_
\delta(\psi(g(\kappa)))$ is contained in a cell of $\Delta$.  

Given $\epsilon>0$ there exist $\delta(\epsilon)>0$ and $N(\epsilon)\in \BN$ 
such that if $i= N(\epsilon)$,   
$\kappa$ is a simplex of $K_i, \sigma$ a simplex of $\Delta, \delta_1, \delta_2 
\le \delta(\epsilon)$ and 
$C$ is the convex hull of $\sigma\cap(\pi_{\delta_1}(\psi(g(\kappa)))\cup 
\pi_{\delta_2}
(\psi(g(\kappa)))\cup \psi(g(\kappa)))$, then given $z_1\in \kappa, 
z_2\in C_i $ we have $d_{PT(S)}(g(z_1), \phi(z_2))<\epsilon$.

Fix $\epsilon>0$.  If $\delta$ is sufficiently small, $i>i(\delta)$ and  $H_i:K_i\to \PML$ is the induced  map arising 
from $\pi_\delta\circ\psi\circ g|K_i^0$, then for each $z\in K_1, d_{PT(S)}(\phi(H_i(z)),g(z))<
\epsilon$. \end{lemma}

\begin{proof}  Fix $0<\delta<\epsilon_1$.  For each $x\in \EL$, there exists a neighborhood $V_x
$ of $\phiinv(x)$ such that $\pi_\delta(V_x)\subset \sigma_x$.  By Lemma \ref{pmlel topology} iv) 
there exists a neighborhood $U_x$  of $x$ such that $\psi(U_x)\subset V_x$.  By compactness 
there exist $U_{x_1}, \cdots, U_{x_n}$ that cover $g(K)$.  There exists $i(\delta)>0$ such that if $i\ge i(\delta)$ and $\kappa $ 
is a simplex of $K_i$, then $g(\kappa)\subset U_{x_j}$ for some $j$ and so $\pi_\delta(\psi(g(\kappa)))
\subset \sigma_{x_j}$.   Since each $\phiinv(x)$ lies in the interior of a cell, it follows that $\sigma_{x_j}
$ is the lowest dimensional cell of $\Delta$ that contains any point of $\pi_
\delta(\psi(g(\kappa))$.   Now let $\delta'<\delta$.   As above, $\pi_{\delta^\prime}
(\psi(g(\kappa)))\subset \sigma_{y} $ for some $y\in g(K_1)$ where $\sigma_y$ 
is a minimal dimensional cell.  It follows from Remark \ref{pi delta face} that $ 
\sigma_{x_j}$ is a face of $\sigma_y$.

The proof of the second  conclusion follows from that of the first and the proof of Lemma \ref{pi convexity}.

If the third conclusion of the lemma is false, then after passing to subsequence there exist $
(\kappa_1, \omega_{1_j}, \delta_1, t_1), (\kappa_2, \omega_{2_j}, \delta_2, t_2), \cdots$ such that 
for all $i, \kappa_{i+1}$ is a codimension-0 subsimplex of $ \kappa_i$ which is a simplex of $K_i, 
\omega_{i_j}\subset \kappa_i$ is a simplex of $K_{i_j}$ some $i_j\ge i,  \delta_i\to 0$ and for 
some $t_i\in \omega_{i_j}, d_{PT(S)}(\phi(H_{i_j}(t_i)), g(t_i))>\epsilon$, where $H_{i_j}$  is the 
induced map corresponding to $\delta_i$ and $K_{i_j}$.  Also $1_j<2_j<\cdots$.  If $t=\cap_{i=1}^
\infty\kappa_i$ and  $B$ is homeomorphic to $\kappa_1$, then there exists maps $g_i:B\to \EL$ 
such that $g_i(B)=g(\kappa_i)$ and lim$_{i\to \infty} g_i(B)=g(t)$.  Let $\sigma_i$ be a cell of $
\Delta$ that contains $H_{i_j}(t_i)$.  Since $H_{i_j}(t_i)$ lies in the convex hull of $\pi_{\delta_i}
(\psi(g_i(B)))\cap \sigma_i$ it follows by Lemma \ref{pi convexity} that for $i$ sufficiently large 
$g(t)\subset N_{PT(S)}(\phi(H_{i_j}(t_i),\epsilon/2)$.  Convergence in the coarse Hausdorff topology 
implies that for $i$ sufficiently large if $z\in\kappa_i$ then $g(t)\subset N_{PT(S)}(g(z), \epsilon/2)$.  
Taking $z=t_i$ we conclude that  $d_{PT(S)}(\phi(H_{i_j}(t_i), g(t_i))<\epsilon$, a contradiction.  
\end{proof}

\noindent\emph{Proof of Proposition \ref{extension}}:  Let $K_1, K_2, \cdots$ be subdivisions of $\Sk $
such that mesh$(K_i)\to 0$ and each $K_{i+1}$ is a subdivision of $K_i$.   Pick $\delta_j<\delta(1/j)$ and $n_j>\max \{i(\delta_j), N(1/j)\}$ such that for $i\ge n_j $ and $\delta\le \delta_j$ the induced map $H_k:K_i\to \PML$ arising from $\pi_\delta\circ\psi\circ g|K_i^0$ satisfies the third conclusion of Lemma \ref{pml approximation}.  Assume that $n_1<n_2<\cdots$.  Replace the original $\{K_i\}$ sequence by the subsequence $\{K_{n_i}\}$.  With this new sequence, let $f_j:K_j\to \PML$ be the induced map arising from  $\pi_{\delta_j}\circ\psi\circ g|K_i^0$.

Apply Lemma \ref{pml 
approximation}  to find, after passing to a subsequence of the 
$K_i$'s,  a sequence of  maps $f_j:K_j\to \PML$ so that each $f_j$ satisfies the conclusions of 
those results for $\epsilon=1/j$.  Let $\delta_j $ denote the $\delta$ used to define $f_j$.  Note that with this new sequence of $K_i$'s, each $K_j$ satisfies the first two conclusions of Lemma \ref{pml approximation} using $\delta_j$.

Define a triangulation  $\mT$  on $ \Sk\times [0,\infty)$ by first letting  $\mT|\Sk\times j=K_j$ and 
then extending in a standard way to each $\Sk\times [j,j+1]$ so that if $\zeta$ is a simplex of $
\mT|\Sk\times [j,j+1]$, then $\zeta^0\subset (\kappa^0\times j)\cup (\kappa_1^0\times (j+1)) $ 
where $\kappa$ is a simplex of $K_j$ and $\kappa_1\subset \kappa$ is a simplex of $K_{j+1}$.  
By Lemma \ref{pml approximation}, $\pi_{\delta_j}(\psi(g(\kappa^0)))\cup \pi_{\delta_{j+1}}
(\psi(g(\kappa_1^0)))$ lie in the same cell of $\Delta$ so the induced maps on $ \mT|\Sk\times\{j, j
+1\}$ extend to one called $f_{j, j+1}$ on $\mT|\Sk\times [j, j+1]$. 

Since $k\le\dim\PML$, $f_1$ extends to a  map $f_1'$ of $B^k$ into $\PML$.  Define $F:B^k
\cup \Sk\times [1, \infty]$ to $\PMLEL$ so that $F|B^{k}=f_1', F|\Sk\times [i,i+1]=f_{i, i+1}$ and 
$F|\Sk\times \infty=g$.  It remains to show that $F$ is continuous at each $(z,\infty) \in \Sk\times 
\infty$.  Let $(z_1,t_1), (z_2, t_2), \cdots \to (z, \infty)$.  By passing to subsequence we can 
assume that $\phi(F(z_i,t_i))\to \mL\in \mL(S)$ where convergence is in the Hausdorff topology.  
If $\mL$ is not a diagonal extension of $g(z)$, then $\mL$ is transverse to $g(z)$ and hence 
$d_{PT(S)}(F(z_i,t_i),g(z))>\epsilon$ for $i$ sufficiently large and some $\epsilon>0$.  By passing to subsequence we can assume that $z_i\in 
\kappa_i$, where $\kappa_i$ is a simplex in $K_{i}$, where $\kappa_{1}\supset 
\kappa_{2}\supset \cdots$ is a nested sequence of simplicies and $\cap \kappa_{i}=z$.  Let $n_i$ 
denote the greatest integer in $t_i$.  Apply 
the second conclusion of Lemma \ref{pml approximation} to $\kappa_i, \delta_{n_i}, \delta_{n_i+1}$ to conclude that lim$_{i\to 
\infty} d_{PT(S)}(z, \phi(F(z_i, t_i)))=0$, a contradiction.      \qed

\vskip 10pt

\begin{remark}\label{convexity}  If there exists a train track $\tau$ such that each $z\in g(S^{k-1})$ is carried by $\tau$ and $V\subset P(\tau)$ is the convex hull of $\phiinv(g(S^{k-1}))$, then there exists a continuous extension $F:B^k\to \PMLEL$ such that $F(\inte(B^k))\subset V$.  Indeed, since $V$ is convex we can dispense with the use of the $\pi_\delta$'s and directly  construct  the maps $f_i, f_{i,i+1}, f_1'$ to have values within $V$.  \end{remark}

The following local version is needed to prove local $(k-1)$-connectivity of $\EL$, when $k\le n$ and $\dim(\PML)=2n+1$.  

\begin{proposition} \label{local extension}  If $z\in\EL$, then for every neighborhood $U$ of $
\phiinv(z)$ there exists a neighborhood $V$ of $\phiinv(z)$ such that  if $g:\Sk \to \EL$ is continuous and $\psiinv(g(\Sk))\subset V$, then there exists a generic PL map $F:B^k
\to \PMLEL$ such that $F|(\inte(B^k))\subset U$ and $F|\Sk=g$.  \end{proposition}

\begin{proof}  There exists a parametrized pair of pants decomposition of $S$ such that $z$ is fully carried by a maximal standard train track $\tau$.  Thus $\hatphiinv(z)$ is a closed convex set in $\inte(V(\tau))-0$.  If $\hat U=\pinv(U)$, then $\hat U\cap (\inte(V(\tau)))$ is a neighborhood of $\hatphiinv(z)$, since $\tau$ is maximal.  Let $\hat V\subset \inte(V(\tau))$ be a convex neighbohood of $\hatphiinv(z)$ saturated by open rays through the origin such that $\hat V\subset \hat U$.  Let $V=p(\hat V)$.  Then $V$ is a convex neighborhood of $\phiinv(x)$ with $V\subset U$.  By Remark \ref{convexity} if $\psi(g(\Sk)) \subset V$, then there exists a continuous map $F:B^k\to \PMLEL$ such that $F|\Sk=g$ and $F(\inte(B^k))\subset V$.  Now replace F by a generic perturbation.\end{proof}

\section{Markers}

In this section we introduce the idea of a \emph{marker} which is a technical device for controlling geodesic laminations in a hyperbolic surface.  In the next section, using markers, we will show that under appropriate circumstances a sequence of maps $f_i:B^k\to \PMLEL,\ i=1,2, \cdots$ extending a given continuous map $g:S^{k-1}\to \EL$ converges to an extension $f_\infty: B^k\to  \EL$.  As always, $S$ will denote a finite type surface with a fixed complete hyperbolic metric.

\begin{definition}  Let $\alpha_0, \alpha_1$ be open embedded geodesic arcs in $S$.  A \emph{path} from $\alpha_0$ to $\alpha_1$ is a  continuous map $f:[0,1]\to S$ such that for $i=0,1, f(i)\subset \alpha_i$.  Two paths are \emph{path homotopic} if they are homotopic through paths from $\alpha_0$ to $\alpha_1$.   Given two path homotopic paths $f,g$ from $\alpha_0$ to $\alpha_1$, a lift $\tilde\alpha_0$ of $\alpha_0$ to $\BH^2$ determines unique lifts $\tilde f, \tilde g, \tilde \alpha_1$ respectively of $f, g, \alpha_1$ so that $\tilde f, \tilde g$ are homotpic paths from $\tilde \alpha_0$ to $\tilde \alpha_1$.  Define $d_{\tilde H}(f,g)=d_H(\tilde f(I), \tilde g(I))$, where $d_H$ denotes Hausdorff distance measured in $PT(\tilde S)$.  Note that this is well defined independent of the lift of $\alpha_0$.\end{definition}

\begin{definition} \label{marker} A \emph{marker} $\mM$ for the hyperbolic surface $S$ consists of two embedded (though not necessarily pairwise disjoint) open geodesic arcs $\alpha_0, \alpha_1$ called \emph{posts} and a path homotopy class $[\alpha]$ from $\alpha_0$ to $\alpha_1$.  A representative $\beta$ of $ [\alpha]$ is said to \emph{span} $\mM$.  The marker $\mM$ is an \emph{$\epsilon$-marker} if whenever $\beta$ and $\beta'$ are geodesics in $S$ spanning $\mM$, then $d_{\tilde H}(\beta, \beta')<\epsilon$ and length$(\beta)\ge 1$.

Let $C$ be a simple closed geodesic in $S$.  A \emph{$C$-marker} is a marker $\mM$ such that if $\beta$ is a geodesic arc  spanning $\mM$, then $\beta$ is transverse to $C$ and $|\beta\cap C|> 4g+p+1$ where $g=\genus(S)$ and $p$ is the number of punctures.

In a similar manner we define the notion of \emph{closed $\epsilon$} or \emph{closed  $C$-marker}. Here the posts are closed geodesic arcs.   In this case the requirement $d_{\tilde H}(\beta, \beta')<\epsilon$ is replaced by $d_{\tilde H}(\beta, \beta')\le\epsilon$.  If $\mM$ is an $\epsilon$ or $C$-marker, then $\bar\mM$ will denote the corresponding closed $\epsilon$ or $C$-marker.  \end{definition}
\vskip10pt

\begin{definition}   We say that the geodesic $L$ hits the marker $\mM$ if there exists $\ge 3$ distinct embedded arcs in $L$ that span $\mM$.  We allow for 
the possibility that distinct arcs have non trivial overlap.  We say that the geodesic lamination $\mL$ \emph{hits} the marker $ \mM$ if 
there exists a leaf $L$ of $\mL$ that hits $\mM$.  If $b_1, \cdots, b_m$ are simple closed geodesics, then we say that $\mM$ is $\mL$-free of $\{b_1, \cdots, b_m\}$ if some leaf $L\notin  \{b_1, \cdots, b_m\}$ of $\mL$ hits $\mM$.  \end{definition}

\begin{lemma}  \label{finding epsilon markers} Let $S$ be a finite type hyperbolic surface with a fixed hyperbolic metric.  Given $\epsilon>0$ there exists $N(\epsilon)\in \BN$ such that if $\beta$ is an embedded geodesic arc, length$(\beta)\le 2$ and $\mL\in \mL(S)$ is such that $|\mL\cap \beta|>N(\epsilon)$, then there exists an $\epsilon$-marker $\mM$ hit by $\mL$ with posts $\alpha_0, \alpha_1\subset \beta$.\end{lemma}

\begin{proof}  $\mL$ has at most $6|\chi(S)|$ boundary leaves.  Thus some leaf  of $\mL$ hits $\beta$ at least $|\mL\cap\beta|/(6 |\chi(S)|)$ times.  Since length$(\beta)$ is uniformly bounded, if $|\mL\cap \beta|$ is sufficiently large, then three distinct segments of some leaf must have endpoints in $\beta$, be nearly parallel and have length $\ge 2$.  Now restrict to appropriate small arcs of $\beta$ to create $\alpha_0$ and $\alpha_1$ and let $[\alpha] $ be the class represented by the three segments.\end{proof}

\begin{lemma}  \label{markers exist} If $\mL\in \LS$ has a non compact leaf $L$, then for every $\epsilon>0$ there exists an $\epsilon$-marker hit by $L\in \mL$.\qed\end{lemma}

\begin{corollary}  \label{markers exist for el}  If $\mL\in \EL$, then for every $\epsilon>0$ there exists an $\epsilon$-marker hit by $\mL$.\qed\end{corollary}

The next lemma states that hitting a marker is an open condition.

\begin{lemma} \label{markers open} If $\mL\subset \mL(S)$ hits the marker $\mM$, then there exists a $\delta>0$ such that if $\mL'\in \mL(S) $and $\mL\subset N_{PT(S)}(\mL',\delta)$, then $\mL'$ hits $\mM$.  \qed\end{lemma}

By super convergence we have

\begin{corollary}  If $x\in \PML$ is such that $\phi(x)$ hits the marker $\mM$, then there exists an open set $U$ containing $x$ such that $y\in U$ implies that $\phi(y)$ hits $\mM$.\qed\end{corollary}

\begin{lemma}  \label{finding c markers} Let $C\subset S$ a simple closed geodesic.  There exists a $k>0$ such that if  $\mL\in \mL(S)$ and $|C\cap \mL|>k$, then there exists a $C$-marker  that is hit by $\mL$.  
\end{lemma}

\begin{proof}  An  elementary topological argument shows that if $k$ is sufficiently large, then there 
exists a leaf  $L$ containing 5 distinct, though possibly overlapping embedded subarcs $u_1, \cdots, u_5$ with endpoints in $C$ which 
represent the same path homotopy class rel $C$ such that the following holds.   Each arc $u_j$ intersects $C$ more than $4g+p+1$ times and fixing a preimage 
$\tilde C$ of $C$ to $\tilde S$, these arcs have lifts to arcs $\bar u_1, \cdots, \bar u_5$ in $\tilde S$ starting at $\tilde C$ and 
ending at the same preimage $\tilde C'$.  After reordering we can assume that $\bar u_1$ and $\bar u_5$ are outermost.

Let $\hat \alpha_0$ and $\hat \alpha_1$ be the maximal closed 
arcs respectively in $\tilde C$ and $\tilde C'$ with endpoints in $\cup \bar u_i$ and for $i=0,1$ let $\bar \alpha_i=\pi(\hat
\alpha_i)$, where $\pi$ is the universal covering map.

 This gives rise to a $C$-marker $\mM$ with posts $\alpha_0, \alpha_1$ where $\alpha_i=\inte(\bar\alpha_i)$ and $u_2, u_3, u_4$ represent the path homotopy class.  Note that if the 
 geodesic arc $\beta$ spans $\bar\mM$, then $\beta$ lifts to $\hat \beta$ with endpoints in $\hat\alpha_0$ and $\hat\alpha_1$.  Being a geodesic it lies in the geodesic rectangle 
 formed by $\hat\alpha_0, \hat\alpha_1, \bar u_1, \bar u_5$.  Thus it intersects 
 $C$ more than $4g+p+1$ times.\end{proof}

In the rest of this section  $V$ will denote the underlying space of a finite simplicial complex.  In application, $V=B^k$ or $S^k\times I$. 

\begin{definition}  A \emph{marker family} $\mJ$ of $V$ is a finite collection $(\mM_1, W_1), \cdots, (\mM_m, W_m)$ where each $\mM_i$ is a marker and each $W_j$ is a compact subset of $V$.   Let $f:V\to \PMLEL$.  
We say that $f$ \emph{hits} the marker family $\mJ$ if for each $1\le i\le m$ and  $t\in W_i$,  $\phi(f(t))$ hits $\mM_i$.  
Let $C=\{b_1,\cdots, b_q\}$ be a set of simple closed 
geodesics.  We say that $\mJ$ is $f$-\emph{free of $C$} if for each $1\le i\le m$ and $t\in W_i$, $\mM_i$ is $\phi(f(t))$-free of $C$.  More generally, if $U\subset V$, then we say that $f$ \emph{hits $\mJ$ along $U$} (resp. \emph{$\mJ$ is $f$-free of $C$ along $U$}) if for each $1\le i\le m$ and $t\in U\cap W_i$, $\phi(f(t))$ hits $\mM_i$ (resp. $\mM_i$ is $\phi(f(t))$-free of $C$).  We say that the homotopy $F:V\times I\to \PMLEL$ is $\mJ$-marker preserving if for each $t\in I$, $F|V\times t$ hits $\mJ$.  

Note that if $\mJ$ is $F$-free of $C$, then $F$ is in particular a $\mJ$-marker preserving homotopy.

An \emph{$\epsilon$-marker cover} (resp. \emph{$C$-marker cover}) of $V$ is a marker family $(\mM_1, W_1),\cdots , (\mM_m, W_m)$ where each $\mM_i$ is an $\epsilon$-marker (resp. $C$-marker)   and the interior of the $W_i$'s form an open cover of $V$. 
  
 \end{definition}
 
 The next lemma gives us conditions for constructing $\epsilon$ and $C$-marker families.

\begin{lemma}  \label{finding markers}  Let $S$ is a finite type hyperbolic surface such that $\dim(\PML)=2n+1$.  Let $V$ be a finite simplical complex.

i)  If $\epsilon>0$ and $f:V\to \PMLEL$ is a generic PL map such that $\dim(V)\le n$, then there exists an $\epsilon$-marker family $\mE$ hit by $f$.

ii)  Given the simple closed geodesic $C$, there exists $N(C)\in \BN$ such that if $f:V\to \PMLEL$ is such that  for all $t\in V$, $|\phi(f(t))\cap C|\ge N(C)$, then there exists a $C$-marker family $\mS$ hit by $f$.\end{lemma}  

\noindent{Proof of i)}.  Since $k\le n$ and $f$ is generic, for each $t\in V$,  $
\mA(\phi(t))\neq\emptyset$, where $\mA(\phi(t))$ is the arational sublamination 
of $\phi(t)$.   By Lemma \ref{finding epsilon markers} for each $t\in V$ there 
exists an $\epsilon$-marker $\mM_t$ and compact set $W_t$ such that $t\in 
\inte(W_t)$ and for each $s\in W_t$, $\phi(f(s))$ hits $\mM_t$.  The result follows 
by compactness of $V$.\qed

\vskip 10 pt

\noindent{Proof of ii)}.  Given $C$, choose $N(C)$ as in Lemma \ref{finding c markers}.  Thus for each $t\in V$ there exists a $C$-marker $\mM_t$ and compact set $U_t$ such that $t\in \inte(U_t)$ and for each $s\in U_t$, $\phi(f(s))$ hits $\mM_t$.  The result follows by compactness of $V$.\qed

\section{convergence lemmas}

This section establishes various criteria to conclude that a sequence of ending laminations converges to a particular ending lamination or to show that two ending laminations are close in $ \EL$.  We also show that markers give neighborhood bases of elements $\mL\in \EL$ and sets in $\PML$ of the form $\phiinv(\mL)$, where $\mL\in \EL$.

\begin{lemma} \label{one closeness} Let $\mu\in \EL$ and $$W_\epsilon (\mu)=\{\mL\in \EL | 
\dpts(\mL, \mu')<\epsilon, \mu' \textrm{is a diagonal extension of } \mu\}.$$  
Then $\mW(\mu)=\{W_\epsilon(\mu)|\epsilon>0\}$  is a neighborhood basis of $
\mu\in \EL$.\end{lemma}

\begin{proof}  By definition of coarse Hausdorff topology, $W_\epsilon(\mu)$ is open in $\EL$.  
Therefore if the lemma is false, then there exists a sequence $
\mL_1, \mL_2, \cdots$ such that $\mL_i\in W_{1/i}(\mu)  $ all $i$ 
and a $c>0$ such that for all $i, \mL_i\notin \Npts(\mu', c)$ for all 
diagonal extensions $\mu'$ of $\mu$.  After passing to subsequence 
we can assume that $\{\mL_i\}\to \mL_\infty$ with respect to the 
Hausdorff topology.  If $ \mL_\infty$ is not a diagonal extension of $
\mu$, then $\mL_\infty$ is transverse to each diagonal extension of 
$\mu$ and hence there exists an $\epsilon>0$ such that $\dpts(\mu',
\mL_i)>\epsilon$ for all i sufficiently large and every diagonal 
extension of $\mu$, a contradiction.\end{proof}

\begin{lemma} \label{pml one closeness}  Let $\mu\in \EL$, $\mu'$ a diagonal extension  and $x_1, x_2, \cdots \in \PML$ such that $\lim_{i\to\infty} \dpts(\phi(x_i), \mu')=0$, then after passing to subsequence $x_i\to x_\infty\in \phiinv(\mu)$.\end{lemma}

\begin{proof}  After passing to subsequence we can assume that $x_i\to x_\infty
\in \PML$.  If $\phi(x_{\infty})\neq \mu$, then $\phi(x_\infty)$  intersects $\mu$ transversely.  Let $p \in \phi(x_\infty)\cap \mu$ and $L$ the leaf of $\mu$ containing $p$.  Then $\phi(x_\infty)$ intersects $L$ at $p$ at some angle $\theta>0$.   By super convergence, for $i$ sufficiently large $\phi(x_i)$ intersects $\mu$ at $p_i\in L$ at angle $\theta_i$, where $p_i$ is very close to $p$ (distance measured intrinsically in $L$) and $\theta_i$ is very close to $\theta$.  Since every leaf of $\mu'$ is dense in $\mu$, it follows that there exists $N>0$ such that if $J$ is a segment of a leaf of $\mu'$ at least length $N$ and $i$ is sufficiently large, then $J\cap\phi(x_i)\neq\emptyset$ with angle of intersection at some point at least $ \theta/2$.  Thus  $\dpts(\phi(x_i), \mu')$ must be uniformly bounded below, else some $\phi(x_i)$ would have a transverse self intersection.  \end{proof}

\begin{lemma}  \label{pts separation}   If $K\subset \PML$ and $L\subset \EL$ are compact and $K\cap \phiinv(L)=\emptyset$, then there exists $\delta>0$ such that if $\dpts(\phi(x),\mu)<\delta$ where $x\in \PML$ and $\mu\in L$, then $x\notin K$.\end{lemma}

\begin{proof}  Otherwise there exists  sequences $x_1, x_2, \cdots \to x_\infty$, $\mu_1, \mu_2, \cdots \to \mu_\infty$ such that the $x_i's \in K$ and the $\mu_j's \in L$ and $lim_{i\to \infty} \dpts(\phi(x_i),\mu_i)=0$.  After passing to subsequence we can assume that the $\mu_i$'s converge to a diagonal extension of $\mu_\infty$.  This contradicts Lemma \ref{pml one closeness}.\end{proof}  

\begin{lemma}  \label{preimage close}  If $L\subset \EL$ is compact and $U\subset \PML$ is open such that $\phiinv(L)\subset U$, then there exists a neighborhood $V$ of $L$ such that $\phiinv(V)\subset U$.\end{lemma}

\begin{proof}  Let $W_\epsilon (L)=\{\mL\in \EL | 
\dpts(\mL, \mu')<\epsilon, \mu' \textrm{is a diagonal extension of } \mu\in L\}.$  Then $W_\epsilon (L)$ is open and  for $\epsilon$ sufficiently small $\phiinv(W_\epsilon(L))\subset U$ otherwise taking $K=\PML\setminus U$ one obtains a contradiction to the previous lemma.\end{proof}

\begin{lemma}  \label{marker closeness}  Let $\mu\in \EL$ and $\mM_1, \mM_2, 
\cdots$ a sequence of markers such that for every $i\in \BZ$, $\mM_i$ is a $1/i
$-marker hit by $\mu$.  If $U_i=\{\mL\in \EL|\mL $ \textrm{hits} $ \mM_i\}$, then $\mU=\{U_i
\}$ is a neighborhood basis of $\mu$ in $\EL$.\end{lemma}

\begin{proof}  By definition of $1/i$-marker, if $\mL$ hits $\mM_i$, then $\dpts(\mL, \mu)<1/i$. Therefore, for all $i$, $U_i\subset W_{1/i}$.  Since each $U_i$ is open in $\EL$ the result follows.\end{proof}

\begin{lemma} \label{two closeness}  Let $\mu\in \EL$.  For each $\epsilon>0$ there exists $ \delta>0$ such that if $\{\mL^1, \mL^2, \cdots, \mL^k, z\} \subset \mL(S)$, $\dpts(\mL^k, z)<\delta$, $\dpts(\mL^i, \mL^{i+1})<\delta$ for $1\le i\le k-1$ and $\dpts(\mu', \mL^1)<\delta$  for some diagonal extension $\mu'$ of $\mu$, then $\dpts(\mu,z)<\epsilon$.\end{lemma}

\begin{proof}  We give the proof for $k=1$, the general case being 
similar.  If the Lemma is false, then there exists a sequence $(\mL_i, 
z_i, \delta_i)$ for which the lemma fails, where $\delta_i\to 0$.  After 
passing to subsequence we can assume that $\mL_i\to \mL_\infty$ 
and $z_i\to z_\infty$ with respect to the Hausdorff topology.  Since 
$z_\infty$ is nowhere tangent to $\mu$, it is transverse to $\mu$.  
Since the ends of every leaf of every diagonal extension of $\mu$ is 
dense in $\mu$ and $\mu$ is filling,  there exists $K>0$ such that 
any length $K$ immersed segment lying in a leaf of $z_\infty$ intersects any 
length $K$ segment lying in any diagonal extension of $\mu$ at 
some angle uniformly bounded away from $0$.  Thus a similar 
statement holds for each $z_i, i$ sufficiently large, where $K$ is 
replaced by $K+1$.  Therefore if $\delta_i$  is sufficiently small and 
$i$ is sufficiently large, then $\mL_i$ has length $K+2$ immersed segments $
\sigma_1, \sigma_2$ such that $\sigma_1$ is nearly parallel  to a 
leaf of $\mu'$ and $\sigma_2$ is nearly parallel to a leaf of $z_i$.  
This implies that $\sigma_1$ nontrivially intersects $\sigma_2$ 
transversely, a contradiction.\end{proof}

\begin{lemma}  \label{unzipping} Let $\epsilon>0$.  Let $\tau_1, 
\tau_2, \cdots$ be a full unzipping sequence of the transversely 
recurrent train track $\tau_1$.  If each $\tau_i$ fully carries the 
geodesic lamination $\mL$, then there exists $N>0$ such that if $
\mL_1$ is carried by $ \tau_i$, for some $i\ge N$, then $
\dpts(\mL_1, \mL)<\epsilon$.\end{lemma}

\begin{proof}  This follows from the proof of Lemma 1.7.9 \cite{PH} 
(see also Proposition 1.9 \cite{G1}).  That argument shows that each 
biinfinite train path of each $\tau_i$ is a uniform quasi-geodesic and 
that given $L>0$, there exists $N>0$ such that any length $L$ 
segment lying in a leaf of a lamination carried by  $\tau_i, i\ge N$, is 
isotopic to a leaf of $\mL$ by an isotopy such that the track of a 
point has uniformly bounded length. \end{proof}

\begin{lemma}  \label{carries diagonal extension} If $\tau$ is a train track that carries $\mu\in \EL$, then $\tau$ fully carries a diagonal 
extension of $\mu$.\end{lemma}

\begin{proof}  By analyzing the restriction of $\tau$ to each closed 
complementary region of $\mu$, it is routine to add diagonals to $
\mu$ to obtain a lamination fully carried by $ \tau$.  \end{proof}

\begin{lemma}  \label{splitting}  Let $\kappa$ be  a
transversely recurrent train track that carries $\mu\in \EL$.  Given $\delta >0$ there exists 
$N>0$ so that if $\tau$ is obtained from $\kappa$ by a 
sequence of $\ge N$ full splittings (i.e. along all the large branches) 
and  $\tau$ carries both $\mu$ and $\mL\in \mL(S)$, then $
\dpts(\mL, \mu)<\delta$.\end{lemma}

\begin{proof} It suffices to show that  $\dpts(\mL,\mu')<\delta_1$  for some  diagonal extension $\mu'$ of $\mu$ and some $\delta_1>0$, that depends on $\delta$ and $\mu$.  Since  there are only finitely many train tracks obtained from a given finite number of full splittings of $\kappa$ it follows that if the lemma is false, then there exist $\tau_1, \tau_2, \cdots$ such that $\tau_1=\kappa,  \tau_i$ is a full splitting of $\tau_{i-1}$ and for each $ i\in \BN$ there exists $\mL_i\in \mL(S)$ carried by some splitting of $\tau_{n_i}$ with $\dpts(\mL_i, \mu)>\delta$ and $n_i\to\infty$.  Note that $\mL_i$ is also carried by $\tau_{n_i}$.  Since $\mu$ has only finitely many diagonal extensions we can assume from the previous lemma that each $\tau_i$ fully carries a fixed diagonal extension $\mu'$ of $\mu$.  

On the other hand, there is a full unzipping sequence $\tau_1'=\tau_1, \tau_2', \cdots$ with the property that each $\tau_i'$ carries exactly the same  laminations as some $\tau_{m_i}$ and $m_i\to \infty$.  Thus by Lemma \ref{unzipping} it follows that for $i$ sufficiently large $\dpts(\mL_i, \mu')<\delta_1$, a contradiction.\end{proof}

For once and for all fix a parametrized pants decomposition of $S$, with the corresponding finite set of standard train tracks.

\begin{proposition} \label{convergence} Given $\epsilon>0, \mu\in \EL$, there exists $N>0, \delta>0$ such that if $\tau$ is obtained from a standard train track by $N$ full splittings and $\tau$ carries $\mu$, then the following holds.  If $\mL\in \mL(S)$ is carried by $\tau$, $z\in \mL(S)$ and $\dpts(\mL, z)<\delta$, then $\dpts(\mu, z)<\epsilon$.  In particular, if $\mL$ is carried by $\tau$, then $\dpts(\mu, \mL)<\epsilon$. \end{proposition}

\begin{proof}  Apply Lemmas \ref{splitting} and \ref{two closeness}.  \end{proof}

 
 \section{A criterion for constructing continuous maps of compact manifolds into $\EL$}

This section is a generalization of the  corresponding one of \cite{G1} where a criterion was established for constructing continuous paths in $\EL$.     Our main result is much more general and  is technically much simpler to verify.  It will give a criterion for extending a 
continuous map $g:S^{k-1}\to \EL$ to a continuous map $\mL:B^k\to \EL$, though it is stated in a somewhat more general form.

Recall that our compact surface $S $ is endowed with a fixed hyperbolic metric. 
Let $\{C_i\}_{i\in \BN}$ denote the set of simple closed geodesics in $S$.

\begin{notation}   If $\mU_j$ is a finite open cover of a compact set $V$, then its elements will be denoted by $U_j(1),\cdots U_j(k_j)$. \end{notation}

\begin{proposition} \label{continuity criterion} Let $V$ be the underlying space of a finite simplicial complex and $W$ the subspace of a subcomplex.   Let $g:W\to \EL$ and for $i\in \BN$ let $f_i:V\to \PMLEL$ be continuous extensions of $g$.  Let $\mL_m(t)$ denote $\phi(f_m(t))$.  Let $\epsilon_1, \epsilon_2, \cdots$ be such that for all $i$, 
$\epsilon_i/2>\epsilon_{i+1}>0$. Let   $\mU_1,\mU_2,\cdots$ be a sequence of finite open covers of $V$.     Suppose that each $U_j(k)$ is assigned both an $\epsilon_j$-marker $\alpha_j(k)$ and a $C_j$-marker $\beta_j(k)$.   Assume that the following two conditions hold.
   
\vskip 10pt

\noindent \textrm{(sublimit)}   For each $t\in U_j(k)$ and $m\ge j$, $\mL_m(t)$ hits $\alpha_j(k)$.  

\vskip 10pt

\noindent (\text{filling})   For each $t\in U_j(k)$ and $m\ge j$, $\mL_m(t)$ hits $\beta_j(k)$. 

\vskip 10pt

Then there exists a continuous map $\mL:V\to \EL$ extending $g$ 
so that   for $t\in V$, $\mL(t)$ is the coarse Hausdorff limit of $\{\mL_m(t)\}_{m\in \BN}$.
\end{proposition}

\begin{proof}    Fix $t$.  We first construct a minimal and filling $\mL(t)$.  After passing to subsequence we can assume that the sequence $\mL_{m_i}(t)$ converges in the Hausdorff topology to a lamination $\mL'(t)$.  If $t \in U_i(j)$, then the filling and sublimit conditions imply that if $k>i$, then some arcs $\gamma_i(k), \sigma_i(k)$ in leaves of $\mL_k(t) $ respectively span the $\epsilon_i$ and $C_i$-markers $\alpha_i(j)$ and $\beta_i(j)$.  This implies that arcs in $\mL'(t)$ span the corresponding closed markers and hence, $\mL'(t)$ intersects each $C_i$ transversely and hence $\mL^\prime(t)$ contains no closed leaves.  Thus spanning arcs in $\mL'(t)$ are embedded (as opposed to wrapping around a closed geodesic) and hence $|\mL'(t)\cap C_i|>4g+p+1$ for all  $i$.  Let $\mL(t)$ be a minimal sublamination of $\mL^\prime(t)$.  If 
$\mL(t)$ is not filling, then there exists a simple closed geodesic $C$, disjoint from 
$\mL(t)$ that can be isotoped into any neighborhood of  $\mL(t)$ in $S$.    An 
elementary topological argument shows that  $|C\cap \mL^\prime(t)| \le 4g +p +1$, contradicting the filling condition.

We next show that $\mL(t)$ is 
independent of subsequence.    Let $\mL_0^\prime(t)\in \EL$ be a  lamination that is the Hausdorff limit of the 
subsequence $\{\mL_{k_i}(t)\}$ and $\mL_0(t)$ the sublamination of $\mL_0'(t)$ in $\EL$.  By the sublimit condition each of $\mL'(t), \mL_0'(t)$ have arcs that span the same set $\{\bar\alpha_i\}$ of closed markers, where $\alpha_i$ is an $\epsilon_i$-marker with associated open set $U_i\subset V$, where $t\in U_i$.  Since $\epsilon_i\to 0$, the lengths of the initial posts $\{\alpha_{i_0}\}$ go to $0$.  Thus after passing to a subsequence of the initial posts, $\{\alpha_{i_{j_0}}\}\to x\in S$.  Now let $v_{i_j}$ be the unit tangent vector to the initial point of some spanning arc of $\alpha_{i_j}$.  After passing to another subsequence, $v_{i_j}\to v$ a unit tangent vector to $x$.  The sublimit condition implies that $v$ is tangent to a leaf of both $\mL'(t)$ and $\mL_0'(t)$ and hence $\mL'(t)$ and $\mL_0'(t)$ have a leaf in common.  It follows that $\mL(t)=\mL_0(t)$.

 We apply Lemma \ref{continuity3} to show that $ f$ is continuous at $t$.  Let $v$ and $\{\alpha_{i_j}\}_{i\in 
 \BN}$ be as in the previous paragraph, where $\{\alpha_{i_j}\}_{i\in 
 \BN}$ is the final subsequence produced in that paragraph.  Fix 
 $\epsilon>0$.  There exists $N\in \BN$ such that for $i\ge N$, $d_{PT(S)}
 (v_{i_j}',v)\le\epsilon$ where $v_{i_j}'$ is any unit tangent vector to the initial point of a spanning arc of $\bar
 \alpha_{i_j}$.   Therefore if $m\ge N_j$ and $s\in U_{N_j}$, then $d_{PT(S)}
  (\mL_m(s), v)\le\epsilon$. Since this is true for all $m\ge N_j$ it follows that for all $s\in U_{N_j}$, \ $d_{PT(S)}(\mL'(s), \mL'(t))\le \epsilon.$\end{proof}


 \section{Pushing off of $B_C$}

Given a generic PL map $f:B^k\to \PMLEL$ and a simple closed geodesic $C$, this section will describe homotopies of $f$ such that if $f_1$ is a resulting map, then $\finv_1(B_C)=\emptyset$.  The map $f_1$ is said to be obtained from $f$ by \emph{pushing off of $C$}.  Various technical properties associated with such push off's will be obtained.  The concept of \emph{relatively pushing $f$ off of $C$} will be introduced and analogous technical results will be established.  In subsequent sections we will produce a sequence $f_1, f_2, \cdots$ satisfying the hypothesis of Proposition \ref{continuity criterion}, where $f_{i+1}$ is obtained by relatively pushing $f_i$ off of a finite set of geodesics, one at a time.

\begin{remark} Recall the convention that $n$ is chosen so that $\dim(\PML)=2n+1$.    Let $C$ be a simple closed geodesic.  Let $\lambda_C$ denote the projective measure lamination with support $C$.  As in \cite{G1}, we denote by $B_C$ the PL $2n$-ball consisting of those projective measured laminations that have intersection number 0 with $\lambda_C$.  Recall that $B_C$ is the cone  of the PL $(2n-1)$-sphere $\delta B_C$ to $\lambda_C$, where $\delta B_C$ consists of those points of $B_C$ which do not have $C$ as a leaf.  Furthermore,  if $x\in B_C\setminus \lambda_C$, then $x=p((1-t)\hat\lambda+t\hat \lambda_C)$, for some $\hat\lambda\in\ML$ representing a unique $\lambda\in \delta C$, some $\hat\lambda_C\in \ML$ representing $\lambda_C$ and  some $t<1$.  \end{remark}

\begin{definition}  The \emph{ray through $x\in B_C\setminus \lambda_C$} is the set of points $r(x)$ in $\PML$ represented by measured laminations of the form $\{t\hat\lambda+(1-t)\hat\lambda_C|0\le t\le1\}$ where $\hat\lambda$ and $\hat\lambda_C$ are as above.      If $K\subset B_C$  and $K\cap \lambda_C=\emptyset$, then define $r(K)=\cup_{x\in K} r(x)$.    \end{definition} 

\begin{remarks} \label{cone neighborhood} Note that $r(x)$ is well defined and $r(K)$ is compact if $K$ is compact.

Using the methods of  \cite{Th1}, \cite{PH} or \cite{G1} it is routine to show that there exists a neighborhood of $B_C$ homeomorphic to $2B^{2n}\times [-1,1]$, where $2B^{2n}$ denotes the radius-2 $2n$-ball about the origin in $\BR^{2n}$, such that $B_C$ is identified with $B^{2n}\times 0$, \ $\lambda_C$ is identified with $(0,0)$ and for each $x\in B_C\setminus \lambda_C$,\  $r(x)$ is identified with a ray through the origin with an endpoint on $S^{2n-1}\times 0$.

While the results in this section are stated in some generality, on first reading one should imagine that if $f:V\to \PMLEL$, then $V=B^k$ and $\finv(\EL)=S^{k-1}$, where $k\le 2n$.  \end{remarks}

\begin{definition}  \label{pushing off} Let $V$ be the underlying space of a finite simplicial complex and $W$ that of a subcomplex.  If  $f:V\to \PMLEL$ and $W=\finv(\EL)$, then  the generic PL map $f_1:V\to \PMLEL$ is said to be obtained from $f$ by  \emph{$\delta$-pushing  off of $B_C$}  if  there exists a homotopy $F:V\times I\to \PMLEL$, called a \emph{$(C, \delta)$ push off homotopy}  such that

i)  $\finv_1(B_C)=\emptyset$,

ii)  $F(t,s)=f(t)$ if either $s=0$ or $d_{V}(t, \finv(B_C))\ge \delta$ or $ d_{V}(t, W)\le \delta$ or  $d_{\PML}(f(t), B_C))\ge \delta$  and

iii)  for each $t\in V$ such that $d_{\PML}(f(t), B_C)<\delta$ there exists an $x\in f(V)\cap B_C$ such that for all $s\in [0,1]$, \ $d_{\PML}(F(t,s), r(x))<\delta$.  Furthermore if $F(t,s)\in B_C$, then $F(t,s)\in r(f(t))$.\end{definition}

\begin{lemma} \label{push} If $f:V\to \PMLEL$ is a generic PL map, $\dim(V)\le 2n$ and $C$ is a simple closed geodesic, then for all sufficiently small $ \delta>0$ there exists a $(C, \delta)$ push off homotopy of $f$.   \end{lemma}

\begin{proof}  It follows by super convergence and compactness that if $C$ is a simple closed geodesic,  then $\finv(B_C)$ is a compact set disjoint from some neighborhood of $W$. By genericity of $f$, i.e. Lemma \ref{missing}, there exists an $\epsilon_1>0$ such that $d_{\PML}(f(V), \lambda_C)\ge \epsilon_1$. Consider a natural  homotopy $F_\epsilon: ((2B^{2n}\setminus \epsilon B^{2n})\times [-1,1])\times I\to 2B^{2n}\times [-1,1]$ from the inclusion to a map whose image is disjoint from $B^{2n}\times 0$, which is supported  in an $\epsilon$-neighborhood of $(B^{2n}\setminus \epsilon B^{2n})\times 0$ and where points in $(B^{2n}\setminus \epsilon B^{2n})\times 0$ are pushed radially out from the origin.  Let $g:N(B_C)\to 2B^{2n}\times [-1,1]$ denote the parametrization given by Remark \ref{cone neighborhood}.  The desired $(C,\delta)$-homotopy is obtained by appropriately interpolating the trivial homotopy outside of a very small neighborhood of $\finv(B_C)$ with  $\ginv\circ F_\epsilon\circ g\circ f$  restricted to a small neighborhood of $\finv(B_C)$ where $\epsilon$ is sufficiently small and then doing a small perturbation to make $f_1$ generic.  \end{proof}

\begin{lemma}  \label{marker preserving}Let $C$ be a simple closed geodesic and $f:V\to \PMLEL$ be a generic PL map with $\dim(V)\le 2n$.  Let $\mJ$ be a marker family of $V$ hit by $f$ that is free of $C$.  If $\delta$ is sufficiently small, then any $(C, \delta)$ homotopy $F$ from $f$ to $f_1$ is $\mJ$ marker preserving, free of $C$.\end{lemma}

\begin{proof}  Since there are only finitely many markers in a marker family, it suffices to show that if $K\subset V$ is compact and $\phi(f(t))$ hits the marker $\mM$ free of $C$ at all $t\in K$, then for $\delta$ sufficiently small  $\phi(F(t,s))$ hits $\mM$ free of $C$ at all  $t\in K$ and $s\in I$.  This is a consequence of super convergence and 
compactness.   Indeed, if $x\in f(K)\cap B_C$, then there exists a leaf  of $\phi(x)$ distinct from $C$ hits $\mM$.  Since  all points in $r(x)\setminus \lambda_C
$ have the same underlying lamination 
this fact holds for all $y\in r(x)$.  By super convergence it holds at all points in a neighborhood of $r(x)$ in $\PML$.  Let $U$ be the union of these neighborhoods over all $x\in f(K)\cap B_C$.  By compactness of $K$ and $B_C$, there exists a $ \eta>0$ such that if $y\in B_C$ and $d_{\PML}(y,f(t))\le \eta$ for some $t\in K$, then $N_{\PML}(r(y), \eta)\subset U$.    Any $(C,\delta)$ homotopy with $\delta<\eta$ satisfies the conclusion of the lemma.
 \end{proof}

\begin{definition}  If $x\in \PML$ and $A$ is a simple closed geodesic, then define $g(x,A)\in \BZ_{\ge 0}\cup \infty$ the 
\emph{geometric intersection number} of $x$ with $A$ by $g(x,A)= \min\{|\phi(x)\cap A'|| A'$ is 
isotopic to $A\}$.  If $f:V\to \PMLEL$  define the \emph{geometric intersection 
number of $f$ with $A$} by $g(f,A)=\min\{g(\phi(f(t)), A)|t\in V\}$.  If 
$0<g(f,A)<\infty$, then we say that the multi-geodesic $J$ is a \emph{stryker curve for $A$} if for some 
$t\in V$, $J\subset \phi(f(t))$ and $|J\cap A|=g(f,A)$. We call $J$ the \emph{$f(t)$-stryker curve} or 
sometimes the \emph{stryker curve at $f(t)$.}     \end{definition}

\begin{remark}  Note that  $|\phi(x)\cap A'|$ is minimized when $A=A'$ unless $A$ is a leaf of $\phi(x)$ in which case $g(x,A)=0$.\end{remark}

\begin{lemma} \label{stryker finiteness} If $f:V\to \PMLEL,  A$ is a simple closed geodesic and  $0<g(f,A)<\infty$, then the set of stryker curves is finite.  Also $m(f,A)=\{t\in V||\phi(f(t))\cap A|=g(f,A)\}$ is compact. Finally $m(f,A)$ is the disjoint union of the compact sets $m_{J_1}(f,A), \cdots, m_{J_m}(f,A)$ where $t\in m_{J_i}(f,A)$ implies that $J_i$ is the stryker curve at $f(t)$. \end{lemma}

\begin{proof}  Super convergence implies that $V\setminus m(f,A)$ is open, hence $m(f,A)$ is compact. If 
the first assertion is false, then there exists $t_1,t_2, \cdots$ converging to $t$ such that if $J_i$ 
denotes the stryker curve at $t_i$, then the $J_i$'s are distinct.  By compactness, $t\in m(f,A)$, so let 
$J$ be the stryker curve at $t$.   Super convergence implies that if  $s$ is sufficiently close to $t$, then 
either $s\notin m(f,A)$ or $J$ is the stryker curve at s, a contradiction.  The final assertion again 
follows from super convergence.  \end{proof}

\begin{lemma}  \label{stryker preserving} Let $f:V\to \PMLEL$ be a generic PL map and $A$ and $C$ disjoint simple closed geodesics such that $0<g(f,A)<\infty$.  Then for $\delta$ sufficiently small, any $(C,\delta)$ push off $f_1$ satisfies $g(f_1, A)\ge g(f,A)$.  If equality holds and  $J$ is a stryker curve for $f_1$, then $J$ is a stryker curve for $f$.  \end{lemma}

\begin{proof}  If $t\in \finv(B_C)$, then $|(\phi(f(t))\setminus C)\cap A|=|\phi(f(t))\cap A|\ge g(f,A)$.  By super convergence there exists a neighborhood $U'$ of $r(f(t))$  such that $y\in U'$ implies that 
$g(\phi(y),A)\ge g(f,A)$ and hence  there exists a neighborhood $U$ of $r(f(V)\cap 
B_C)$ with the same property.   If $\delta$ is sufficiently small to have any $(C,\delta)$ push off 
homotopy supported in $U$, then $g(f_1,A)\ge g(f,A)$.  

Now assume that equality holds. If $t\in m(f,A)\cap \finv(B_C)$ and $J $ is the stryker curve at $f(t)$, then by super convergence there exists a neighborhood $U'$ of $r(f(t))$ such that if $x\in U'$ and $g(x, A)=g(f,A)$, then $J$ is the stryker curve at $x$.  Thus, there exists a neighborhood $U$ of $f(V)\cap B_C$ such that $x\in U$ and $g(x, A)=g(f,A)$, then the stryker curve at $x$ is a stryker curve of $f$.  
If $\delta$ is sufficiently small to have any $(C,\delta)$ push off of $f$ supported in $U$, then the second conclusion holds.  \end{proof}

This argument proves the following sharper result.

\begin{lemma} \label {refined partial stryker preserving} If $f:V\to \PMLEL$ and $\eta>0$, then there exists $\delta>0$ such that if $\delta$ is sufficiently small and $f_1$ is the result of  a $(C,\delta)$ push off homotopy of $f$, then $m(f_1, A)\subset N_{V}(m(f,A), \eta)$ and if $t\in m(f_1,A)$ and $d_{V}(t, m_{J_i}(f,A))<\eta$, then $J_i$ is the stryker curve at $f_1(t)$.\qed\end{lemma}

We need \emph{relative versions} of generalizations of the above results.

\begin{definition}  \label{partial pushing} Let $V$ be the underlying space of a finite simplicial complex and $W$ that of a finite subcomplex.  Let $f:V\to \PMLEL$ be a generic PL map with $\finv(\EL)=W$.  Let $K\subset\finv(B_C)$ be closed.  We say that the generic PL map $f_1:V\to \PMLEL$ is  obtained from $f$ by \emph{$(C, \delta, K)$ pushing off } if there exists a homotopy $F:V\times I\to \PMLEL$ called a $(C, \delta, K)$ \emph{ push off homotopy} such that 

i)  $f_1(K)\cap B_C=\emptyset$,

ii)  $F(t,s)=f(t)$ if either $s=0$ or $d_{V}(t, K)\ge \delta$ or  $d_{V}(t, W)\le \delta$ or $d_{\PML}(f(t), f(K))\ge \delta$ .

iii)  for each $t\in V$ such that $d_{\PML}(f(t), f(K))<\delta$ there exists an $x\in r(f(K))$ such that for all $s\in [0,1]$, \ $d_{\PML}(F(t,s), r(x))<\delta$.  Furthermore, if $F(t,s)\in B_C$, then $F(t,s)\in r(f(t))$ and if $f_1(t)\in\inte(B_C)$, then $f(t)\in \inte(B_C)$.\end{definition}

\begin{lemma} \label{partial push} If $f:V\to \PMLEL$ is a generic PL map, $\dim(V)\le 2n$, $C$  a simple closed geodesic and $K$ a closed subset of $\finv(B_C)$, then for every sufficiently small $\delta>0$ there exists a $(C, \delta, K)$ push off homotopy of $f$.   \end{lemma}

\begin{proof}  For $\delta$ sufficiently small let  $U\subset V\setminus N_{V}(W,\delta)$ be open such that $K\subset  U\subset N_{V}(K,\delta/10)\cap 
\finv(N_{\PML}(r(f(K)),\delta/10))$.  Let $\rho:V\to [0,1]$ be a continuous function such that $\rho(K)=1$ and $\rho(V\setminus U)=0$.  If $F(t,s)$ defines  a $(C,\delta)$ push off 
homotopy, then $F(t, \rho(t) s)$ suitably perturbed defines a $(C, \delta, K)$ push off homotopy.\end{proof}

\begin{lemma} \label{avoiding}  Let $V$ be a finite $p$-complex and $f:V\to \PMLEL$ a generic PL map, $C$  a simple closed geodesic, $L\subset \PML$ a finite $q$-subcomplex of $\CS$  and $K$ a closed subset of $\finv(B_C)$.  If $p+q\le 2n-1$ or $p\le n$, then for every sufficiently small $\delta>0$ any $(C, \delta, K)$ push off homotopy of $f$ is supported away from $L$.\end{lemma}   

\begin{proof}  Let $Z=(\partial B_C\cap L)*C$.  By Lemma \ref{missing},  $f(V)\cap (Z\cup L)=\emptyset$ and hence $r(f(K))\cap (Z\cup L)=\emptyset$.  Thus  the conclusion of the lemma holds provided $\delta<d_{\PML}(r(f(K)),Z\cup L)$.  \end{proof}



We have the following relative version of  Lemma \ref{marker preserving}.

\begin{lemma}  \label{partial marker preserving} Let $b_1, b_2, \cdots, b_r, C$ be simple closed 
geodesics, $f:V\to \PMLEL$ a generic PL map and $K\subset \finv(B_C)$ be compact.  Let $\mJ$ be a marker family  that is $f$-free of $\{b_1, \cdots, b_r, C\}$.  If $\delta$ is sufficiently small, then $
\mJ$ is $F$-free of $\{b_1, \cdots, b_r, C\}$ for any  $(C, \delta, 
K)$  push off homotopy $F$ of $f$. \qed\end{lemma}

\begin{definition}  Let $f:V\to \PMLEL$.  Let $Y$ be a compact subset of $V$ and $A$ a simple closed geodesic.  
Define $g(f,A;Y)=\min\{g(f(t),A)|t\in Y\}$, the \emph{$Y$-geometric intersection number of $f$ and $A$}.  If $0<g(f,A;Y)<\infty$, 
define the \emph{$Y$-stryker curves for $f$ and $A$} to be those multi-geodesics $\sigma$ such that for some $t\in Y$, 
$\sigma\subset \phi(f(t))$ and $|\sigma\cap A| = g(f,A,Y)$.  Define $m(f,A;Y)=\{t\in Y|g(f(t),A)=g(f,A;Y)\}$ and if $\sigma$ is a Y-stryker curve then define $m_\sigma(f,A;Y)=\{t\in Y| \sigma\subset \phi(f(t))$ and $ |\sigma\cap A|=g(f,A,Y) \}$.  We say that the set $B$ of simple closed geodesics 
 \emph{solely hits the marker $\mM$ at $t$} if  each leaf of the lamination $\phi(f(t))$ that hits $\mM$ lies in $B$.    If  $Z\subset V$ then let 
$S(f,\mM, B, Z)$ denote the set of points in  $Z$ where $f$ solely hits $\mM$.       \end{definition}

The proof of Lemma \ref{stryker finiteness} holds for in the relative case.

\begin{lemma} If $f:V\to \PMLEL$, $A$ is a simple closed geodesic, $Y\subset V$ is compact and $0<g(f,A;Y)$, then the set of $Y$-stryker curves is finite.   Also $m(f,A; Y)$ is compact and is the disjoint union of compact sets $m_{J_1} (f,A;Y), \cdots , m_{J_m}(f,A;Y)$ where $t\in m_{J_i}(f,A;Y)$ implies that $J_i$ is a $Y$-stryker curve at $f(t)$.\qed\end{lemma}

By super convergence we have the following result.

\begin{lemma}  If f:$V\to \PMLEL$ is a generic PL map, $B$ is a finite set of simple closed geodesics, $Z\subset V$ is compact,  then $S(f,\mM, B, Z)$ is compact.    \qed\end{lemma}

We have the following analogy of Lemma \ref{stryker preserving}.

\begin{lemma} \label{partial stryker preserving} Let $f:V\to \PMLEL$,  be a generic PL map, $\eta>0$ and $A
$ and $C$ disjoint simple closed geodesics.    Let $K$ be a closed 
subset of $\finv(B_C)$ and $Y\subset V$ be compact.  If $0<g(f,A;Y)<\infty$, then there exists a neighborhood $U$ of $Y$ such that for $\delta$ sufficiently small, any $(C, \delta, K)$ push off $f_1$ satisfies 
$g(f_1, A; \bar U)\ge g(f,A; Y)$.  If equality holds, then $m(f_1, A; \bar U)\subset N_{V}
(m(f,A;Y), \eta)$ and if $t\in m(f_1, A;\bar U)$ and $d_{V}(t, m_{J_i}(f,A;Y))<\eta$, then $J_i
$ is the $\bar U$-stryker curve to $A$ at $f_1(t)$.    \qed\end{lemma}

\section{Marker tags}

  \begin{definition} \label{tag} Let $A$ be a simple closed multi-geodesic in  $S$.  We say that $\tau$ is a \emph{tag} for $A$, if $\tau$ is a compact embedded geodesic curve (with $\partial \tau$ possibly empty) transverse to $A$ such that $\partial \tau\subset A$ and $\inte(\tau)\cap A\neq\emptyset$.  
  
    Let $\mM$ be a marker hit by the simple closed multi-geodesic $A$.  Then $r\ge 3$ distinct subarcs of $A$ span $\mM$, where $r\in \BN$ is maximal.  These arcs run from $\alpha_0$ to $\alpha_1$, the posts  of $\mM$.  Suppose that the initial points of these arcs intersect  $\alpha_0$ at  $c_1, \cdots, c_r$.  Let $\tau$ be the maximal subarc of $\alpha_0$ with endpoints in $ \{c_1,\cdots, c_r\}$.    Such a tag is called a \emph{marker tag}.\end{definition}

Given $ f:V\to \PMLEL$ that hits the marker $\mM$ we may need to find a new $f$ that hits the marker $\mM$ free of a particular multi-geodesic.   Tags are introduced to measure progress in that effort.  We will find a sequence of push off homotopies whose resulting maps intersect a given tag more and more so that we can ultimately invoke the following result.

\begin{lemma}  \label{tag freedom} Let $f:V\to \PMLEL$, $\mM$ a marker, $A$ a simple closed multi-geodesic that hits $\mM$ and $\tau$ the corresponding marker tag.    Let $b_1, \cdots, b_r$ be  simple closed geodesics such that for all   $t\in \finv(B_A)$,  $ |(\phi(f(t))\setminus ( \cup_{i=1}^r b_i\cup A))|\cap \tau\ge 3(3g-3 +p)$, then $\mM$ is $f$-free of $\{A, b_1, \cdots, b_r\}$ along $\finv (B_A)$. \end{lemma}

\begin{proof} If $f(t)\in B_A$, then any leaf $L $ of $\phi(f(t))$ distinct from $A$ with $L\cap \tau\ge m$ has at least $ m$ distinct subarcs that span $\mM$.  If $L$ is a non compact leaf of $\phi(f(t))$ and $L\cap \tau\neq\emptyset$, then $|L\cap \tau|=\infty$, since $L$ is non proper.  If only closed geodesics of $\phi(f(t))$ intersect $\tau$, then since $\phi(f(t))$ can have at most $3g-3+p$ such geodesics, one of them $L$ distinct from $\{b_1, \cdots, b_r, A\}$ must satisfy $|L\cap \tau|\ge 3$.\end{proof}

\begin{definition}   Let $\tau$ be a tag for the multi-geodesic $A$,  $f:V\to \PMLEL$ a generic PL map and $Y$ a compact subset of $\finv(B_A)$.  Define $g(f,\tau, Y)= \min\{|(\phi(f(t))\setminus A)\cap \tau||t\in Y\}$  the \emph{geometric intersection number of $f $ with $\tau$ along $Y$}.

If $0<g(f,\tau,Y)<\infty$, then define the multi-geodesic $J$ to be a \emph{$Y$-stryker curve} for $\tau$ if $J\subset \phi(f(t))$, $J\cap A=\emptyset$ and $|J\cap \tau|=g(f,\tau, Y)$.  \end{definition}

The proof of Lemma \ref{stryker finiteness}  readily generalizes to the following result.

\begin{lemma} \label{tag stryker finiteness} If $f:V\to \PMLEL$ is a generic PL map, $\tau$ is a tag for the simple closed geodesic $A$ and $Y$ is closed in $\finv(B_A)$, then the set of $Y$-stryker curves is finite.   Also the set $m(f,\tau, Y)=\{t\in Y|g(f,\tau,Y)=|(\phi(f(t))\setminus A)\cap\tau|\}$ is compact and and canonically partitions as the disjoint union of the compact sets $m_{J_1}(f,\tau, Y), \cdots, m_{J_k}(f,\tau,Y)$ where $J_i$ is the $Y$-stryker curve to $\tau$ at all $t\in m_{J_i}(f,\tau, Y)$.  \qed\end{lemma}

Similarly, Lemma \ref{partial stryker preserving} generalizes to the following result.

\begin{lemma} \label{partial tag stryker preserving} Let $f:V\to \PMLEL$ be a generic PL map, $\tau$ a tag for the simple closed geodesic $A$, 
$Y$ a closed subset of $\finv(B_A)$ and $0<g(f,\tau,Y)<\infty$.  Let $C$ be a simple closed geodesic such that $C\cap(A\cup\tau)=\emptyset$ and $K$ a closed 
subset of $\finv(B_C)$.  If $\eta>0$, then there exists a neighborhood $U$ of $Y$ such that for $\delta$ sufficiently small, any $(C, \delta, K)$ push off $f_1$ satisfies 
$g(f_1, \tau, Y_1)\ge g(f, \tau, Y)$, where $Y_1=\finv_1(B_A)\cap \bar U$.   If equality holds, then $m(f_1, \tau, Y_1)\subset N_{V}(m(f,\tau,Y),\eta)$ and if $t\in m(f_1,\tau, Y_1)$ and $d_{V}(t, m_{J_i}(f,\tau,Y))<\eta$, then $J_i$ is the $Y_1$ stryker curve to $\tau$ at $f_1(t)$. In particular if $J$ is a $Y_1$-stryker curve for $f_1$ and $\tau$, then $J$ is a $Y$-stryker 
curve for $f$ and $\tau$.  \qed\end{lemma}

\section{Marker Cascades}

This very technical section begins to address the following issue.  To invoke Proposition 
\ref{continuity criterion} we need to find a sequence $f_1, f_2, \cdots$ satisfying the sublimit and filling conditions, 
in particular satisfying the property that $\finv_j(B_{C_i})=\emptyset$ for $i\le j$.  We cannot just create $f_{i}$ from 
$f_{i-1}$ by pushing off of $C_i$, because $C_i$ may be needed to hit previously constructed markers.   To make $C_i$ 
free of these markers we may need to relatively push off of other curves.  We may not be able to push off of 
those curves because they in turn are needed to hit markers.  In subsequent sections we shall see that finiteness of $S$, genericity of $f $ and 
the $k\le n$ condition will force this process to terminate.  Thus before we push off of $C_i$ we will do a sequence of 
relative push offs of other curves.  

We introduce the notion of \emph{marker cascade} to keep track of progress.   Given $f:V\to \PMLEL$, markers $\mM_1, \cdots, \mM_m$ and pairwise disjoint simple closed curves $a_1, \cdots, a_v$ a marker cascade is a (complicated) measure of how far $\mM_1, \cdots, \mM_m$ are from being free of $a_1, \cdots, a_v$.  At the end of this section we will show that under appropriate circumstances  relative pushing 
preserves freedom as measured by a marker cascade. The next section shows that judicious relative pushing increases the level of freedom.   See Proposition \ref{finished}.

\begin{definition}  \label{cascade} Let $V$ be the underlying space of a finite simplicial complex.  Associated to  $f:V\to \PMLEL$ a generic PL map, $\mJ=(\mM_1,W_1), \cdots, (\mM_m, W_m)$ a marker family hit by $f$, $\mM_1<\cdots < \mM_m$ the ordering induced from this enumeration and $a_1, \cdots, a_v$ a sequence of pairwise disjoint simple closed geodesics we define a \emph{marker cascade $\mC$}  which is a $v+1$-tuple  $ (\mA_1, \cdots,  \mA_v, \mP)$ where each $\mA_i$ is a 3-tuple $(\mA_i(i), \mA_i(ii), \mA_i(iii))$ that is defined below and $\mP$ is a finite set of $v$-tuples defined in Definition \ref{packet}.  To start with $\mA_i$ is organized as follows.\vskip10pt

\noindent$\mA_i(i)$ is either a marker $\mM_{i_j}$ or $\infty$.

\vskip8pt
\noindent$\mA_i(ii) \in \BZ_{\ge 0}\cup \infty$ is the geometric intersection number of $f$ with the tag $\tau_i
$ associated to  $a_i$ and $\mM_{i_j}$ along the compact set 
$m_i(\mC)\subset V$, unless $ \mA_i(i)=\infty$ in which case $\mA_i(ii)=\infty$.
\vskip8pt

\noindent$\mA_i(iii)$ is the set of stryker curves for $\tau_i$ along the compact set $m_i(\mC)\subset V$ unless $\mA_i(ii)=\infty$ in which case $\mA_i(iii)=\infty$.\vskip10pt

We define the $\mA_i$'s and the auxiliary $m_i(\mC)$'s as follows. \vskip10pt

\noindent$\mA_1(i)$ is defined to be the maximal marker $\mM_{1_j}$ such that $\mM_i$ is $f$-free of $a_1$ along $W_i$ for all $i<1_j$.  If $a_1$ is free of $\mJ$, then define $\mA_1(i)=\infty$.
\vskip 10 pt

If $\mM_{1_j}$ exists, then define $\tau_1$ to be the marker tag associated to $a_1$ and $\mM_{1_j}$ and  define  $m_1(\mC)=\{t\in S(f, a_1, \mM_{1_j}, W_{1_j})| g(f,\tau_1, S(f, a_1, \mM_{1_j}, W_{1_j}))=|(\phi(f(t))\setminus a_1)\cap \tau_1|\}$.\vskip8pt

\noindent$\mA_1(ii)= g(f,\tau_1, m_1(\mC))$ if $m_1(\mC)\neq\emptyset$ or $\mA_1(ii)= 
\infty$ otherwise.
\vskip8pt

\noindent$\mA_1(iii)$ is defined to be $\Stryker_1$  the set of $m_1(\mC))$-stryker curves 
for $\tau_1$, unless $\mA_1(ii)=\infty$ in which case $\mA_1(iii)=\infty$.
\vskip 10pt

Having defined $\mA_i, i< u$, then $\mA_u$ is defined as follows.  (The reader is encouraged to first read  Remark \ref{a(i)}).  To start with define $B^1_u, \cdots, B^m_u$ where $B^r_u=\{b^r_1, \cdots, b^r_u\}$, where $b^r_u=a_u$ and for $q<u, b^r_q=a_q$ if $r<q_j$ and $b^r_q=\emptyset$ otherwise.  
\vskip10pt

\noindent $\mA_u(i)$ is defined to be either the maximal marker $\mM_{u_j}$, such that $r<u_j$ implies that $\mM_r$ is free of $B^r_u$ along $ m_{u-1}(\mC)\cap W_r$ or $\mA_u(i)=\infty$ if for all $r\le m$, $\mM_r$ is free of $B^r_u$ along $ m_{u-1}(\mC)\cap W_r$.  \end{definition}

\begin{remark}  \label{a(i)} In words $\mA_u(i)=\mM_{u_j}$ is the maximal 
marker  such that all lower markers are free of $\{a_1, \cdots, a_u\}$, 
\emph{where applicable}.  Where applicable means two things.  First, the only 
relevant points are those of $m_{u-1}(\mC)$.  Second if say  $\mM_1, \mM_2, 
\mM_3$ are free of $a_1$ but $\mM_4$ is not and along $m_1(\mC)$, $\mM_1, 
\mM_2, \mM_3$ are  free of $\{a_1, a_2\}$ and  $\mM_4, \mM_5$ are free of 
$a_2$ but $\mM_6$ is not free of $a_2$, then $\mM_{2_j}=\mM_6$.  In 
particular, $a_1$ is irrelevant when considering $\mM_p$, for $p\ge 4$.  In this 
case $B^1_2=B^2_2=B^3_2=\{a_1, a_2\}$ and for $r>3$, $B^r_2=\{a_2\}$.  Note 
that if  $\mM_1$ is free of $\{a_1, a_2\}$ along $m_1(\mC)$ but $\mM_2$ is not, 
then $\mM_{2_j}=\mM_2$.  \end{remark}

\noindent\emph{Definition \ref{cascade} continued}  If $\mM_{u_j}$ exists, then define $\tau_u$ to be the marker tag arising from $a_u
$ and $\mM_{u_j}$.  Let $S_u=S(f,B^{u_j}_u, \mM_{u_j}, W_{u_j})\cap m_{u-1}
(\mC)$.  Define $m_u(\mC)= \{t\in S_u| g(f,\tau_u, S_u)=|\phi(f_t)\setminus a_u)
\cap \tau_u|\}$.

\vskip10pt
\noindent$\mA_u(ii)$ is defined to be either $g(f,\tau_u, m_u(\mC))$ or $\infty$ if 
$m_u(\mC)=
\emptyset$.
\vskip8pt
\noindent$\mA_u(iii)$ is defined to be 
the set $\Stryker_u$ which is either the set of $m_u(\mC))$-stryker curves for $\tau_u$ if $m_u(\mC)\neq\emptyset$ or  $\infty$ otherwise.

\vskip10pt

We say that the cascade $\mC$ is \emph{finished} if $m_v(\mC)=\emptyset$ and \emph{active} otherwise.  
We say that the cascade is \emph{based} on $\{a_1, \cdots, a_v\}$ and has \emph{length} $v$.   For $r\le v$, then the length-r cascade  based on $\{a_1,\cdots, a_r\} $ is called the length-$r$ \emph{subcascade} and denoted $\mC_r$.  Note that $\mC_r$ and $\mC$ have the same values of $\mA_1, \cdots, \mA_r$.

\begin{notation} The data corresponding to a cascade depends on $f$.  When the function must to be explicitly stated, we will use notation such as $\mC(f), m_i(\mC, f), \mA_p(f)$ or $\mA_r(ii,f)$.   \end{notation}

We record for later use the following result.

\begin{lemma} \label{leaf presence} Let $\mJ$ be a marker family hit by the generic PL map $f:V\to \PMLEL$.  If $\mC$ is an active cascade based on $a_1, \cdots, a_v$, then for every $t\in m_v(\mC)$, each $a_i$ is a leaf of $\phi(f(t))$.\end{lemma}

\begin{proof}  By definition $m_1(\mC) \subset S(f,a_1, \mM_{1_j}, W_{1_j})$, hence $a_1$ is a leaf of $\phi(f(t))$ at all $t\in m_1(\mC)$.  

Now assume the lemma is true for all subcascades of length $<u$.  Let $ t\in 
m_u(\mC)$.   Let $A = B^{u_j}_u$.  By definition, $\mM_{u_j}$ is not $f$-free of 
$A$ along $m_{u-1}(\mC)$, but $\mM_{u_j}$ is $f$-free of $A\setminus a_u$ 
along $m_{u-1}(\mC)$.  Since $m_u(\mC)\subset S(f,A, W_{u_j})\cap 
m_{u-1}(\mC)$ it follows that $a_u$ is a leaf of $\phi(f(t))$.\end{proof}

\begin{definition}  \label{packet} Let $\mC$ be an active cascade.  To each $t\in m_v(\mC)$ corresponds a $v$-tuple $(p_1, \cdots, p_v)$ 
where $p_j$ is the (possibly empty) stryker multi-geodesic  for $\tau_j$ at $t$.  Such a $(p_1, \cdots, 
p_v)$ is called a \emph{packet}. There are only finitely many packets, by the 
finiteness of stryker curves.  Thus $m_v(\mC)$ canonically decomposes into a disjoint union of 
closed sets $S_1, \cdots, S_r$ such that each point in a given $S_j$ has the same packet.  Let $\mP=\{P_1, \cdots, P_r\}$ denote the set of 
packets, the last entry in the definition of $\mC$.    We will use the notation $\mP(f)$, when needed to clarify 
the function on which this information is based.\end{definition}

\begin{definition}\label{cascade order}  In what follows all cascades use the same set of markers and simple closed geodesics, however the function $f:V\to \PMLEL$ will vary.  We put an equivalence relation on this set of cascades and  then partially order the classes.  We say that $\mC(g)$ is equivalent to $\mC(f)$ if $\mP(f)=\mP(g)$ and for all $r$,  $\mA_r(f)=\mA_r(g)$.  We  lexicographically partial order the equivalence classes by comparing the $v$-tuples $(\mA_1(\mC(f)), \cdots, \mA_v(\mC(f)), \mP(f))$ using the rule that $\mA_r(i,f)\le\mA_r(i,g)$ if $\mM_{r_j}(f)\le \mM_{r_j}(g)$, with $\infty$ being considered the maximal value and  $\mA_r(ii,f)\le\mA_r(ii,g)$ if their values satisfy that inequality and $\mA_r(iii,f)\le\mA_r(iii,g)$ if  $\Stryker_r(g)\subset \Stryker_r(f)$.  Finally, $\mP(f)\le \mP(g)$ if $\mP(g)\subset \mP(f)$.  \end{definition}

\begin{remark}  \label{cascade ordering}  More or less, $\mC(f)<\mC(g)$ means 
that the markers's are freer with respect to the function $g$ than  with respect to 
$f$.  In particular, this inequality holds if  $\mM_1$ and $\mM_2$ are $g$-free of 
$a_1$, but only $\mM_1$ and $f$-free of $a_1$.     If  $\mM_1$ is both $g$-free 
and $f$-free of $a_1$ but $\mM_2$ is neither $f$-free nor $g$-free of $a_1$, 
then $\mA_1(ii)$ is a measure of how close $\mM_2$ is from being free of 
$a_1$.   The bigger the number, the closer to freedom as motivated by Lemma 
\ref{tag freedom}.   If this number is the same with respect to both $f$ and $g$, 
then $\mA(iii)$ measures how much work is needed to raise the number.  More 
stryker multi-geodesics with respect to $f$, than $g$, means more needs to be 
done to $f$, so again $\mC(f)<\mC(g)$. If $\mA_1(f)=\mA_1(g)$, then  $\mA_2$ 
is used to determine the ordering.  Finally, if all the $\mA_i$'s are equal, then $f$ having more  packets means that \emph{more sets} of $V$ need 
to be cleaned up to finish the cascade.  \end{remark}

\begin{proposition}\label{preservation}  Let $V$ be the underlying space of a finite simplical complex.  Let $f:V\to \PMLEL$ a generic PL map, $\dim(V)\le n$ and $\mJ=(\mM_1,W_1), \cdots, (\mM_m, W_m)$ a marker family hit by $f$.  Let $\mC$ be an active cascade based on $\{a_1, \cdots, a_v\}$,  $C$ a simple closed geodesic 
disjoint from the $ a_i$'s and the $\tau_i$'s and $K\subset m_v(\mC)\cap \finv(B_C)$ compact.  For $1\le i\le m$, let $B^i_C=\{a_u|   i<u_j\}\cup\{C\}$.   Assume that   for $1\le i\le m$, $\mM_i$ is $f$-free of $B^i_C$ along  $K\cap W_i$.  If $\delta$ is sufficiently small and $f_1$ is obtained from $ f$ by a $(C, \delta, K)$ push off homotopy, then $[\mC(f)]\le [\mC(f_1)]$  and the homotopy from $f $ to $f_1$ is $\mJ$-marker preserving.    
\end{proposition}

\begin{proof}    Since $\mJ$ is f-free of $C$ along $K$,  it follows by Lemma \ref{partial marker preserving} that  for $\delta$ sufficiently small, any $(C, \delta, K)$ homotopy is $\mJ$ marker preserving.    It remains to show that if $\delta$ is sufficiently small and $f_1$ is the resulting map, then $[\mC(f)]\le [\mC(f_1)]$.  

From $\mC(f)$ we conclude that for all $u\in\{1, \cdots , v\}$ and $q<u_j$, $\mM_q$ is $f$-free of $B^q_u$ along $m_v(\mC)\cap W_q$.  By hypothesis each $\mM_q$ is also $f$-free of $B^q_C$ along $W_q\cap K$.  Note that $B^q_u\subset B^q_C$ when $q<u_j$.  By Lemma \ref{partial marker preserving},  if $\delta$ is sufficiently small, then there exists a neighborhood $V$ of $ K$ such that any $(C, \delta, K)$ homotopy is supported in $V$ and each $\mM_q$ is $F$-free of $B^q_C$ along $W_q\cap \bar V$.  It follows that with respect to the lexicographical ordering $(\mA_1(i,f), \cdots, \mA_v(i,f))\le (\mA_1(i,f_1), \cdots, \mA_v(i,f_1))$.  

A similar argument using Lemma \ref{partial stryker preserving} shows that if $\delta$ is sufficiently small and $\mA_i(i,f)=\mA_i(i,f_1)$ for $1\le i\le u$, then $\mA_i(ii,f)\le \mA_i(ii,f_1)$ for $1\le i\le u$.  

By Lemma \ref{tag stryker finiteness} there exists  $\eta>0$ such that if $d_{V}
(t_1, t_2)<2 \eta$  and $t_1, t_2\in m_u(\mC)$ for some $u$, then $f(t_1), 
f(t_2)$ have the same stryker curve to $
\tau_u$.  Let $\delta$ be sufficiently small to satisfy Lemma \ref{partial tag stryker preserving} with this $\eta$ in addition to the 
previously required conditions.   That lemma implies that if  $\mA_i(i,f)=\mA_i(i,f_1)$ and $\mA_i(ii,f)=
\mA_i(ii,f_1)$ for all $i\le u$, then $\Stryker_i(f_1)\subset \Stryker_i(f)$ for all $i\le 
u$.  

Finally, Lemma \ref{partial tag stryker preserving} with this choice of $\eta$ also implies that if $(\mA_1(f), \cdots, \mA_v(f))=(\mA_1(f_1), \cdots, \mA_v(f_1))$, then $\mP(f_1)\subset \mP(f)$. It follows that $[\mC(f)]\le [\mC(f_1)]$.\end{proof}

\begin{remark}  Note that $[\mC(f)]\le [\mC(f_1)]$ holds with respect to the lexicographical ordering, but may not hold entry-wise, since there may be no direct comparison between later entries once earlier ones differ.  For example, say $2=\mA_1(i,f)<\mA_1(i,f_1)=3$, then showing $\mA_2(i,f)=3$ involves verifying that $\mM_2$ is $f$-free of $\{a_2\}$ while showing $\mA_2(i,f_1)=3$ involves verifying that $\mM_2$ is $f_1$-free of $\{a_1, a_2\}$.\end{remark}

 \section{Finishing Cascades}
 
The main result of this section is the following.
 
\begin{proposition}\label{finished}  Let $V$ be the underlying space of a finite simplicial complex$h:V\to \PMLEL$ be a generic PL map such that $k=\dim(V)\le n
$ where dim$(\PML)=2n+1$.  Let $\mJ$ be a marker family hit by $h$ and $\mC$ be an active 
cascade. Then there exists a marker preserving homotopy of $h$ to $h'$ such that $\mC$ is 
finished with respect to $h'$ and $[\mC(h)]< [\mC(h')]$.    The homotopy is a concatenation of 
relative push offs.   If $L\subset \PML$ is a finite subcomplex of $\CS$, then the homotopy can be chosen to be disjoint from $L$.    \end{proposition}

\begin{lemma}  \label{push finiteness} Let  $f_i:V\to \PMLEL$, $i\in \BN$ be 
generic PL maps, $\mJ$ be a marker family and let  $\{\mC(f_i)\}$ be active 
cascades based on the same set of simple closed geodesics.  Any sequence 
$[\mC(f_1)]\le [\mC(f_2)]\le \cdots$ has only finitely many terms that are strict 
inequalities.\end{lemma} 

\begin{proof}  There are only finitely many possible values for $\mA_u(i,f_j)$.

It follows from Lemma \ref{tag freedom} that each $\mA_u(ii,f_j)$ is uniformly bounded above.  Indeed, there are only finitely many $a_i$'s and only finitely many markers in $\mJ$, thus only finitely many marker tags arising from them.  If $\tau_u$ is such a tag, then for $i\in \BN$, $g(f_i, \tau, m_u(\mC(f_i)))\le |(\cup_{j=1}^v a_j)\cap \tau_u| + 3(3g-3+2p)$.

The finiteness of stryker curves (Lemma 4.5) shows that the number of possibilities for both $\mP(f_j)$ and each $\mA_u(iii, f_j)$  is bounded.\end{proof}



\noindent\emph{Proof of Proposition \ref{finished}}.  It  suffices to show that given any active cascade $\mC(h_1)$, there exists a $\mJ$ marker preserving homotopy from $h_1$ to $h_2$, that is a concatenation of relative push offs,  such that $[\mC(h_1)]<[\mC(h_2)]$.  For if $\mC(h_2)$ is not finished, then we can similarly produce an $h_3$ with $[\mC(h_2)]<[\mC(h_3)]$.  Eventually we obtain a finished $\mC(h_q)$ else we contradict Lemma \ref{push finiteness}.

We retain the convention 
that dim$(\PML)=2n+1$.  We will assume that $k=n $ leaving the easier $k<n$ 
case to the reader.  The proof is by downward induction on $L=\length(\mC)$.    
Suppose that $\mC(h)$ is an active cascade based on $\{a_1, \cdots, a_L\}$.  
Since the $a_i$'s are pairwise disjoint, it follows that $L\le n+1$.   Let $\tau_i$ 
denote the marker tag associated to $a_i$ and $\mM_{i_j}(h)$.  Let $R' = 
N(a_1\cup\cdots\cup a_L\cup \tau_1\cdots\cup \tau_L)$.  Let $R$ be $R' $ 
together with all its complementary discs, annuli and pants.  Let $T=S
\setminus(\inte(R))$.  Note that $T=\emptyset$ if $L\ge n$.   

Let $ \mA(h(t))$ denote the arational sublamination of $\phi(h(t))$, i.e. the sublamination obtained by deleting all the compact leaves.   Since $h$ is generic and $k= n$, it follows that $\mA(h(t))\neq\emptyset$ for all $t\in B^n$.   The key observation is that if $\mA(h(t))\cap R\neq \emptyset$, then  $t\notin m_L(\mC(h))$.  To start with, Lemma \ref{leaf presence} implies that $a_1, \cdots, a_L$ are leaves of $\phi(h(t))$.  This implies that $\mA(h(t))\cap \tau_i \neq\emptyset$, where $\tau_i $ is  the marker tag  between $a_i$ and $\mM_{i_j}$.  By Lemma \ref{tag freedom} $\phi(h(t))$ hits $\mM_{i_j}$ and hence $t\notin m_i(\mC(h)) $contradicting the fact that $m_L(\mC(h))\subset m_i(\mC(h))$. It follows that $\mC(h)$ is finished if $L\ge n$.

Note that if $k<n$, then this argument shows that $\mC$ is finished if $L\ge k$.  In particular, if $k=1$, (the path connectivity case) all cascades are finished.

Now assume that the Proposition is true for all cascades of length greater than $L$, where $L<n$.  Let $\mC$ be an active cascade of length $L$.

Let $P=(p_1, \cdots, p_L)\in \mP(h)$ and $S_P$ the  closed subset of $m_L(\mC)$ consisting of points whose packet is $P$.  Let $\sigma$ be the possibly empty multi-geodesic $p_1\cup\cdots\cup p_L$.  By definition, if $t\in S_P$, then each $p_i$ is a leaf of $\phi(h(t))$ and by Lemma \ref{leaf presence} each of $a_1, \cdots, a_L$ is a leaf of $\phi(h(t))$.    Let $Y'=N(a_1\cup\cdots\cup a_L\cup \tau_1\cdots\cup \tau_L\cup \sigma)$ and $Y$ be $Y'$ together with all complementary discs, annuli and pants.  Let $Z=S\setminus\inte(Y)$.  Note that $Z\neq\emptyset$, else for all $t\in S_P$, $|\phi(h(t))\cap\tau_{i(t)}| =\infty$ for some $i(t)$ which is a contradiction as before.  Let $C$ be a simple closed geodesic isotopic to some component of $\partial Z$.

Observe that $ C$ is neither an $a_i$ nor is $C\subset \sigma$.  Indeed, since each $a_i$ is 
crossed transversely by a tag at an interior point it follows that no $a_i$ is isotopic to a component of $\partial Y'$ 
and hence $\partial Y$.  Similarly, each component of $\Stryker_i$ is transverse to $
\tau_i$ at an interior point, hence no component of $\sigma$ is  isotopic to a component of $\partial Y$.  Next observe that $S_P\subset \finv(B_C)$.  Indeed, $t\in S_P$ implies that  for all $i, \phi(h(t))\cap \tau_i\subset a_i\cup \sigma$, hence $C$ is either a leaf of $\phi(h(t))$ or $C\cap \phi(h(t))=\emptyset$.

 Extend the cascade $\mC(h)$ to the length $L+1$ cascade $\mC'(h)$  by letting 
 $C$ be our added $a_{L+1}$.  If  $\mC'(h)$ is active, then by induction, there is 
 a marker preserving homotopy of $h$ to $h_1$ so that $\mC'(h_1)$ is finished and $[\mC'(h)]<[\mC'(h_1)]$.  If $\mC'(h)$ is finished, then let's unify 
 notation by denoting $h$ by $h_1$.  In both cases, by restricting to the length $L$ subcascade either $[\mC(h)]<[\mC(h_1)]
 $ or $[\mC(h)]=[\mC(h_1)]$ and each $\mM_i$ is $h_1$-free of $B^i_C$ for all $t\in m_L(\mC(h_1))\cap W_i$, where $B^i_C$ is as in the statement of 
 Lemma \ref{preservation}.    Since freedom is an open condition and $S_P(h_1)$ is 
 compact as are all the $W_i$'s, there exists $U\subset V$ open such that 
 $m_L(\mC(h_1))\subset U$ and each $\mM_i$ is $h_1$-free of $B^i_C$ 
 for all $t\in \bar U\cap W_i$.

If $[\mC(h)]=[\mC(h_1)]$, then invoke Lemma  \ref{preservation} by taking $f = h_1$,  keeping the original $\mJ$ and $\mC$, using the above constructed $C$ and letting $K=S_P(h_1)$.  Choose $\delta$ sufficiently small to satisfy the conclusion of Lemma  \ref{preservation} and so that any $(C,\delta, K)$ push off homotopy is supported in $U$.  

To complete the proof it suffices to show that if $\delta$ is sufficiently small and $h_2$ is the resulting map, then $[\mC(h_1)]<[\mC(h_2)]$.   Since Lemma  \ref{preservation} implies that $[\mC(h_1)]\le[\mC(h_2)]$ it suffices to show that $P\notin\mP(h_2)$ if for all $i$, $\mA_i(h_1)=\mA_i(h_2)$.    
If $P\in \mP(h_2)$, then let $t\in S_P(h_2)$.  As above $h_2(t)\in B_C$.  Since 
$h_2$ is the result of a relative push off homotopy it follows from the last 
sentence of Definition \ref{partial pushing} iii) that either $\phi(h_2(t))=\phi(h_1(t))$ 
or $\phi(h_2(t))\neq \phi(h_1(t))$ but $C\cup \phi(h_2(t))=\phi(h_1(t))$.  In the 
former case, $t\in S_P(h_1)$ and hence both $h_1(t), h_2(t)\in B_C$, 
contradicting the fact that $h_2$ is the result of a $(C, \delta, K)$ push off 
homotopy.  In the latter case $t\in U\setminus S_P(h_1)$, hence there exists an $
\mM_i$ that is $h_1$-free of $B^i_C\setminus C$ at $t\in W_i$, but $\mM_i$ is 
not $h_2$-free of $B^i_C\setminus C$ at $t$.  This implies that $\mM_i$ is not 
$h_1$-free of $B^i_C$ at $t$ which contradicts the fact that $t\in U$.

Since $k\le n$ Lemma \ref{missing} it follows that $h(V)\cap K=\emptyset$.  By Lemma \ref{avoiding}, by using sufficiently small $\delta$'s all $(C, \delta, K)$ pushoff homotopies in the above proof could have been done to avoid $L$.  \qed

\begin{corollary}  \label{free}Let $V$ be the underlying space of a finite simplicial complex.  Let $S$ be a finite type surface and $f:V\to \PMLEL$ be a generic PL map and  $\dim(V)\le n$, where $\dim(\PML)=2n+1$.  Let $\mJ$ be a marker family hit by $f$ and $C$ a simple closed geodesic.  Then $f$ is homotopic to $f_1$ via a marker preserving homotopy such that $g(f,C)>0$.  If $L\subset \PML$ is a finite subcomplex of $\CS$, then the homotopy can be chosen to be disjoint from $L$.\end{corollary}

\begin{proof}  By Lemma \ref{marker preserving}, if $\mJ$ is free of $C$ and $\delta$ is sufficiently small, then any $(C,\delta)$ push off homotopy is $\mJ$-marker preserving.  If $f_1$ is the resulting map, then $ g(f_1,C)>0$.

If $\mJ$ is not free of $C$, then let  $\mC$ be the active length-1 cascade based on $C$.  By Proposition \ref{finished} $f$ is homotopic to $f'$ by a marker preserving homotopy such that $\mC(f')$ is finished and the homotopy is a concatenation of relative push off homotopies.   Now argue as in the first paragraph. \end{proof}

\section{Stryker Cascades}

The main result of this section is 

\begin{proposition}  \label{increasing} Let $V$ be the underlying space of a finite simplicial complex.  Let $f:V\to \PMLEL$ be a generic PL map such that $\dim(V)\le n$ where $\dim(\PML)=2n+1$.  If $a_1$ is a simple closed geodesic such that $\infty>g(f,a_1)>0 $ and $\mJ$ a marker family hit by $f$, then there exists then there exists a marker preserving homotopy of $f$ to $f_1$ such that $g(f_1,a_1)>g(f,a_1).$  The homotopy is a concatenation of relative push offs.  If $L\subset \PML$ is a finite subcomplex of $\CS$, then the homotopy can be chosen to be disjoint from $L$.\end{proposition}

\begin{remarks}  The proof is very similar to that of Corollary \ref{free}, except that \emph{stryker cascades} are used in place of marker cascades.   A stryker cascade is essentially a marker cascade except that the first term is a curve $a_1$ with $g(f,a_1)>0$.   

Closely following the previous two sections, we give the definition of stryker cascade and prove various results about them.  \end{remarks}

\begin{definition}  \label{cascade stryker} Associated to  $f:V\to \PMLEL$ a 
generic PL map, $\mJ=(\mM_1,W_1), \cdots, (\mM_m, W_m)$ a marker family 
hit by $f$, $\mM_1<\cdots < \mM_m$ the ordering induced from this 
enumeration and $a_1, \cdots, a_v$ a sequence of pairwise disjoint simple 
closed geodesics with $g(f,a_1)>0$ we define a \emph{stryker cascade $\mC$}  
which is a $v+1$-tuple  $ (\mA_1, \cdots,  \mA_v, \mP)$ where $\mA_1$ is a 2-tuple $(\mA_1(i), \mA_1(ii))$ and for $i\ge 2$, each $\mA_i$ is a 3-tuple $
(\mA_i(i), \mA_i(ii), \mA_i(iii))$ and  $\mP$ is a finite set of $v$-tuples defined in 
Definition \ref{packet stryker}. 
\vskip 10pt

Define $\mA_1=(g(f,a_1), \Stryker_1)$  where $\Stryker_1$ is the set of stryker curves to $f$ and $a_1$.\vskip10pt

Define $m_1(\mC)=\{t\in V|g(f,a_1)=|\phi(f(t))\cap a_1|\}$.
\vskip10pt

Having defined $\mA_i, 1\le i< u$, then $\mA_u$ is defined as follows.    To start with define $B^1_u, \cdots, B^m_u$ where $B^r_u=\{b^r_2, \cdots, b^r_u\}$, where $b^r_u=a_u$ and for $2\ge q<u, b^r_q=a_q$ if $r<q_j$ and $b^r_q=\emptyset$ otherwise.  
\vskip10pt

\noindent $\mA_u(i)$ is defined to be either the maximal marker $\mM_{u_j}$, such that $r<u_j$ implies that $\mM_r$ is free of $B^r_u$ along $ m_{u-1}(\mC)\cap W_r$ or $\mA_u(i)=\infty$ if for all $r\le m$, $\mM_r$ is free of $B^r_u$ along $ m_{u-1}(\mC)\cap W_r$.

If $\mM_{u_j}$ exists, then define $\tau_u$ to be the marker tag arising from $a_u
$ and $\mM_{u_j}$.  Let $S_u=S(f,B^{u_j}_u, \mM_{u_j}, W_{u_j})\cap m_{u-1}
(\mC)$.  Define $m_u(\mC)= \{t\in S_u| g(f,\tau_u, S_u)=|\phi(f_t)\setminus a_u)
\cap \tau_u|\}$.

\vskip10pt
\noindent$\mA_u(ii)$ is defined to be either $g(f,\tau_u, m_u(\mC))$ or $\infty$ if 
$m_u(\mC)=
\emptyset$.
\vskip8pt
\noindent$\mA_u(iii)$ is defined to be 
the set $\Stryker_u$ which is either the set of $m_u(\mC))$-stryker curves for $\tau_u$ if $m_u(\mC)\neq\emptyset$ or  $\infty$ otherwise.

\vskip10pt

We say that the cascade $\mC$ is \emph{finished} if $m_v(\mC)=\emptyset$ and \emph{active} otherwise.  
We say that the cascade is \emph{based} on $\{a_1, \cdots, a_v\}$ and has \emph{length} $v$.   For $r\le v$, then length-q cascade  based on $\{a_1,\cdots, a_r\} $ is called the length-$r$ \emph{subcascade} and denoted $\mC_r$.  Note that $\mC_r$ and $\mC$ have the same values of $\mA_1, \cdots, \mA_r$.\end{definition}

\begin{notation} The data corresponding to a cascade depends on $f$.  When the function must to be explicitly stated, we will use notation such as $\mC(f), m_i(\mC, f), \mA_p(f)$ or $\mA_r(ii,f)$.   \end{notation}

We record for later use the following result whose proof is essentially identical to that of Lemma \ref{leaf presence}.

\begin{lemma} \label{leaf presence stryker} Let $\mJ$ be a marker family hit by the generic PL map $f:V\to \PMLEL$.  If $\mC$ is an active  stryker cascade based on $a_1, \cdots, a_v$, then for every $t\in m_v(\mC)$ and 
$i\ge 2$, each $a_i$ is a leaf of $\phi(f(t))$.\qed\end{lemma}

\begin{definition}  \label{packet stryker}  Let $\mC$ be an active stryker cascade.  To each $t\in m_v(\mC)$ corresponds an $v$-tuple $(p_1, \cdots, p_v)$ 
where $p_1$ is the stryker multi-geodesic  for $a_1$ at $t$ and for $i\ge 2$, $p_i$ is the (possibly empty) stryker multi-geodesic for $\tau_i$ at $t$.  Such a $(p_1, \cdots, 
p_v)$ is called a \emph{packet}. There are only finitely many packets, by the 
finiteness of stryker curves.  Thus $m_v(\mC)$ canonically decomposes into a disjoint union of 
closed sets $S_1, \cdots, S_r$ such that each point in a given $S_j$ has the same packet.  Let $\mP=\{P_1, \cdots, P_r\}$ denote the set of 
packets, the last entry in the definition of $\mC$.   \end{definition}

\begin{definition}\label{cascade order stryker}  Use the direct analogy of Definition \ref{cascade order} to put an equivalence relation on the set of stryker cascades based on the same ordered set of simple closed curves and to partially order the classes.  \end{definition}

\begin{proposition}\label{stryker preservation}  Let $f:V\to \PMLEL$ a generic 
PL map, $\dim(V)\le n$ and $\mJ=(\mM_1,W_1), \cdots, (\mM_m, W_m)$ a marker family $\mJ
$ hit by $f$.  Let $\mC$ be an active stryker cascade based on $\{a_1, \cdots, 
a_q\},  C$ a simple closed geodesic 
disjoint from the $a_i$'s and $\tau_j$'s and let $K\subset m_v\cap \finv(B_C)$.  
For $1\le i\le m$, let $B^i_C=\{a_p, p\ge 2|   i<p_j\}\cup\{C\}$.   Assume that   for 
$1\le i\le m$,  $\mM_i$ is $f$-free of $B^i_C$ along $K\cap W_i$.  If $f_1$ is 
obtained from $f$ by a $(C, \delta, K)$ push off homotopy and 
$\delta$ is sufficiently small, then $[\mC(f)]\le [\mC(f_1)] $ and the homotopy 
from $f$ to $f_1$ is $\mJ$-marker preserving.    
\end{proposition}

\begin{proof}    The fact $\mA_1(f)\le \mA_1(f_1)$ follows from Lemma \ref{partial stryker preserving}.  The rest of the argument follows as in the proof of Proposition \ref{preservation}.  \end{proof}  

We have the following analogue of Lemma \ref{push finiteness}.

\begin{lemma}  \label{push finiteness stryker} Let  $f_i:V\to \PMLEL$, $i\in \BN$ be 
generic PL maps, $\mJ$ a marker family and let  $\{\mC(f_i)\}$ be active 
stryker cascades based on the same set of simple closed geodesics.  Any sequence 
$[\mC(f_1)]\le [\mC(f_2)]\le \cdots$ with $g(f_1,a_1)=g(f_2, a_1)=\cdots$ has only finitely many terms that are strict 
inequalities.\qed\end{lemma}

\begin{proposition}\label{stryker finished}  Let $h:V\to \PMLEL$ be a generic PL map such that $k=\dim(V)\le n$ where $\dim(\PML)=2n+1$.  Let $\mJ$ be a marker family hit by $h$ and let $\mC$ be an active stryker cascade with $\infty > g(h,a_1)>0$. Then there exists a marker preserving homotopy of $h$ to $h'$ such that $[\mC(h)]< [\mC(h')]$ and either $g(h',a_1)>g(h,a_1)$ or $\mC(h')$ is finished.  The homotopy is a concatenation of relative push offs.  If $L\subset \PML$ is a finite subcomplex of $\CS$, then the homotopy can be chosen to be disjoint from $L$.  \end{proposition}

\begin{proof}  As in the proof of Proposition \ref{finished} it suffices to show that given any active stryker cascade $\mC(h_1)$ there exists a marker preserving homotopy of $h_1$ to $h_2$, that is a concatenation of relative push offs, such that $[\mC(h_1)]<[\mC(h_2)]$.  Here we invoke Lemma \ref{push finiteness stryker} instead of Lemma \ref{push finiteness}.  

Again we discuss the k=n case, the easier $k<n$ cases being left to the reader.  The proof by downward induction on the length of the cascade follows essentially exactly that of Proposition \ref{finished} until the step of proving it for length-1 cascades.  So now assume that the Proposition has been proved for cascades of length $\ge 2$.

Let $\mC(h)$ be a length-1 cascade based on $a_1$.  Here $\mP$, the set of packets, consists of the set of stryker multi-curves to $h$ and $a_1$.  Let $P$ be one such multi-curve.  Let $Y'=N(a_1\cup P)$ and $Y$ be the union of $Y'$ and all components of $S\setminus Y'$ that are discs, annuli and pants. Again, genericity and the condition $k=n$ implies that $Y \neq S$.  Let $C$ be a simple closed geodesic isotopic to a component of $\partial Y$.  The condition $g(h,a_1)>0$ implies that $C$ is not isotopic to $a_1$.   Since each component of $P$ non trivially intersects $a_1$, $C$ is not isotopic to any component of $P$.  Let $\mC'$ denote the length-2 stryker cascade based on $\{a_2, C\}$. By induction there exists a marker preserving homotopy of $h$ to $h_1$ that is a concatenation of relative push offs such that $[\mC'(h)]<[\mC'(h_1)]$ and either $g(h_1,a_1)>g(h,a_1)$ or $\mC'(h_1)$ is finished.  In the latter case we have either $ [\mC(h)]< [\mC(h_1)]$ or $[\mC(h)]=[\mC(h_1)]$ and for each $1\le i\le m$, $\mM_i$ is free of $C$ along $m_1(\mC, h_1)$.  Now argue as in the proof of Proposition \ref{finished} that if $h_2$ is the result of a $(C, \delta, K)$ homotopy, where $\delta$ is sufficiently small and $K=S_P$, then either $g(h_2,a_1)>g(h_1, a_1)$ or $g(h_2, a_1)=g(h_1, a_1)$ and  $[\mC(h_1)]<[\mC(h_2)]$.\end{proof}

\noindent\emph{Proof of Proposition \ref{increasing}.}  Let $\mC(f)$ be the length-1 stryker cascade based on $a_1$.  By Proposition \ref{stryker finished} there exists a marker preserving homotopy from $f$ to $f_1$ that is the concatenation of relative push off homotopies such that $\mC(f)<\mC(f_1)$.  Therefore, either $g(f_1,a_1)>g(f,a)$ or equality holds and $|\mP(\mC(f_1))|<|\mP(\mC(f))|$.  After finitely many such homotopies we obtain $f'$ such that $g(f',a_1)>g(f_1,a)$.\qed

\section{$\EL$ is (n-1)-connected}  

\begin{theorem}  \label{marker theorem} Let $V$ be the underlying space of a finite simplicial complex and $W$ that of a subcomplex.  Let $S$ be a finite type hyperbolic surface with $\dim(\PML)=2n+1$.  If $f:V\to \PMLEL$ is a generic PL map, $\finv(\EL)=W$, $\dim(V)\le n$ and $\mJ$ is a marker family hit by $f$, then there exists a map $g:V\to \EL$ such that $g|W=f$ and $g$ hits $\mJ$.  The map $g$ is the  concatenation of (possibly infinitely many) relative push offs.   If $L\subset \PML$ is a finite subcomplex of $\CS$, then the homotopy can be chosen to be disjoint from $L$.\end{theorem}

\begin{proof} Let $C_1, C_2, \cdots$ be  an enumeration of the simple closed geodesics on $S$.  
It  suffices to find a sequence, $f_0, f_1, f_2, \cdots$ of extensions of $f|W$ and  sequences 
$\{\mE_i\}$, $\{\mS_i\}$ of marker covers, such that $\mE_i$ is a marker covering of $V$ by $1/i
$ markers, $\mS_i$ is a marker covering of $V$ by $ C_i$ markers and for $i\le j, f_j$ hits each 
$\mE_i$ and $\mS_i$ family of markers.  Furthermore, $f_{i+1}$ is obtained from $f_i$ by a finite 
sequence of relative push offs.  Then Proposition \ref{continuity criterion} produces $g$ as the 
limit of the $f_j$'s.   Since $f_{i+1}$ is obtained from $f_i$ by concatenating finitely many push offs, concatenating all these push off homotopies  produces a 
map $F:V\times [0, \infty]\to\PMLEL$ with $F|V\times \infty=g$.  The proof of Proposition \ref
{continuity criterion} shows that $F$ is continuous.  It follows from Lemma \ref{avoiding} that the 
homotopy can be chosen disjoint from $L$.

Suppose that $f_0, f_1, \cdots, f_{j-1}, \mE_1, \cdots, \mE_{j-1}$, and $\mS_1, \cdots, \mS_{j-1}$ have been constructed so that $f_q$ hits each $\mE_p$ and $\mS_p$ family of markers whenever $p\le q\le j-1$.  We will extend the sequence by constructing 
$f_j$, $\mE_j$ and $ \mS_j$ to satisfy the corresponding properties.  The theorem then follows by induction.   In what follows, $\mJ_i$ denotes the marker family that is the union  of all the markers in the $\mE_r$ and $\mS_s$ marker families, where $r,s\le i$.  \vskip 10pt

\noindent \emph{Step 1}.  Construction of $f_j$ and $\mS_j$:

\begin{proof} Let $N(C_j)$ be as in Lemma \ref{finding markers} ii).  Given $f_{j-1}$, obtain $f^1_j$ by applying Corollary \ref{free} to $f_j$ and the marker cover $\mJ_{j-1}$.  Now repeatedly apply Proposition \ref{increasing} to obtain $f^2_j, f^3_j, \cdots, f^{N(C_j)}_j$ with the property that $f^q_j $ is a generic PL map obtained from $f^{q-1}_{j-1}$ via a $\mJ_{j-1}$ marker preserving homotopy such that $g(f^q_j, C)\ge q$.  Let $f_j=f^{N(C_j)}_j$ and apply Lemma \ref{finding markers} ii) to obtain $\mS_j$. \end{proof}

\noindent \emph{Step 2}.  Construction of  $\mE_j$:   
\begin{proof}Apply Lemma \ref{finding markers} i) to the 
generic PL map $f_j$ to obtain  $\mE_j$ a marker cover by $1/j$ markers.\end{proof}

\begin{theorem} \label{connectivity} If $S$ is a finite type hyperbolic surface such that $\dim(\PML)=2n+1$, then $\EL$ is $(n-1)$-connected.\end{theorem}

\noindent\emph{Proof.}   Let $k\le n$ and $g:S^{k-1}\to \EL$ be 
continuous and $f:B^k\to \PMLEL$ be a generic PL extension of $g$, provided by Proposition \ref{extension}.  Now apply Theorem \ref{marker theorem} to extend $g$ to a map of $B^k$ into $\EL$.\end{proof}

\section{$\EL$ is $(n-1)$-locally connected}  

\begin{theorem}  \label{local connectivity}   If $S$ is a finite type hyperbolic surface, then $\EL$ is $(n-1)$-locally connected, where $\dim(\PML)=2n+1$.\end{theorem}

\begin{proof}  We need to show that if $\mL\in \EL$, and $U'\subset \EL$ is an open set 
containing $\mL$, then there exists an open set $U\subset U'$ containing $\mL$ so that if $k\le n$, 
$g:S^{k-1}\to \EL  $ and $g(S^{k-1})\subset U$, then $g$ extends to a map $F:B^k\to \EL$ such 
that $F(B^k)\subset U'$.   

Choose a pair of pants decomposition and parametrizations of the the cuffs and pants so that some complete maximal train track  $\tau$  fully carries $\mL$.  Thus $V(\tau)$ is a convex subset of $\ML$ that contains $ \hatphiinv(\mL)$ in its interior.  

Let $\mM_1, \mM_2, \cdots$ be a sequence of  markers such that each $\mM_i$ is a $1/i$-marker that is that is hit by $\mL$.  Let $U_i=\{x\in \EL|
\phi(x)$ hits $\mM_i\}$.    Then each $U_i$ is an open set containing $\mL$ and 
by Lemma \ref{marker closeness} there exists $N\in \BN$ such that if $i\ge N$,  
then $ \hatphiinv (U_i)\subset 
\inte(V(\tau))\cap \hatphiinv(U')$.    Reduce the ends of each post of $\mM_N$ slightly to obtain 
$\mM^*_N$ so that if  $U^*_N=\{x\in \EL|\phi(x)$ hits $\mM^*_N\}$ and $\bar U^*_N=\{x\in \EL|\phi(x)$ hits $\bar
\mM^*_N\}$, then $\mL\in U^*_N\subset\bar U^*_N\subset U_N$.  Let $\hat W \subset \hatphiinv(U^*_N)$ be 
an open convex subset of $\ML$ that contains $\hatphiinv(\mL)$ and is saturated by rays through the 
origin.  By reducing $\hat W$ we can assume that $\hat W=\inte(V(\tau))$ for some train track $\tau$ carrying $\mL$. Next choose  $j$ such that $\hatphiinv(U_j)\subset \hat W$.  Let $U=U_j$ and $\mM=\mM_j$.  

Let $k\le n$ and $g:S^{k-1}\to \EL$ with $g(\Sk)\subset U$.  Since $\hat W$ is  convex and contains $\hatphiinv(U)$, we can apply Proposition \ref{extension} and Remark \ref{convexity} to find a generic PL map $f_0:B^k\to \PMLEL $ extending $g$ 
such that $f_0(\inte(B^k))\subset W$.  Thus for all $t\in B^k, f_0(t)$ hits $\mM$.  Theorem 
\ref{marker theorem} produces  $f:B^k\to \EL$ extending $g$ such 
that for $t\in B^k$ and $i\in \BN, f_i(t)$ hits $\mM$.  Therefore, for each $t\in B^k, f(t)$ hits the closed 
marker $\bar M^*_N$ and hence $f(B^k)\subset \bar U_N\subset U'$.  \end{proof}

\section{$\PML$ and $\EL$ Approximation Lemmas}  

The main technical results of this paper are that under appropriate circumstances any map into $\EL$ can be closely approximated by a map into $\PML$ and vice versa.  In this section we isolate out these results.

We first give a mild extension  of Lemma \ref{pml approximation}, which is about approximating a map into $\EL$ by a map into $\PML$.  The subsequent two results are about approximating maps into $\PML$ by maps into $\EL$.

\begin{lemma}  \label{pml approximation two} Let $K$ be a finite simplicial complex,  $g:K\to \EL$  and $\epsilon>0$.  Then there exists a generic PL map $h:K\to \PML$ such that for each $t\in K$, $\dpts(\phi(h(t)), g(t))<\epsilon$ and $d_{\PML}(h(t), \phiinv(g(t')))<\epsilon$ some $t'\in K$ with $d_K(t, t')<\epsilon$. \end{lemma}

\begin{proof}  Given $g$, Lemma \ref{pml approximation} produces mappings $h_i:K\to \PML$ such that $\dpts(\phi(h_i(t)), g(t))<\epsilon/i$.  By super convergence  we can assume that each $h_i$ is a generic PL map satisfying the same conclusion. 

We show that if $i$ is sufficiently large, then the second conclusion holds for $h=h_i$.  Otherwise, there exists a sequence $t_1, t_2, \cdots$ of points in $K$  for which it fails respectively for $h_1, h_2, \cdots$.   Since for each $i$, $h_i^{-1}(\FPML)$ is dense in $K$, it follows from Lemma \ref{convergent sequence} that we can replace the $t_i$'s by another such sequence satisfying $d_{\PML}(h_i(t_i), \phiinv(g(t_i)))>\epsilon/2$,  $t_i\to t_\infty$, $h_i(t_i)\to x_\infty$ and $\phi(h_i(t_i))\in \EL$ all $i\in \BN$.   By Lemmas \ref{two closeness} and \ref{one closeness}, $\phi(h_i(t_i))\to g(t_\infty)$ and hence by Lemma \ref{convergent sequence}, $\phi(x_\infty)=g(t_\infty)$.  Therefore, $\lim_{i\to \infty} d_{\PML}(h_i(t_i), \phiinv(g(t_\infty)))\le  \lim_{i\to \infty} d_{\PML}(h_i(t_i), x_\infty) =0$, a contradiction.\end{proof}

\begin{lemma}  \label{weak el approximation}  Let $V$ be the underlying space of a finite simplicial complex.  Let $W$ be that of a finite subcomplex.  $S$ be a finite type hyperbolic surface with $\dim(\PML)=2n+1$ and let $\epsilon>0$.  If $h:V\to \PMLEL$ is a generic PL map, $\hinv(\EL)=W$ and $\dim(V)\le n$, then there exists $g:V\to \EL$ such that $g|W=h$ and for each $t\in V$, $d_{PT(S)}(g(t), \phi(h(t)))<\epsilon$.\end{lemma}

\begin{proof}  Since $k\le n$, $\mA(\phi(h(t)))\neq\emptyset$ for every $t\in V$, where $\mA(\phi(h(t)))$ denotes the arational sublamination of $\phi(h(t))$.  Therefore, by Lemma \ref{markers exist} for each $t\in V$, there exists an $ \epsilon$-marker $\mM_t$ that is hit by $\phi(h(t))$.  By Lemma \ref{markers open}, $\mM_t $ is hit by all $ \phi(h(s))$ for all $s$ sufficiently close to $t$.  By compactness of $V$, there exists an $\epsilon$-marker family $\mJ$ hit by $h$.  By Theorem \ref{marker theorem} there exists $g:V\to \EL$ such that $g$ hits $\mJ$.  Since both $h$ and $g$ hit $\mJ$, the conclusion follows.\end{proof}   

\begin{lemma} \label{el approximation}  $S$ be a finite type hyperbolic surface with $\dim(\PML)=2n+1$.   Let $V$ be the underlying space of a finite simplicial complex with $\dim(V)\le n$.  For $i\in \BN$, let $\epsilon_i>0$  and 
$h_i:V\to \PML$ be generic PL maps.   If for each $z\in \EL$ there exists $
\delta_z>0$ so that for $i$ sufficiently large 
$d_{\PML}(h_i(V),\phiinv(z))>\delta_z$, then    there exists $g_1, 
g_2, \cdots :V\to \EL$ such that 

\smallskip  i)  For each $i\in \BN$ and $t\in V, d_{PT(S)}(g_i(t), \phi(h_i(t)))<\epsilon_i$.

\smallskip  ii)  For each $z\in \EL$, there exists a neighborhood $U_z\subset \EL$ of $z$ such that for $i$ sufficiently large $g_i(V)\cap U_z=\emptyset$.\end{lemma}
   
\begin{proof}  Given $z\in \EL$ and $\delta_z>0$, it follows from Lemma \ref{pml one closeness} 
that there exists $\epsilon(z)>0$ such that if $x\in \PML$ and $\dpts(\phi(x),z')\le 
\epsilon(z)$, then $d_{\PML}(x, \phiinv(z))<\delta_z$ and if $y\in \EL$ and  $
\dpts(y,z')\le \epsilon(z)$, then $\phiinv(y)\subset N_{\PML}(\phiinv(z), \delta_z)$.    Here $z'$ is a diagonal extension of $z$.
Let $U_z$ be a neighborhood of $z$ such that $y\in U_z$ implies $y\subset 
N_{PT(S)}(z', \epsilon(z)/2)$ for some diagonal extension $z'$ of $z$.

Now let $f:V\to \PML$ be a generic PL map with $k\le n$.  If $0<\epsilon\le 
\epsilon(z)$ and $d_{\PML}(f(V)$, $\phiinv(z))>\delta_z$, then we now show 
that there exists $g:V\to \EL$ such that for every $ t\in V$, $\dpts(g(t)$, $
\phi(f(t)))<\epsilon$ and $g(V)\cap U_z=\emptyset$.  By definition of $
\epsilon(z)$, for every $t\in V$ and diagonal extension $z'$  of $z$, $
\dpts(\phi(f(t)), z')>\epsilon(z)$ and in  particular $\dpts(A(t),z')>\epsilon$, where 
$A(t)$ is the arational sublamination of $\phi(f(t))$.  By Lemma \ref{finding 
epsilon markers} for each $t\in V$, there exists an $\epsilon/2$ marker hit by $
\phi( f(t))$.  Since hitting markers is an open condition and $V$ is compact 
there exists an $\epsilon/2$-marker family $\mJ$ hit by $f$.  Now apply Theorem 
\ref{marker theorem} to obtain $g:V\to \EL$ such that $g$ hits $\mJ$.  
Therefore for each $t\in V$, $\dpts(g(t), \phi(f(t)))\le\epsilon/2$ and hence $g(t)
\subsetneq N_{PT(S)}(z', \epsilon(z)/2)$ for all diagonal extensions $z' $ of $z$.  
Therefore $g(t)\notin U_z$.  It follows that $g(V)\cap U_z=\emptyset$.

Since $ \EL$ is separable and metrizable it is Lindelof,  hence there exists a countable cover of $\EL$ of the form $\{U_{z_j}\}$.   By hypothesis for every $j\in \BN$, there exists $n_j\in \BN$ such that $i\ge n_j$ implies $d_{\PML}(h_i(V), \phiinv(z_j))>\delta_{z_j}$.  We can assume that $n_1<n_2<\cdots$.  Let $m_i=\max\{j|n_j\le i\}$.

For $i\in \BN$ apply the argument of the second paragraph with $f=h_i$ and $\epsilon=\min\{\epsilon(z_1), \cdots, \epsilon(z_{m_i}), \epsilon_i\}$ to produce $g_i:V\to \EL$ satisfying for all $t\in V$, $\dpts(g_i(t), \phi(h_i(t)))<\epsilon$ and such that $g_i(V)\cap U_{z_j}=\emptyset$ for $ j\le m_i$, assuming that $m_i\neq\emptyset$.  It follows that $g_1, g_2, \cdots$ satisfy the conclusion of the Lemma.  \end{proof}

\section{Good cellulations of $\PML$}  

The main result of this section, Proposition \ref{cellulation} produces a sequence of cellulations of $\PML$ such that each cell is the polytope of a train trace and face relations among cells correspond to carrying among train tracks.   Each cellulation is a subdivision of the previous one, subdivision of cells corresponds to splitting of train tracks and every  train track eventually gets split arbitrarily much.  

\begin{definition}  If $\tau$ is a train track, then let $V(\tau)$ denote the set of measured laminations carried by $ \tau$ and $P(\tau)$ denote the polyhedron that is the quotient of $V(\tau)\setminus 0$, by rescaling.  A train track is \emph{generic} if exactly three edges locally emanate from each switch.  All train tracks in this section are generic.  We say that  $\tau_2$ is obtained from $\tau_1$ by a \emph{single splitting} if $\tau_2$ is obtained by splitting without collision along a single large branch of $\tau_1$. Also $\tau'$ is a \emph{full splitting} of $\tau$ if it is the result of a sequence of single splittings along each large branch of $\tau$.\end{definition}

\begin{remark}  By elementary linear algebra if $\tau_R$ and $\tau_L$ are the train tracks obtained from $\tau$ by a single splitting, then  $V(\tau_R) = V(\tau)$ or $V(\tau_L)= V(\tau)$ or $V(\tau_R)$, $V(\tau_L)$ are obtained by slicing $V(\tau)$ along a codimension-1 plane through the origin.    \end{remark}

\begin{definition}   We say that a finite set $R(\tau)$ of train tracks is 
\emph{descended} from $\tau$, if there exists a 
sequence of sets of train tracks $R_1=\{\tau\}, R_2, \cdots, R_k=R(\tau) $ 
such that $R_{i+1}$ is obtained from $R_i $ by deleting one train 
track $\tau\in R_i $ and replacing it either by the two train tracks 
resulting from  a single splitting of $\tau$ if $P(\tau)$ is split, or 
one of the resulting tracks with polyhedron equal to $P(\tau)$ otherwise.   Finally, if a replacement track is not recurrent, then replace it by its maximal recurrent subtrack.\end{definition}   
 
\begin{remark}  If $R(\tau)$ is descended from $\tau$, then $P(\tau)=\cup_{\kappa \in R(\tau)} P(\kappa)
$, each $P(\kappa)$ is codimension-0 in $P(\tau)$ and any 
two distinct $P(\kappa)$'s have pairwise disjoint interiors.  
Thus,  any set of tracks descended from $\tau$  gives rise to a subdivision of $\tau$'s polyhedron.  A consequence of  Proposition \ref{cellulation} is that after subdivision, the lower dimensional faces of the subdivided $P(\tau)$ are also polyhedra of train tracks.  \end{remark}

\begin{definition}  Let $\Delta$ be a cellulation of $\PML$ and $T$ a finite set of birecurrent generic train tracks.  We say that $T$ is \emph{associated} to $\Delta$ if there exists a bijection between elements of $T$ and cells of $\Delta$ such that if $\sigma\in \Delta$ corresponds to $\tau_\sigma\in T$, then $P(\tau_\sigma)=\sigma$.  \end{definition}

\begin{proposition} \label{cellulation} Let $S$ be a finite type hyperbolic surface.  There exists a sequence of cellulations $\Delta_0, \Delta_1, \cdots$ of $\PML$ such that 
\smallskip

i) Each $\Delta_{i+1}$ is a subdivision of $\Delta_i$.
\smallskip

ii) Each $\Delta_i$ is associated to a set $T_i $ of train tracks.
\smallskip

iii)  If $\sigma_j $ is a face of $\sigma_k$, then $\tau_{\sigma_j}$ is carried by $\tau_{\sigma_k}$.  If $\sigma_p\in \Delta_i, \sigma_q\in \Delta_j, \sigma_p\subset \sigma_q , \dim(\sigma_p)=\dim(\sigma_q) $ and $i>j$, then $\tau_{\sigma_p}$ is obtained from $\tau_{\sigma_q}$ by finitely many single splittings and possibly deleting some branches.  
\smallskip

iv)  There exists a subsequence $\Delta_{i_0}=\Delta_0, \Delta_{i_1}, \Delta_{i_2}, \cdots$ such that each complete $\tau\in T_{i_N}$ is obtained from a complete $\tau'\in T_0$ by $N$ full splittings. \end{proposition}

\begin{definition}  A sequence of cellulations $\{\Delta_i\}$ satisfying the conclusions of Proposition \ref{cellulation} is called a \emph{good cellulation sequence}.\end{definition}

\begin{proof}   Let $T_0$ be the set of standard train tracks associated to a parametrized pants decomposition of $S$.  As detailed in \cite{PH}, $T_0$ gives rise to a cellulation $\Delta_0$ of $\PML$ such that each $\sigma\in \Delta_0$ is the polyhedron of a track in  $T_0$.  

The idea of the proof is this.  Suppose we have constructed $\Delta_0, \Delta_1, \cdots, \Delta_i$, and $T_0, T_1, \cdots T_i$ satisfying conditions i)-iii).  Since any full splitting of a train track is the result of finitely many single splittings  it suffices to construct $\Delta_{i+1}$ and $T_{i+1}$ satisfying i)-iii)  such that the complete tracks of $T_{i+1}$ consist of the complete tracks of $T_i$, except that a single complete track $\tau$ of $T_i$ is replaced by $\tau_R$ and/or $\tau_L$.  Let $T_i'$ denote this new set of complete tracks.    The key technical Lemma \ref{shattering} implies that if $\sigma$ is a codimension-1 cell of $\Delta_i$ with associated train track $\tau_\sigma$, then every element of some set  of train tracks descended from $\tau_\sigma$  is carried by an element of $T_i'$.    This implies that $\sigma$  has been subdivided in a manner compatible with the subdivision of the top dimensional cell $P(\tau)$.  Actually, we must be concerned with lower dimensional cells too and new cells that result from any subdivision.      Careful bookkeeping together with repeated applications of Lemma \ref{shattering} deals with this issue.  

\begin{lemma}  \label{shattering} Suppose that $\kappa$ is carried by $\tau$.  Let $\tau_R$ and $\tau_L$ be the  train tracks obtained from a single splitting   of $\tau$ along the large branch $ b$.  Then there exists a set $R(\kappa)$ of train tracks descended from $\kappa$ such that for  each $\kappa'\in R(\kappa)$, $\kappa'$ is carried by one of $\tau_R$ or $\tau_L$. 

If $P(\tau)=P(\tau_R$) (resp. $P(\tau_L)$), then each $\kappa'$ is carried by the maximal recurrent subtrack of $\tau_R$ (resp. $\tau_L$).  \end{lemma}

\noindent\emph{Proof.}  Let $N(\tau)$ be a fibered neighborhood of $\tau$.  Being carried by $\tau$, we can assume that $\kappa\subset \inte(N(\tau))$ and is transverse to the ties.   Let $J\times [0,1]$ denote $\piinv(b)$ where $\pi:N(\tau)\to \tau$ is the projection contracting each tie to a point.  Here $J=[0,1]$ and each $J\times t$ is a tie.   Let $\partial_s(J\times i)$, $i\in \{0,1\}$ denote the singular point of $N(\tau)$ in $J\times i$.   The canonical projection $h:J\times [0,1]$ to $[0,1]$ gives a \emph{height} function on $J\times I$.   

We can assume that distinct switches of $\kappa$ inside of $J\times I$ occur at distinct heights and no switch occurs at heights 0 or 1.  Call a switch $s$ a \emph{down} (resp. \emph{up}) switch if two branches emanate from s that lie below (resp. above) $s$.  A down (resp. up) switch is a \emph{bottom} (resp. \emph{top}) switch if the branches emanating below (resp. above) $s$ extend to smooth arcs in $\kappa\cap J\times I$ that intersect $J\times 0\setminus \partial_s(J\times 0)$ (resp. $J\times 1\setminus \partial_s(J\times 1$)) in distinct components.  Furthermore $h(s)$ is minimal (resp. maximal) with that property.  There is at most one top switch and one bottom switch.  Let $s_T$ (resp. $s_B$) denote the top (resp. bottom) switch if it exists.

We define the, possibly empty,  \emph{b-core}  as the unique embedded arc in $\kappa\cap J\times I$ transverse to the ties with endpoints in the top and bottom switches.  We also require that $h(s_B)<h(s_T)$.  Uniqueness follows since $\kappa$ has no bigons.

We now show that if the b-core is empty, then $\kappa$ is carried by one of $\tau_R$ or $\tau_L$.  To do this it suffices to show that after normal isotopy, $\kappa$ has no switches in $J\times I$.  By \emph{normal isotopy} we mean isotopy of $\kappa$ within $N(\tau)$ through train tracks that are transverse to the ties.  It is routine to remove, via normal isotopy, the switches in $J\times I$ lying above $s_T$ and those lying below $s_B$.  Thus all the switches can be normally isotoped out of $J\times I$ if either $s_T$ or $s_B$ do not exist or $h(s_B)>h(s_T)$.  If $s_T$ exists, then since the $b$-core $=\emptyset$ all smooth arcs descending from $s_T$ hit only one component of $J\times 0\setminus \partial_s(J\times 0)$.  Use this fact to  first normally isotope $\kappa$ to  remove from $J\times I$  all the switches lying on smooth arcs from $s_T$ to $J\times 0$ and then to isotope $s_T$ out of $J\times I$.

Now assume that the $b$-core exists.  Let $u_1, \cdots, u_r$, $d_1, \cdots, d_s$ be respectively the up 
and down switches of $\kappa$ that lie on the core.  For $i\in 
\{1,\cdots, s\}$, let $u(i)$ be the number of up switches in the b-core 
that lie above the down switch $d_i$.  Define $C(\kappa)=
\sum_{j=1}^s u(i)$.  Define $C(\kappa)=0$ if the $b$-core is empty.  Note that if $C(\kappa)=1$, then $r=s=1$ and 
$u_1=s_T$ and $d_1=s_B$.  Furthermore splitting $\kappa$ along 
the large branch between $s_T$ and $s_B$ gives rise to tracks 
whose b-cores are empty and hence are carried by one of $\tau_R$ 
or $\tau_S$.  

Assume by induction that the first part of the lemma holds for all train tracks $\kappa$ with  $C(\kappa)<k$.     Let $\kappa$ be a track with $C(\kappa)=k$.  Let $e$ be a large branch of $\kappa$ lying in $\kappa$'s b-core.  The two train tracks obtained by splitting along $e$ have reduced $C$-values.  Therefore, the first part of the lemma follows by induction.

Now suppose that $P(\tau)=P(\tau_R)$.  It follows that $P(\tau_L)\subset P(\tau_R)$ and hence $P(\tau_L)=P(\tau_C)$, where $\tau_C$ is the train track obtained by splitting $\tau$ along $b$ with collision.  Therefore if $\kappa'$ is the result of finitely many simple splittings, and is carried by $\tau_L$, then the maximal recurrent subtrack of $\kappa'$ is carried by $\tau_C$ and hence by the maximal recurrent subtrack of  $\tau_R$.  \end{proof}

\begin{remarks} i)  If each $u_i$ is to the \emph{left} of the b-core  and each $d_j$ is to the \emph{right} of the b-core, then if $\kappa_R$ and $\kappa_L$ are the tracks resulting from a single splitting along a large branch in the b-core, then $C(\kappa_R)= C(\kappa)-1$, while $C(\kappa_L)=0$.  
\smallskip

ii)  Any large branch of $\kappa$ that intersects the b-core is contained in the b-core.  It follows that if $\kappa'$ is the result of  finitely many single splittings of $\kappa$, then $C(\kappa')
\le C(\kappa)$.\end{remarks}

Now assume we have a sequence $(\Delta_0, T_0), \cdots, (\Delta_i, T_i)$ of cellulations and associated train tracks that satisfy i)-iii) of the proposition.  Let $\sigma$ be a cell of $\Delta_i$ with associated train track $\tau$.   Since any full splitting of a track is the concatenation of splittings along large branches, to complete the proof of the proposition, it suffices to prove the following claim.  

\begin{claim}  If $\tau_R$ and $\tau_L$ are the result of a single splitting of $\tau$, then there exists a cellulation $\Delta_{i+1}$ with associated train tracks $T_{i+1}$ extending the sequence and satisfying i)-iii), so that  $\tau$ is replaced with the maximal recurrent subtracks of   $\tau_R$ and/or $\tau_L$   and if $\tau'\in T_i$ is such that  $\dim(P(\tau'))\ge \dim(P(\tau))$ and $\tau'\neq\tau$, then $\tau'\in T_{i+1}$.  \end{claim}

\begin{proof}  Proof by induction on $\dim(\sigma)$.  We will assume that each of $P(\tau_R)$ and $P(\tau_L)$ are proper subcells of $P(\tau)$, for proof in the general case is similar.  The claim is trivial if $\dim(\sigma)=0$.  Now assume that the claim is true if $\dim(\sigma)<k$.  Assuming that $\dim(\sigma)=k$ let $\sigma_1, \cdots, \sigma_p$, be the $(k-1)$-dimensional faces of $\sigma$ with corresponding train tracks $\kappa_1, \cdots, \kappa_p$.  By Lemma \ref{shattering} there exists sets $R(\kappa_1), \cdots, R(\kappa_p)$ descended from the $\kappa_i$'s such that any train track in any of these sets is carried by one of $ \tau_R$ and $\tau_L$.  

If $R(\kappa_1)\neq \{\kappa_1\}$, then there exists a single splitting of $\kappa_1$ into $
\kappa_1^1$ and $\kappa_1^2$ such that $R(\kappa_1)=R(\kappa_1^1)\cup R(\kappa_1^2)$, 
where $R(\kappa_1^i)$ is descended from $\kappa_1^i$. (As usual, only one may be relevant and 
it might have stuff deleted.) By induction there exists a cellulation $\Delta_i^1$ with associated 
$T_i^1$ such that the sequence $(\Delta_0, T_0), \cdots, (\Delta_i, T_i)$, $(\Delta_i^1, T_i^1)$ 
satisfies i)-iii) and in the passage from $T_i$ to $T_i^1$, $P(\kappa_1)$ is the only polyhedron of 
dimension $\ge k-1$ that gets subdivided and it is replaced by $P(\kappa_1^1)$ and $P
(\kappa_1^2)$.    By repeatedly applying the induction hypothesis we obtain the sequence $
(\Delta_0, T_0), \cdots, (\Delta_i, T_i), (\Delta_i^1, T_i^1), \cdots$, $(\Delta_i^p, T_i^p)$ satisfying 
i)-iii) such that in the passage from $\Delta_i$ to $\Delta_i^p$, each $\sigma_j$ is subdivided into 
the polyhedra of $R(\kappa_j)$ and no cells of dimension $\ge k$ are subdivided and their 
associated train tracks are unchanged.    It follows that if $\Delta_{i+1}$ is obtained by subdividing 
$\Delta_i^p$ by replacing $\sigma$ by $P(\tau_R)$, $P(\tau_L)$ and $P(\tau_C)$ and $T_{i+1}$ is 
the set of associated train tracks, then $(\Delta_1, T_1), \cdots, (\Delta_i, T_i), (\Delta_{i+1}, T_{i
+1})$ satisfies i)-iii) and hence the induction step is completed.\end{proof}

\begin{lemma}  \label{sequence properties} Let $\Delta_1, \Delta_2, \cdots$ be a good cellulation sequence.  

i)  If $\sigma$ is a cell of $\Delta_i$, then $\phi(\sigma)\cap \EL$ is closed in $\EL$.

ii) If $\mL=\phi(x)$ some $x\in\PML$, then for each $i\in \BN$ there exists a unique cell $\sigma\in \Delta_i$ such that $\phiinv(\mL)\subset \inte(\sigma)$.

iii) If $U$ is an open set of $\PML$ that is the union of open cells of $\Delta_i$, then $\phi(U)\cap \EL$ is open in $\EL$.

iv) If $\sigma$ is a cell of $\Delta_i$, then $\inte(\sigma)$ is open in $\phi(\sigma)\cap \EL$.

v)  If $\mu\in \EL$ and $\phiinv(\mu)\in \inte(\sigma)$, $\sigma$ a cell of $\Delta_i$, then there exists $\epsilon(\mu)>0$ such that $B(\mu,\epsilon(\mu))=\{x\in \PML|\dpts(\phi(x),\mu)<\epsilon(\mu)\}\subset \hat\sigma$.

 \end{lemma}

\begin{proof}  Conclusion i) follows from Lemma \ref{closed map}.  

Conclusion ii) follows from the fact that if  $\tau$ carries $\mL$, then $\phiinv(\mL)\subset P(\tau)$.

Conclusion iii) follows from Conclusion i) and the fact that $\phi(U)\cap \phi(\PML\setminus U)\cap \EL=\emptyset$.

Conclusion iv) follows from Conclusion i) and the fact that $\phi(\inte(\sigma))\cap \phi(\partial \sigma)\cap \EL=\emptyset$. 

Let $\kappa$ be a cell of $\Delta_i$ with associated train track $\tau$.  If $\tau$ carries $\mu$, then $\kappa\cap\sigma\neq\emptyset$ and hence by ii) $\sigma$ is a face of $\kappa$ and hence $\inte(\kappa)\subset\hat \sigma$.  If $\tau$ does not carry $\mu$, then by Lemma \ref{train track closeness}  $B(\mu, \epsilon)\cap P(\tau)=\emptyset$ for $\epsilon$ sufficiently small.  Since $\Delta_i$ has finitely many cells, the result follows.  \end{proof}

We have the following PML-approximation result for good cellulation sequences.
 
\begin{proposition} \label{pml approximation three}  Let $\Delta_1, \Delta_2, \cdots$ be a good cellulation sequence, $K$ a finite simplicial complex, $i\in \BN$ fixed and $g:K\to \EL$.  Then there exists $h:K\to \PML$ a generic PL map such that for each $t\in K$, there exists $\sigma(t)$ a cell of  $\Delta_i$ such that $h(t)\cup\phiinv(g(t))\subset \hat\sigma(t)$.\end{proposition}

\begin{proof}  Given $\mu\in \EL$, let $\sigma$ be the cell of $\Delta_i$ such that $\phiinv(\mu)\subset \inte(\sigma)$.  Let $\epsilon(\mu)$ be as in Lemma \ref{sequence properties} v).  By Lemmas \ref{two closeness} and \ref{one closeness} there exists a neighborhood $U_\mu\subset \EL$ of $\mu$ and $\delta_\mu>0$ such that $\dpts(\mL, z)<\delta_\mu, z\in U_\mu, \mL\in \LS$ implies that $\dpts(\mL, \mu)<\epsilon(\mu)$.  

The compactness of $g(K)$ implies that there exists a finite open cover $U_{\mu_1}, \cdots, U_{\mu_k}$ of $g(K)$.  Let $\delta=\min\{\delta_{\mu_1}, \cdots, \delta_{\mu_k}\}$.  Then any $\delta$ PML-approximation as in Proposition \ref{pml approximation two} satisfies the conclusion of the proposition.\end{proof}

\begin{lemma} \label{carries diagonal} If $\sigma=P(\tau)$ is a cell of $
\Delta_i$, $\sigma'$ is a face of $
\sigma$ and $\mu\in \phi(\sigma')\cap\EL$, then $\tau$ fully carries a 
diagonal extension of $\mu$.\end{lemma}

\begin{proof} Apply Lemma \ref{carries diagonal extension} and Proposition \ref{cellulation} iii).\end{proof}

\begin{lemma} \label{fine splitting}  Let $\Delta_1, \Delta_2, \cdots$ be a good cellulation sequence.  Given $\epsilon>0$ there exists $N\in \BN$ such that if $i\ge N$, $\mL_1\in \EL$, $\mL_2\in \mL(S)$ and $\mL_1, \mL_2$ are carried by  $\tau$ some $\tau\in T_i, i\ge N$, then $\dpts(\mL_1, \mL_2)<\epsilon$.\end{lemma}

\begin{proof}  Apply Propositions \ref{cellulation} and \ref{convergence}.\end{proof}

\begin{definition}  If $\sigma$ is a cell of $\Delta_i$, then define the \emph{open star} $\hat\sigma=\cup_{\eta\in \Delta_i|\sigma\subset  \eta}\inte(\eta)$.  \end{definition}

\begin{lemma}  \label{contractible} If $\sigma$ is a cell of $\Delta_i$, then $\hat\sigma$ is open in $\PML$ and contractible.  Indeed, it deformation retracts to $\inte(\sigma)$.\qed\end{lemma}

\begin{lemma}  If $\sigma$ is a face of $\kappa$, then $\hat\kappa\subset \hat\sigma$.   \qed\end{lemma}

\begin{definition}  Define $U(\sigma)=\phi(\hat\sigma)\cap \EL$.\end{definition}

\begin{lemma}  \label{basis}$\mB= \{U(\sigma)|   \phi(\inte(\sigma))\cap\EL\neq
\emptyset $ and $\sigma\in \Delta_i$  for some $i\in \BN\}$ is a neighborhood basis of $
\EL$.\end{lemma}

\begin{proof}  By Lemma  \ref{one closeness} it suffices to show that for each $\mu\in \EL$ and $\epsilon>0$, there exists a neighborhood $U(\sigma)$ of $\mu$  such that $z\in U(\sigma)$ implies $\dpts(z,\mu)<\epsilon$.  Choose $N$ such that the conclusion of Lemma \ref{fine splitting} holds.  Let $\sigma$ be the simplex of $\Delta_N$ such that $\phiinv(\mu)\subset \inte(\sigma)$.  If $z\in U(\sigma)$, then $\phiinv(z)\in P(\tau)$ where $\sigma$ is a face of $\tau$.  Therefore, $z$ and $\mu$ are carried by $\tau$ and hence $\dpts(\mu,z)<\epsilon$. \end{proof}

\begin{lemma}  \label{refinement2}  Let $\mU$ be an open cover of $\EL$. Then there exists a refinement $\mU_2$ of $\mU$ by elements of $\mB$, maximal with respect to inclusion, such that for each $U(\sigma_2)\in \mU_2$ there exists  $U\in \mU$ and  $\delta(U(\sigma_2))>0$ with the following property.   If $x\in \hat\sigma_2$ and $z\in \EL$ are such that $\dpts(\phi(x),z)<\delta(U(\sigma_2))$, then $z\in U$.\end{lemma}

\begin{proof}  It suffices to show that if $\mu\in \EL$ and $U\in \mU$, then there exists $\sigma_2$ and $\delta(\sigma_2)$ such that the last sentence of the lemma holds.    By Lemma \ref{one closeness} there exists $\delta_1>0$ such that if $z\in \EL$ and $\dpts(z,\mu)<\delta_1$, then $z\in U$.     By Proposition \ref{convergence} there exists $N_1>0$ and $\delta>0$ such that if $\tau$ is obtained by fully splitting any one of the train tracks associated to the top dimensional cells of $\Delta_1$  at least $N_1$ times, $\mu$ is carried by $\tau, \mL\in \mL(S)$ is carried by $\tau, z\in \EL$ with $\dpts(z,\mL)<\delta$, then $\dpts(z,\mu)<\delta_1$ and hence $z\in U$.  Therefore, if $\sigma$ is the cell of $\Delta_{N_1}$ such that $\inte(\sigma)$ contains $\phiinv(\mu)$, then let $\sigma_2=\sigma$ and $\delta(U(\sigma_2))=\delta$. \end{proof}


\section{Bounds on the dimension of $\EL$}

In this section we give an upper bound for $\dim(\EL)$ for any finite type hyperbolic surface $S$.  When  $S$ is either a punctured sphere or torus we give lower bounds for $\dim(\EL)$.  For punctured spheres these bounds coincide.   We conclude that if $S$ is the $n+4$ punctured sphere, then $\dim(\EL)=n$.

To compute the lower bounds we will show that  $\pi_n(\EL)\neq 0$ for $S=S_{0,4+n}$ or $S_{1, 1+n}$.  Since $\EL$ is $(n-1)$-connected and $(n-1)$-locally connected Lemma \ref{dranishnikov} applies. 

To start with, $\EL$ is a separable metric space and hence  covering dimension, inductive dimension and cohomological dimension  of $\EL$ all coincide \cite{HW}.  By \emph{dimension} we mean any of these equal values.

The next lemma is key to establishing the upper bound for  $\dim(\EL)$.

\begin{lemma} \label{dimension n+1} Let $S=S_{g,p}$, where $p> 0$ and let 
$\{\Delta_i \}$ be a good cellulation sequence.  If $\sigma$ is a cell of $\Delta_i$ and $
\phi(\sigma)\cap \EL\neq\emptyset$, then $\dim(\sigma)\ge n+1-g$, where $
\dim(\PML)=2n+1$.

If $S=S_{g,0}$ and $\sigma$ is as above, then $\dim(\sigma)\ge n+2-g$.\end{lemma}

\begin{proof}  We first consider the case $p\neq 0$.  By Proposition \ref{cellulation}, $
\sigma=P(\tau)$ for some train track $\tau$.  It suffices to consider the case that $\tau$ is 
generic.  A generic track with e edges has $2e/3$ switches, hence $
\dim(V(\tau))\ge e-2e/3=e/3$, e.g. see page 116 \cite{PH}.  Since $\tau$ carries an element of $\EL$, 
all complementary regions are discs with at most one puncture.   After 
filling in the punctures, $\tau$ has say $f$ complementary regions all of 
which are discs, thus  $2-2g=\chi(S_g)=2e/3-e+f$ and hence $
\dim(V(\tau))\ge e/3=f-2+2g$.     Therefore,  $\dim(P(\tau))=
\dim(V(\tau))-1\ge f-3+2g\ge p-3+2g= n+1-g$, since $2n+1=6g+2p-7$.

If $p=0$, then $f\ge 1$ and hence the above argument shows  $\dim(P(\tau))=
\dim(V(\tau))-1\ge f-3+2g\ge 1-3+2g=n+2-g$. \end{proof}

\begin{corollary}  If $S=S_{g,p}$ with $p> 0$ (resp. $p=0$), then for each $m\in \BZ$, $\EL=U(\sigma_1)\cup\cdots\cup U(\sigma_k)$, where $\{\sigma_1, \cdots, \sigma_k\}$ are the cells of $\Delta_m$ of dimension $\ge n+1-g$ (resp. $\ge n+2-g$).\qed\end{corollary}

\begin{proposition}  \label{upper dim bound} Let $S=S_{g,p}$.  If $p> 0$ (resp. $p=0$), then $
\dim(\EL)\le 4g+p-4=n+g$ (resp. $\dim(\EL)\le 4g-5=n+g-1)$, where $
\dim(\PML)=2n+1$. \end{proposition}

\begin{proof} To minimize notation we give the proof for the $g=0$ case.  The general case follows similarly, after appropriately shifting dimensions and indices.  

Let $\mE_r$ denote those elements of $\mu\in \EL$ such that $\mu=\phi(x)$ for some $x$ in the $r$-skeleton $\Delta_i^r$ of some $\Delta_i$.  We will show that for each $i$, $\dim(\mE_i)\le \max\{i-(n+1),-1\}$.  We use the convention that $\dim(X)=-1$ if $X=\emptyset$.  

By Lemma \ref{dimension n+1} $ \Delta^r_i\cap \EL=\emptyset$ for all $i\in \BN$ and $r\le n$ and hence $\mE_n=\emptyset$.  Now suppose by induction that for all $k<m$, $\dim(\mE_{n+1+k})\le k$.  We will show that $\dim(\mE_{n+1+m})\le m$.  Now $\mE_{n+1+m}=\cup_{i\in \BN} \phi(\Delta_i^{n+1+m})\cap \EL$, each $\phi(\Delta_i^{n+1+m})\cap \EL$ is closed in $\EL$ by Lemma \ref{closed map} and $\EL$ is separable metric, hence to prove that $\dim(\mE_{n+1+m})\le m$ it suffices to show by the sum theorem (Theorem III 2 \cite{HW}) that $\dim(\phi(\Delta_i^{n+1+m})\cap\EL)\le m$ for all $i\in \BN$.  

Let $X=\phi(\Delta_i^{n+1+m})\cap \EL$.  To show that the inductive dimension of $X$ is $\le m$ it suffices to show that if $U$ is open in $X$, $\mu\in U$, then there exists $V$ open in $X$ with $\mu\in V$ and $\partial V\subset \mE_{n+m}$.

Let $\sigma\in \Delta_j$, some $j\ge i$, such that $\mu\in U(\sigma)\cap X\subset U$.  Such a $ \sigma$ exists by Lemma \ref{basis}.  Now $\hat\sigma=\cup_{u\in J} \inte(\sigma_u)$ where the union is over all cells in $\Delta_j$ having $\sigma$ as a face.  Since $\Delta_i^{n+1+m}\subset \Delta_j^{n+1+m}$ it follows that after reindexing, $U(\sigma)\cap X\subset \phi(\inte(\sigma_1)\cup\cdots\cup \inte(\sigma_q))\cap\EL$, where $\sigma_1, \cdots, \sigma_q$ are those $\sigma_u$'s such that $\dim(\sigma_i)\le n+m+1$.  

By Lemma \ref{closed map} $(\cup_{i=1}^q \phi(\sigma_i))\cap \EL$ is closed in $\EL$ and hence restricts to a closed set in $X$.  It follows that $\partial (U(\sigma)\cap X)\subset (\cup_{i=1}^q \phi(\partial \sigma_i))\cap \EL\subset \mE_{n+m}.$\end{proof}

To  establish our lower bounds on $\dim(\EL)$ we will use the following basic result.

\begin{lemma}  \label{dranishnikov} Let $X$ be a separable metric space such that $X$ is $(n-1)$-connected and $(n-1)$-locally connected.  If $\pi_n(X)\neq 0$, then $\dim(X)\ge n$.\end{lemma}

\begin{proof}  (Dranishnikov \cite{Dr})  Let $f:S^n\to X$ be an essential map.  If $\dim(X)\le n-1$, then by \cite{HW} $\dim(f(S^n))\le \dim(X)$ and hence $\dim(f(X))$ is at most $n-1$. By Bothe \cite{Bo} there exists a compact, metric, absolute retract $Y$ such $\dim(Y)\le n$ and $f(S^n)$ embeds in $Y$.  By Theorem 10.1 \cite{Hu}, since $Y$ is  metric and $X$ is $(n-1)$-connected and $(n-1)$-locally connected, the inclusion of $f(S^n)$ into $ X$ extends to a map $g:Y\to X$.  Now $Y$ is contractible since it is an absolute retract. (The  cone of $Y$ retracts to $Y$.)  It follows that $f$ is homotopically trivial, which is a contradiction.\end{proof}

We need the following controlled homotopy lemma for $\PML$ approximations that are very close to a given map into $\EL$.

\begin{lemma} \label{controlled homotopy}  Let $K$ be a finite simplicial complex.  Let $g:K\to \EL$.  Let  $U\subset \PML$ be a neighborhood of $\phiinv(g(K))$.  There exists $\delta>0$ such that if $f_0, f_1:K\to \PML$ and for every $t\in K$ and $i\in \{0,1\}$, $\dpts(\phi(f_i(t))$, $g(t))<\delta$, then there exists a homotopy from $f_0$ to $f_1$ supported in $U$.  \end{lemma}

\begin{proof}  Let $\Delta_1, \Delta_2, \cdots$ be a good cellulation sequence of $\PML$.  Given $x\in \EL$ and $i\in \BN$ let $\sigma_x^i$ denote the cell of $\Delta_i$ such that $\phiinv(x)\subset \inte(\sigma_x^i)$.  Such a $\sigma_x^i$ exists since each cell of $\Delta_i$ is associated to a train track.  By Lemma \ref{preimage close}, 
given $x\in \EL$ there exists a neighborhood $W_x$ of $x$ such that $\phiinv(W_x)\subset U$ and lies in a small neighborhood of the cell $\phiinv(x)$, i.e. that only intersects $\sigma_x^i$ and higher dimensional cells that have $\sigma_x^i$ as a face.  Therefore,  $y\in W_x$ implies that $\sigma_x^i$ is a face of  $\sigma_y^i$ and hence $\hat \sigma_y^i\subset \hat \sigma_x^i$.

Let $U'$ be a neighborhood of $\phiinv(g(K))$ such that $\bar U'\subset U$.  
By Lemmas \ref{basis} and \ref{preimage close}, for $i$ sufficiently large, $
\phiinv(U(\sigma_x^i))\subset U'$.  This implies that $\hat \sigma_x^i\subset U
$, since $\FPML$ is dense in $\PML$.     Let $U_i=\cup_{x\in g(K)}\hat 
\sigma_x^i$.  By compactness of $g(K)$ and the previous paragraph it follows 
that for $ i$ sufficiently large $U_i\subset U$.  Fix such an $i$.  Let $\Sigma=\{\sigma\in \Delta_i|$ some $\sigma_x^i$ is a face of $\sigma.\}$.

There exists $\delta>0$ and a function $W:g(K)\to \Sigma$ such that if $y\in \PML, x\in g(K)$ and $\dpts(x, \phi(y))<\delta$, then $y\in \hat W(x)$.  Furthermore, if $\dpts(x_1,\phi(y))<\delta, \cdots, \dpts(x_u, \phi(y))<\delta$, then after reindexing the $x_i$'s, $W(x_1)\subset \cdots W(x_u)$. Modulo the epsilonics (to find $\delta$) which follow from  Lemmas \ref{pts separation} and \ref{pml one closeness} (i.e. super convergence), $W$ is defined as follows.  Define $d:g(K)\to \BN$ by $d(x)=\dim(\sigma_x^i)$.  Let $p$ be the minimal value of $d(g(K))$.  Define $W(x)=\sigma_x^i$ if $d(x)=p$.  Next define $W(x')=\sigma_x^i$ if $\dpts(x, x')$ is very small, where $d(x)=p$.  To guarantee  that the second sentence holds, we require  that if both $\dpts(x_1, x')$ and $\dpts(x_2,x')$ are very small, then $W(x_1)=W(x_2)$.  If $W(x)$ has not yet been defined and $d(x)=p+1$, then define $W(x)=\sigma_x^i$.  Next define $W(x')= \sigma_x^i$ if $W(x')$ hasn't already been defined and $\dpts(x, x')$ is very very small, where $d(x)=p+1$.  Again we require that if $\dpts(x_1, x')$ and $\dpts(x_2, x')$ are very very small,  then after reindexing $W(x_1)\subset W(x_2)$. Inductively, continue to define $W$ on all of $g(K)$.  Since  $\dim(\Delta_i)$ is finite this process eventually terminates.   Take $\delta$ to be the minimal value used to define smallness.

Next subdivide $K$ such that the following holds.  For every simplex $\kappa$ of $K$, there exists $t\in \kappa$ such that  for all $s\in \kappa$,  $ \dpts(\phi(f_0(s))$, $g(t))<\delta$ and $\dpts(\phi(f_1(s))$, $g(t))<\delta$.  Thus by the previous paragraph, for each simplex $\kappa$ of $K$, $f_0(\kappa)\cup f_1(\kappa)\subset \hat W(g(t))$ for some $t\in \kappa$ where $t$ satisfies the above property.  Let $\sigma(\kappa)$ denote the maximal dimensional simplex of $\Sigma$ with these properties.  Note that by the second sentence of the previous paragraph, $\sigma(\kappa)$ is well defined and if $\kappa'$ is a face of $\kappa$, then $\sigma(\kappa)$ is a face of $\sigma(\kappa')$.  

We now construct the homotopy $F:K\times I\to \PML$ from $f_0$ to $f_1$.   Assume that $K$ has been subdivided as in the previous paragraph.   If $v$ is a vertex of $K$, then both $f_0(v), f_1(v)\in \hat\sigma(v)$ which is contractible by Lemma \ref{contractible}.  Thus $F$ extends over $v\times I$ such that $F(v\times I)\subset \hat\sigma(v)$.  Assume by induction that if $\eta$ is a simplex of $K$ and $\dim(\eta)<m$, then $F$ has been extended over $\eta\times I$  with $F(\eta)\subset \hat\sigma(\eta)$.   If $\kappa$ is an $m$-simplex, then the contractibility of $\hat\sigma(\kappa)$ enables us to extend $F$ over $\kappa\times I$, with $F(\kappa\times I)\subset \hat \sigma(\kappa)$.\end{proof}

\begin{theorem}  \label{homotopy group} If $S=S_{0,4+n}$ or $S_{1, 1+n}$, then $\pi_n(\EL)\neq 0$. \end{theorem}

\begin{proof}  Let $S$ denote either $S_{0,4+n}$ or $S_{1, 1+n}$.  In either case $\dim(\PML)=2n
+1$.   By Harer \cite{Hr} (see also Ivanov \cite{Iv}, \cite{IJ}) the curve complex $\CS$ is homotopy 
equivalent to a non trivial wedge of $n$-spheres.  Thus, there exists an essential map $h:S^n\to 
\CS$.  We can assume that $h(S^n)$ is a simplicial map with image a finite subcomplex $L$ of 
the $n$-skeleton of $\CS$ and hence $0\neq [h_*([S^n])]\in H_n(L)$.  We abuse notation by 
letting $L$ denote the image of $L$ under the natural embedding, given in Definition \ref{curve complex}.  Let 
$Z$ be a simplicial $n$-cycle that represents $[h_*([S^n])]$. Let $\sigma$ be a $n$-simplex of $L$ 
in the support of $Z$.  Let $f_0:S^n\to \PML\setminus L$ be a PL embedding which links $\sigma
$, i.e.  there exists an extension  $F_0:B^{n+1}\to \PML$ of $f_0$ such that $F_0$ is transverse to 
$L$ and intersects $L$ at a single point in $\inte(\sigma)$.  It follows that $f_0[S^n]$ has non 
trivial linking number with $Z$ and hence $f_0$ is homotopically non trivial as a map into $\PML
\setminus L$.   The theorem is now a consequence of  the following lemma.\end{proof}
\begin{lemma}  \label{homotopy} Let $S$ be a finite type hyperbolic surface with $\dim(\PML)=2n+1$.  Let $L$ be an $n$-dimensional subcomplex of $\CS$ that is naturally embedded in $\PML$.    If  $\pi_n(\PML\setminus L)\neq 0$, then $\pi_n(\EL)\neq 0$.  Indeed, if $f:S^n\to \PML\setminus L$ is essential, then there exists a map $F: S^n\times I \to \PMLEL\setminus L$ such that $F|S^n\times 0=f$, $F(S^n\times [0,1))\subset \PML\setminus L$ and $0\neq [F|S^n\times 1]\in \pi_n(\EL)$.\end{lemma}

\begin{proof} View $f$ as a map into $\PMLEL$ with image in $\PML$.  
\vskip 10pt
\noindent\emph{Step 1}.   $ f $ extends to a map $f:S^n\times I \to \PMLEL$ such that $f(S^n\times [0,1)\subset \PML\setminus L$ and $f(S^n\times 1)\subset \EL.$

\vskip 10 pt

\noindent\emph{Proof of Step 1}.  Let $\Sigma$ be a triangulation of $S^n$.  It 
suffices to consider the case that $f$ is a generic PL map.  Thus 
$f(\Sigma^0)\subset \FPML$.  Extend $f|\Sigma^0\times 0$ to $\Sigma^0\times I
$ by $f(t,s)=f(t,0)$ where $f(t,1)$ is viewed as an element of $\EL$. Now for some 
$u\le n$, assume that $ f$ extends as desired to $\Sigma^u\times I \cup S^n
\times 0$.  If $\sigma$ is a $(u+1)$-simplex of $\Sigma$, then $f|\sigma\times 
0\cup \partial \sigma\times I$ maps into $\PMLEL$ with $f(\partial \sigma\times 
1)\subset \EL$ and the rest mapping into $\PML\setminus L$.  By Lemma \ref{marker theorem} $f$ 
extends to $\sigma\times I$ such that $f(\sigma\times 1)\subset \EL$ and $f(\sigma\times [0,1))\subset \PML\setminus L$.  Step 1 now 
follows by  induction.\qed

\vskip 10 pt
\noindent\emph{Step 2}  If $g=f|S^n\times 1$, then $g$ is an essential map into $\EL$.

\vskip 10pt

\noindent\emph{Proof of Step 2}.  Otherwise there exists $G:B^{n+1}\to \EL$ extending $g$.  Since $L\subset \CS$, $L\cap \phiinv(G(B^{n+1}))=\emptyset$.  Let $F:B^{n+1}\to \PML$ be a $\delta-\PML$ approximation of $G$.  By Lemma \ref{pml approximation two} $F(B^{n+1})\cap L=\emptyset$ for $\delta$ sufficiently small.  Now $F|S^n$ and $f|S^n\times s$ are respectively $\delta$, $\delta'$-$\PML$ approximations of $g$, where $\delta'\to 0$ as $s\to 1$.  Since  $\delta$ can be chosen arbitrarily small it follows from Lemma \ref{controlled homotopy} that $F$ can be chosen such that for $s$ sufficiently large, $f|S^n\times s$ and $F|S^n$ are homotopic via a homotopy supported in $\PML\setminus L$.  By concatenating this homotopy with $F$, we conclude that $f$ is homotopically trivial via a homotopy supported in $\PML\setminus L$; which is a contradiction.\end{proof}

\begin{proposition}  \label{lower bound}   If $S=S_{0,p}$, then $\dim(\EL)\ge p-4$.  If $S=S_{1,p}$, then $\dim(\EL)\ge p-1$.  In either case if $\dim(\PML)=2n+1$, then $\dim(\EL)\ge n$.  \end{proposition}

\begin{proof}  By Theorems \ref{connectivity} and \ref{local connectivity}, $\EL$ is $(n-1)$-connected and $(n-1)$-locally connected.  By Theorem \ref{homotopy group}, $\pi_n(\EL)\neq 0$.  Therefore by Lemma \ref{dranishnikov}, $\dim(\EL)\ge n$.\end{proof}

By  Proposition \ref{lower bound} and Proposition \ref{upper dim bound}  we obtain;

\begin{theorem}  \label{punctured sphere dimension} If $S$ is the $(4+n)$-punctured sphere, then $\dim(\EL)=n$.\qed  \end{theorem} 

\begin{theorem}  If $S$ is the $(n+1)$-punctured torus, then $n+1\ge\dim(\EL)\ge n$.\qed\end{theorem}

\begin{remark}  In a future paper we will show that if $S=S_{g,p}$ with $p>0$, then $\dim(\EL)\le n+g-1$.\end{remark}

\section{Nobeling spaces} 

The $n$-dimensional \emph{Nobeling space} $\BR^{2n+1}_n$ is the space of points in $\BR^{2n+1}$ with at most $n$ rational coordinates.  The goal of the next two sections is to complete the proof of the following theorem.

\begin{theorem}\label{nobeling}  If $S$ is the $n+p$ punctured sphere, then $\EL$ is homeomorphic to the $n$-dimensional Nobeling space.\end{theorem}

By Luo \cite{Luo}, $\mC(S_{2,0})$ is homeomorphic to $\mC(S_{0,6})$ and thus by Klarreich \cite{K} $\mE\mL(S_{2,0})$ is homeomorphic  to $\mE\mL(S_{0,6})$.

\begin{corollary} \label{genus two} If $S$ is the closed surface of genus-2, then $\EL$ is homeomorphic to the Nobeling surface, i.e. the 2-dimensional Nobeling space.\end{corollary}

\begin{remarks}  \label{s12} Using \cite{G1}, Sabastian Hensel and Piotr Przytycki earlier proved that the ending lamination space of the 5-times punctured sphere is homeomorphic to the Nobeling curve.  They used Luo \cite{Luo} and Klarreich \cite{K} to show that  $\mE\mL(S_{1,2})$ is also homeomorphic to the Nobeling curve.

Hensel and Przytycki boldly conjectured that if $\dim(\PML)=2n+1$, then $\EL$ is homeomorphic to $\BR^{2n+1}_n$.  Theorem \ref{nobeling} gives a positive proof of their conjecture for punctured spheres.  

Independently, in 2005, Ken Bromberg and Mladen Bestvina asked \cite{BB} if ending lamination spaces are Nobeling spaces.  \end{remarks}

\begin{remarks}  Historically, the $m$-dimensional Nobeling  space was called the \emph{universal Nobeling space of dimension $m$} and a Nobeling space was one that is locally homeomorphic to the universal Nobeling space.  In 2006, S. Ageev \cite{Ag}, Michael Levin \cite{Le} and Andrzej Nagorko \cite{Na}  independently showed that any two connected Nobeling spaces of the same dimension are homeomorpic.  The 0 and 1-dimensional versions of that result were respectively given in \cite{AU} 1928 and \cite{KLT} 1997.

These spaces were named after Georg Nobeling who showed \cite{N} in 1930 that any $m$-dimensional separable metric space embeds in an $m$-dimensional Nobeling space.  
This generalized a result by Nobeling's mentor, Karl Menger, who defined  \cite{Me} in 1926 the Menger compacta and showed that any $1$-dimensional compact metric space embeds in the Menger curve.    A topological characterization of m-dimensional Menger compacta was given in \cite{Be} by Mladen Bestvina in 1984.\end{remarks}

The following  equivalent form of the topological  characterization of $m$-dimensional Nobeling spaces is due to Andrzej Nagorko.  

\begin{theorem} (Nagorko \cite{HP}) \label{nobeling characterization} A topological space X is homeomorphic to the    
m-dimensional Nobeling space if and only if the following conditions hold.

i) X is separable

ii) X supports a complete metric

iii)  X is $m$-dimensional

iv)  X is $(m-1)$-connected 

v) X is $(m-1)$-locally connected

vi) X satisfies the locally finite m-discs property\end{theorem}


\begin{definition} \label{m-discs property}  \cite{HP}   The space $X$ satisfies the \emph{locally finite $m$-discs property} if for each open cover $\{\mU\}$ of $X$ and each sequence $f_i:B^m\to X$, there exists a sequence $g_i:B^m$ to $X$ such that 

i)  for each $x\in X$ there exists a neighborhood $U$ of $x$ such that $g_i(B^m)\cap U=\emptyset$ for $i$ sufficiently large.

ii)  for each $t\in B^m$, there is a $U\in \mU$ such that $f_i(t), g_i(t)\in U$.\end{definition}

\noindent\emph{Proof of Theorem \ref{nobeling}.}  By Remark \ref{separable and complete} $\EL$ is separable and supports a complete metric.  If $S=S_{0, n+4}$, then conditions iii)-v) of Theorem \ref{nobeling characterization} respectively follow from Theorems \ref{connectivity}, \ref{local connectivity} and \ref{punctured sphere dimension}.

To complete the proof of Theorem \ref{nobeling} it suffices to show that $\mE\mL(S_{0, n+4})$ satisfies the locally finite $n$-discs property.

\section{The locally finite n-discs property}

\begin{proposition}  \label{n-discs}  If $S$ is the $(n+4)$-punctured sphere, then $\EL$ satisfies the locally finite $n$-discs property.\end{proposition}

	Our proof is modeled on the proof, due to Andrzej Nagorko \cite{HP}, that $\BR^{2n+1}_n$ satisfies the locally finite n-discs property.  In particular, we modify his notions of \ \emph{participates} and \emph{attracting grid}.    
	
	From now on $S$ will denote the (n+4)-punctured sphere.   Let $\mU$ be an open cover of $\EL$.  Let $\mU_2$ denote the refinement of $\mU$ produced by Lemma \ref{refinement2} and let $\Delta_1, \Delta_2, \cdots$ denote a good cellulation sequence.

\begin{definition}  We say that $\sigma\in \Delta_i$ \emph{participates} in $\mU_2$ if $U(\sigma')\in \mU_2$ for some face $\sigma'$ of $\sigma$.

Define $A_i=\{\sigma\in \Delta_k, k\le i|\sigma\ \textrm{participates in } \mU_2\}$ and define  $\Gamma_i=\PML\setminus\cup _{\sigma\in A_i} \hat\sigma.$
\end{definition}

\begin{remarks}  i)  By Lemma \ref{dimension n+1} and Definition \ref{basis} each cell of $A_i$ has dimension $\ge n+1$.

\smallskip\noindent ii) Note that $\Gamma_i $ is obtained from $
\Gamma_{i+1}  $ by attaching the cells of $A_{i+1}\setminus A_{i}$.  
\end{remarks}

\begin{definition}  Call $\Gamma_i $ the \emph{i'th approximate attracting grid} and $\cap_{i=1}^{\infty}\Gamma_i$ the \emph{attracting grid}.\end{definition}

\begin{definition}  Let $Y_i\subset \PML\setminus \Gamma_i$ denote the \emph{dual cell complex} to $A_i$.  Abstractly it is a simplicial complex with vertices the elements of $A_i$ and $\{v_0, v_1, \cdots, v_k\}$ span a simplex if for all $i$, $v_i\subset \partial v_{i+1}$.  \end{definition}

\begin{remark}  Since $\dim(\PML)=2n+1$ and each cell of $A_i$ has dimension at least $n+1$ it follows that  $\dim(Y_i)\le n$.  This is the crucial fact underpinning the proof of Proposition \ref{n-discs}.\end{remark}

The next result follows by standard PL topology.

\begin{lemma}  \label{retraction} For every $a>0$ there exists $N(\Gamma_i)$ a regular neighborhood of $\Gamma_i$ in $\PML$ such that $N(\Gamma_i)\subset N_{\PML}(\Gamma_i, a)$.  Furthermore, there exists a homeomorphism 

$$q:\partial(N(\Gamma_i))\times [0,1)\to \PML\setminus (\inte(N(\Gamma_i))\cup Y_i)$$ 
such that $q|\partial N(\Gamma_i)\times 0$ is the canonical embedding and a retraction $$p:\PML\setminus Y_i\to N(\Gamma_i)$$ 
that is the identity on $N(\Gamma_i)$, quotients  $[0,1)$ fibers to points and has the following additional property.  If $\sigma\in A_i$ and $x\in \inte(\sigma)\setminus Y_i$, then $p(x)\in \inte(\sigma)$. \qed\end{lemma}

\vskip10pt
\noindent\emph{Proof of Proposition \ref{n-discs}}.   Let $f_i:B^n\to \EL, i\in \BN$ be a sequence of continuous maps.  Being compact $f_i(B^n)$ is covered by finitely many elements $U(\sigma^i_{1}), \cdots, U(\sigma^i_{i_k})$ of $\mU_2$ and hence $\phiinv(f_i(B^n))\subset \PML\setminus \Gamma_{n_i}$ for some $n_i\in \BN$. 

Let $\epsilon_i=\min\{\delta(U(\sigma_j))|\sigma_j \in A_{n_i}\}$ where $\delta(U(\sigma_j))$ is as in Lemma \ref{refinement2}.

Since both $\Gamma_{n_i}$ and $\phiinv(f_i(B^n))$ are compact and disjoint it follows that \\$\dpml(\Gamma_{n_i}, \phiinv(f_i(B^n)))=a_i>0$.

\begin{claim}  It suffices to show that for each $i\in \BN$ there exists $h_i:B^n\to \PML$, a generic PL map, such that for every $t\in B^n$

\smallskip i) $ \dpml(\Gamma_{n_i}, h_i(t))<1/i$ and

\smallskip ii)  $h_i(t)\cup\phiinv(f_i(t))\subset \hat\sigma_j$ for some $\sigma_j\in A_{n_i}$.\end{claim}

\begin{proof}  To start with Condition i) implies that for every $z\in \EL$, there exists $\delta_z>0$ such that $ \dpml(h_i(B^n)$, $\phiinv(z))>\delta_z $ for $i$ sufficiently large.  This uses the fact that for each $z\in \EL$, $\phiinv(z)$ is compact and disjoint from  $\Gamma_i$ for $i$ sufficiently large.  

Now apply the $\mE\mL$-approximation Lemma \ref{el approximation} to the $\{\epsilon_i\}$ and $\{h_i\}$ sequences to obtain maps $g_i:B^n\to \EL, i\in \BN$ such that 

\smallskip a)  For every $z\in \EL$ there exists $U_z$ open  in $\EL$ such that $g_i(B^n)\cap U_z=\emptyset$ for $i$ sufficiently large, and 

\smallskip b) if $t\in B^n$, then  $\dpts(g_i(t), \phi(h_i(t)))<\epsilon_i$.
\smallskip

Conclusion a) gives the local finiteness condition.
By ii) of the Claim, if $t\in B^n$, then $h_i(t)$, $\phiinv(f_i(t))$ lie in the same $\hat \sigma_j$.  Since $\epsilon_i<\delta(U(\sigma_j))$, Lemma \ref{refinement2} implies that $g_i(t)$ lies in $U\in \mU$ for some  $U(\sigma_j)\subset U$.  To apply that lemma, let $x=h_i(t)$ and $z=g_i(t)$.  Since $f_i(t)\in U(\sigma_j)$ the approximation condition holds too.\end{proof}

Fix $i>0$.  Let $a_i>0$ and $\epsilon_i$ be as above.  Let $a=\min\{1/2i, a_i/2\}$.  Using this $a$, let $p:\PML\setminus Y_{n_i}\to N(\Gamma_{n_i})$ be the retraction given by Lemma \ref{retraction}.     

By the PML-approximation Lemma \ref{pml approximation three}, there exists $h_i':B^n\to \PML$ such that for each $t\in B^n$, $\dpml(h_i'(t), \phiinv(f_i(t)))\subset \hat \sigma$ for some cell $\sigma$ of $\Delta_{n_i}$.  Since $\dim(B^n)+\dim(Y_{n_i})\le 2n<2n+1$ we can assume that $h_i'(B^n)\cap Y_{n_i}=\emptyset$ and this property still holds.  Let $h_i=p\circ h_i' $ perturbed slightly to be a generic PL map.

By construction $p(h_i'(B^n))\subset p(\PML\setminus Y_{n_i})\subset N(\Gamma_{n_i}) \subset \Npml(\Gamma_{n_i}, a)\subset \Npml(\Gamma_{n_i}, 1/2i)$.  Thus Condition i) of the Claim holds for $h_i$, if it is obtained by a sufficiently small perturbation of $p\circ h_i'$.

By Lemma \ref{retraction} and Lemma \ref{sequence properties} ii), $p\circ h_i'(t)$ and $\phiinv(z)$ lie in the 
same $\hat \sigma_j$.  Thus Condition ii) holds for $p\circ h_i'$. 
Since each $\hat\sigma_j$ is open, this condition holds for any 
sufficiently small perturbation of $p\circ h_i'$ and so it holds for $h_i
$.  This completes the proof of Proposition \ref{n-discs} and hence Theorem \ref{nobeling}. \qed
\vskip 10pt

After appropriately modifying dimensions, the proof of Proposition \ref{n-discs} generalizes to a proof of the following.

\begin{proposition}  \label{m-discs} Let $S=S_{g,p}$.  Then $\EL$ satisfies the locally finite $k$-discs property for all $k\le m$, where $m=n-g$, if $p\neq 0$ and $m=n-(g-1)$, if $p=0$.\qed\end{proposition}

\section{Applications}

By Klarreich \cite{K} (see also \cite{H1}), the Gromov boundary $\partial C(S)$ of the curve complex $C(S)$ is homeomorphic to $\EL$. We therefore obtain the following results.

\begin{theorem} Let $ C(S)$ be the curve complex of the surface $S$ of genus $g$ and $p$ punctures.  Then $\partial C(S)$ is $(n-1)$-connected and $(n-1)$-locally connected.  If $g=0$, then $\partial C(S)$ is homeomorphic to the $n$-dimensional Nobeling space.  Here $n=3g+p-4$.
Also $\partial \mC(S_{2,0})=\BR^5_2$ and $\partial \mC(S_{1,2})=\BR^3_1$.  \end{theorem}

\begin{remark}  The cases of $S_{0,5}$ and $S_{1,2}$ were first proved in \cite{HP}.\end{remark}

\begin{theorem}   Let $C(S)$ be the curve complex of the $p+n$-punctured sphere, then the asymptotic dimension of $C(S)$ is equal to $n+1$.\end{theorem}

\begin{proof}  As $C(S)$ is a Gromov-hyperbolic space \cite{MM}, its asymptotic dimension is equal to $1+\dim(\partial C(S))$ by Buyalo and Lebedeva \cite{BL}.  Now apply Theorem \ref{punctured sphere dimension}.\end{proof}

Let $S$ be a finite type hyperbolic surface.  Let $DD(S)$ denote the space of doubly degenerate marked hyperbolic structures on $S\times \BR$.  These are the complete hyperbolic structures with limit set all of $S^2_\infty$ whose parabolic locus corresponds to the cusps of S.  It is topologized with the algebraic topology.  See \S 6 \cite{LS} for more details.  As a consequence of many major results in hyperbolic 3-manifold geometry, Leininger and Schleimer proved the following.

\begin{theorem} \cite{LS} $ DD(S)$ is homeomorphic to $\EL\times \EL\setminus \Delta$, where $\Delta$ is the diagonal.\end{theorem}

\begin{corollary}  If $S$ is the $(n+p)$-punctured sphere, then $DD(S)$ is homeomorphic to $\BR^{2n+1}_n\times \BR^{2n+1}_n\setminus\Delta$.  In particular $DD(S)$ is $(n-1)$-connected and $(n-1)$-locally connected.\end{corollary}

The subspace of marked hyperbolic structures on $S\times \BR$ which have the geometrically finite structure $Y$ on the $ -\infty$ relative end is known as the Bers slice $B_Y$. As in \cite{LS}, let $\partial_0B_Y(S)$ denote the subspace of the Bers slice whose hyperbolic structures on the $\infty$ relative end are degenerate.  

\begin{theorem}\cite{LS}  $\partial_0B_Y(S)$ is homeomorphic to $\EL$.\end{theorem}

\begin{corollary}  If $S$ is a hyperbolic surface of genus-$g$ and $p$-punctures, then $\partial_0B_Y(S)$ is $(n-1)$-connected and $(n-1)$-locally connected.  If $g=0$, then $\partial_0B_Y(S)$  is homeomorphic to the $n$-dimensional Nobeling space.  Here $n=3g+p-4$.\end{corollary}


\section{Problems and Conjectures}

To start with we restate the following long-standing classical question.

\begin{problem}  Find a topological characterization of the Nobeling type space $\BR^p_q$, the space of points in $\BR^p$ with at most $q$ rational coordinates.  \end{problem}

\begin{remarks}  i)  The fundamental work of Ageev \cite{Av}, Levin \cite{Le} and Nagorko \cite{Na} solve this problem for $p=2m+1$ and $q=m$.

ii)  Of particular interest are the spaces of the form  $\BR^{2n+1}_m$, where $m\ge n$.

iii) In analogy to Theorem \ref{nobeling characterization}, does it suffice that $X$ be separable, complete metric, $q$-dimensional, $(q-1)$-connected, $(q-1)$-locally connected, and satisfy the locally finite $(p-q-1)$-discs property?\vskip 8pt

\end{remarks}

We offer the following very speculative;

\begin{conjecture}  Let $S$ be a $p$-punctured surface of genus-$g$.  Then $\EL$ is homeomorphic to $\BR^{2n+1}_{n+k}$, where $n=3g+p-4$ and  $k=0$ if $g=0$, $k=g-1$ if both $p\neq 0$ and $g\neq 0$, and $k=g-2$ if $p=0$.\end{conjecture}  

The value of $k$ in this conjecture is motivated by the following duality conjecture.

\begin{conjecture} (Duality Conjecture) \label{duality}  If $S$ is  the $p$-punctured surface of genus-$g$, then  
\vskip 8pt

\noindent 1)  $\dim(EL)+ h\dim(\CS) +1 =\dim(\PML)$, where $h(\dim(\CS))$ is the homological dimension of the curve complex.
\vskip 8pt

If  $\dim(\PML)=2n+1$, then 

\vskip 8pt

\noindent 2)  for every $m$, there is a natural injection  $H_m^{st}(\EL) \to H^{2n-m}(\CS)$ and 
\vskip 8pt

\noindent 3)  for every $m$, there is a natural injection of $H_m(\CS) \to \check H^{2n-m}(\EL).$

 \end{conjecture}
 
 \begin{remark}  Harer \cite{Hr} computed the homological dimension of $\CS$, hence duality conjecture 1) is equivalent to the conjecture that $\dim(\EL)=n$ if $g=0$; $\dim(\EL)=n+(g-1)$ if $p\neq 0$ and $g\neq 0$; and $\dim(\EL)=n+(g-2)$ if $p=0$.\end{remark}

\begin{remarks}  In the above conjectures, homology is Steenrod and cohomology is Cech. Both coincide with the corresponding singular theory for CW-complexes.  When $m=0$ or $2n-m=0$, the above is reduced homology or cohomology.

Except for a single dimension all the maps and groups should be trivial.

The conjecture is motivated by the observation that $\EL$ and $\CS$ \emph{almost live} in $\PML=S^{2n+1}$ as complementary objects and hence Alexander - Sitnikov duality \cite{Sit} applies.  Indeed, since $\PML$ is the disjoint union of $\FPML$ and $\UPML$ the following holds by Alexander - Sitnikov duality.\end{remarks}

\begin{theorem}  \label{phi isomorphism} If $S$ is a finite type surface and $\dim(\PML)=2n+1$, then  
\vskip 8pt

\noindent 1)  $H_m^{st}(\FPML)$ is isomorphic to  $\check H^{2n-m}(\UPML)$ and 
\vskip8pt

\noindent 2) $H_m^{st}(\UPML)$ is isomorphic to  $\check H^{2n-m}(\FPML)$.\qed\end{theorem}

\begin{remarks} \label{coarser} Again we use reduced homology and cohomology for $m=0$ or $2n-m=0$.

Now $\UPML$ has a CW-structure $\UPML_{cw}$ and there is an inclusion $\CS\to \UPMLcw$ which induces a homotopy equivalence \cite{G2}.  However, the topology on $\UPMLcw$ is finer than that of $\UPML$ for essentially the same reason that the topology on $\CS$ is finer than that of $\CS_{sub}$.  See Remark \ref{sub is coarser}.  There is the analogy that $\CS$ is to $\CS_{sub}$ as the infinite wedge of circles is to the Hawaiian earings.  Note that the inclusion of the former into the latter induces proper injections of $H_1$ and $H^1$.  For that reason the maps in Conjecture \ref{duality} are  injections rather than isomorphisms.

The map $\phi:\FPML\to \EL$ is a closed map (Corollary \ref{closed map}) and point inverses are cells (Lemma \ref{preimage simplex}); hence by  \cite{Sp} or \cite{Sk} we obtain the following result.\end{remarks}

\begin{theorem}  \label{spanier} Let $S$ be a finite type surface.  Then for all $m$
\vskip 8pt

\noindent 1)  $\phi_*:H_m^{st}(\FPML)\to H_m^{st}(\EL)$ is an isomorphism and
\vskip 8pt

\noindent 2) $\phi^*:\check H^m(\EL)\to \check H^m(\FPML)$ is an isomorphism.\end{theorem}

\begin{corollary}  \label{spanier sitnikov} Let $S$ be a finite type surface and $\dim(\PML)=2n+1$,   then for all $m$ there are natural isomorphisms
\vskip 8pt

\noindent 1)  $ \check H^{m}(\EL) \cong H_{2n-m}^{st}(\UPML)$ and
\vskip 8pt

\noindent 2) $H_m^{st}(\EL) \cong \check H^{2n-m}(\UPML)$.\end{corollary}

\begin{remark} It follows from Corollary \ref{spanier sitnikov} and  Remarks \ref{coarser} that to prove parts 2) and 3) of the duality conjecture it suffices to show that the bijective map $I:\UPMLcw\to \UPML$ induces inclusions in Steenrod homology and Cech cohomology.  That will also be a big step towards proving 1), since it implies that  $\dim(\EL)$ is at least as large as the conjectured value.  \end{remark}

\begin{conjecture}  Let $S$ be a finite type surface.  The bijective maps $I:\CS\to \CSsub$ and $I:\UPMLcw\to \UPML$ induce inclusions in Steenrod homology and Cech cohomology.\end{conjecture}

\begin{question}  Let $S$ be a finite type surface.  Is $\EL$ $m$-connected (resp. $m$-locally connected) if and only if $\FPML$ is $m$-connected (resp. $m$-locally connected).\end{question}

\begin{remark}  It follows from Theorem \ref{fpml connectivity} that $\EL$ is connected if and only if $\FPML$ is connected.\end{remark}

\begin{definition}  Let $X$ be a regular cell complex.  We say that $X$ is \emph{super symmetric} if $X/ \Isom(X)$ is a finite complex.\end{definition}

\begin{example}  Examples of such spaces are the curve complex and arc complex of a finite type surface and the disc complex of a handlebody. \end{example}

\begin{question}  What spaces arise  as the boundaries of  super symmetric Gromov-hyperbolic regular cell complexes?  Under what conditions does the $m$-dimensional Nobeling space arise?   More generally, under what conditions does a Nobeling type space arise.\end{question}  

\begin{remark}  This is a well known question, when the cell complex is a Caley graph of a finitely generated group.  Indeed, Kapovich and Kleiner give an essentially complete answer when  $\partial G$ is 1-dimensional  \cite{KK}.  The Gromov boundaries of locally compact Gromov hyperbolic spaces are locally compact, thus the point of this question is to bring attention to non locally finite complexes.  Note that Gromov and Champetier \cite{Ch} assert that the \emph{generic} finitely presented Gromov-hyperbolic group has the Menger curve as its boundary.\end{remark}

\begin{question}  Let $x\in \EL$ be non uniquely ergodic.  Does $\phiinv(x)$ arise as in the construction of  Theorem 9.1 \cite{G2} or as limits of cells of the form $B_{a_1}\cap\cdots\cap B_{a_k}$ where $a_1, \cdots, a_k$ are pairwise disjoint simple closed curves?\end{question}

\enddocument